\documentclass[11pt]{amsart}
\usepackage{amssymb}
\usepackage{amsfonts}
\usepackage{amsmath}
\usepackage{graphicx}
\usepackage{xcolor}
\usepackage{amsmath}
\usepackage{mathrsfs}
\usepackage{stmaryrd}
\usepackage{epsfig,color}
\usepackage{blindtext}
\usepackage{enumerate}
\usepackage{hyperref}
\usepackage{url}
\usepackage{bbm}
\usepackage{nicefrac,mathtools}
\usepackage{bm}   
\usepackage{tabularx}
\DeclareGraphicsExtensions{.pdf,.jpeg,.png}
\usepackage{epstopdf}
\usepackage{cancel} 
\usepackage[normalem]{ulem} 
\usepackage{verbatim} 
\usepackage{scalerel}
\usepackage{enumitem}
\newcounter{dummy}
\makeatletter
\newcommand\myitem[1][]{\item[#1]\refstepcounter{dummy}\def\@currentlabel{#1}}
\makeatother

\usepackage{subcaption}
\usepackage{soul}

\usepackage{color}
\usepackage[msc-links, lite]{amsrefs}
\usepackage{geometry}
\usepackage{stackengine,scalerel}
\usepackage{tikz-cd}

\newtheorem{theorem}{Theorem}[section]

\newtheorem{conjecture}[theorem]{Conjecture}
\newtheorem{proposition}[theorem]{Proposition}
\newtheorem{lemma}[theorem]{Lemma}
\newtheorem{corollary}[theorem]{Corollary}

\newtheorem{claim}[]{Claim}
\newtheorem*{acknowledgements}{Acknowledgements}
\theoremstyle{definition}
\newtheorem{definition}[theorem]{Definition}
\theoremstyle{remark}
\newtheorem{remark}[theorem]{Remark}

\numberwithin{equation}{section}

\newcommand{\mf}{\mathbf}
\newcommand{\mb}{\mathbb}
\newcommand{\mc}{\mathcal}
\newcommand{\ms}{\mathscr}
\newcommand{\mk}{\mathfrak}

\newcommand{\wti}{\widetilde}

\newcommand{\Vol}{\mathrm{Vol}}
\newcommand{\Area}{\mathrm{Area}}
\newcommand{\Id}{\mathrm{Id}}

\newcommand{\dist}{\operatorname{dist}}
\newcommand{\inj}{\operatorname{inj}}
\newcommand{\codim}{\operatorname{codim}}
\newcommand{\bd}{\partial}

\newcommand{\rom}[1]{\expandafter\romannumeral #1}
\newcommand{\Rom}[1]{\uppercase\expandafter{\romannumeral #1}}

\DeclareMathOperator{\Index}{index}

\DeclareMathOperator{\Ric}{Ric}

\DeclareMathOperator{\Diff}{Diff}
\DeclareMathOperator{\spt}{spt}
\DeclareMathOperator{\interior}{int}
\DeclareMathOperator{\closure}{Clos}

\DeclareMathOperator{\Graph}{Graph}

\newcommand{\cCcG}{\mathcal{C}^{\mathcal{G}}}
\newcommand{\cBcG}{\mathcal{B}^{\mathcal{G}}}
\newcommand{\cScG}{\mathcal{S}_{\mathcal{G}}}

\newcommand{\Ah}{\mathcal{A}^h}
\newcommand{\cGh}{(\mathcal{G},h)}

\DeclareMathOperator{\VarTan}{VarTan}

\newcommand{\red}{\color{red}}

\title{Multiplicity one for equivariant min-max theory in prescribed homology classes}

\author{Tongrui Wang}
\address{School of Mathematical Sciences, Shanghai Jiao Tong University, 800 Dongchuan RD, Minhang District, Shanghai, 200240, China}
\email{wangtongrui@sjtu.edu.cn}

\begin{document}
\maketitle

\begin{abstract}
	For a closed Riemannian manifold $M$ with a compact Lie group $G$ acting by isometries, we show a generic multiplicity one theorem in equivariant min-max theory, and show in generic sense that there are infinitely many $G$-invariant minimal hypersurfaces in a fixed $G$-homology class. 
    We also establish an equivariant min-max theory for $G$-invariant hypersurfaces of prescribed mean curvature with $G$-index upper bounds. 
\end{abstract}

\section{Introduction}

\subsection{Background and motivations}
In Riemannian Geometry, an important theme is to study the submanifolds that are critical points of certain geometric functionals (such as length, area, or volume, with various constraints).
For instance in an $(n + 1)$-manifold $(M, g_{_M})$, the critical point of the $n$-volume functional are {\em minimal hypersurfaces}, the construction of which has been a central topic of research for centuries. 
In the 1960s, Almgren \cite{almgren1962homotopy}\cite{almgren1965theory} initiated a general {\em min-max theory} for the area functional. 
With the improvement made later by Pitts \cite{pitts2014existence} and Schoen-Simon \cite{schoen1981regularity}, this theory shows the existence of a smoothly embedded closed minimal hypersurface in $M^{n+1}$ with $3\leq n+1\leq 7$. 
(We refer to \cite{smith1982minmax}\cite{colding2002min-max}\cite{de2013existence} for other variants of min-max theory.)
Motivated by these results and the Morse theoretic considerations, Yau \cite{yau1982problem} formulated the following conjecture in his celebrated 1982 Problem Section.

\begin{conjecture}[S.-T. Yau \cite{yau1982problem}]\label{Conj: Yau}
	Every closed three-dimensional manifold contains infinitely many (immersed) minimal surfaces.
\end{conjecture}

Recent advances in min-max theory have enabled a series of works to resolve this conjecture from various perspectives, where the notion of {\em volume spectrum} plays a crucial role. 
By Almgren \cite{almgren1962homotopy} (see also \cite{marques2021morse}), the space of mod-$2$ boundary-type $n$-cycles $\mc B(M;\mb Z_2)$ is weakly homotopic to $\mb{RP}^\infty$. 
Then the volume spectrum $\{\omega_k(M^{n+1}, g_{_M} )\}_{k\in \mb Z_+}$ is a non-linear geometric quantity introduced by Gromov \cite{gromov1988width}\cite{gromov2003waists}, Guth \cite{guth2009minimax}, and Marques-Neves \cite{marques2017existence} as a non-decreasing sequence of min-max values for the area functional in $\mc B(M;\mb Z_2)$ with certain asymptotic behavior \cite{liokumovich2018weyl}. 

Using Almgren-Pitts min-max theory, Marques-Neves \cite{marques2016morse} showed in closed Riemannian manifold $(M^{n+1}, g_{_M} )$ with $3\leq n+1\leq 7$ that the $k$-th volume spectrum $\omega_k(M^{n+1}, g_{_M} )$ is realized by the area of a disjoint collection of closed, embedded, minimal hypersurfaces $\{\Sigma_i^k\}_{i=1}^{l_k}$ with integer multiplicities $\{m_i^k\}_{i=1}^{l_k}$ and index bounded by $k$, i.e.
\begin{align}
	\omega_k(M, g_{_M})= \sum_{i=1}^{l_k} m^k_i \Area(\Sigma^k_i) \quad{\rm and}\quad \sum_{i=1}^{l_k}\Index(\Sigma^k_i) \leq k .
\end{align}
Note that the potential multiplicity $\{m_i^k\}$ of the min-max minimal hypersurfaces presents one of the central challenges. 
Namely, from the variational point of view, a fixed minimal hypersurface with different integer multiplicities are different critical points. 
However, they cannot make more contributions to Conjecture \ref{Conj: Yau} since they are geometrically the same. 

Marques-Neves proposed the {\em Multiplicity One Conjecture} \cite{marques2021morse}: for {\em bumpy} metrics (that are generic metrics admitting no degenerate immersed closed minimal hypersurface \cite{white1991space}\cite{white2017bumpy}), $m_i^k=1$ and $\Sigma_i^k$ is $2$-sided for each $k\geq 1$ and $1\leq i\leq l_k$. 
This conjecture was confirmed by Zhou \cite{zhou2020multiplicity} using the min-max theory for {\em prescribed mean curvature} (PMC) hypersurfaces established by Zhou-Zhu \cite{zhou2019cmc}\cite{zhou2020pmc}; see also Chodosh-Mantoulidis  \cite{chodosh2020ACmultiplicity}. 
Zhou's multiplicity one result, together with the Morse index lower bound \cite{marques2021morse}, concludes under bumpy metrics that there is a closed embedded minimal hypersurface of Morse index $k$ and area $\omega_k(M, g_{_M})$ for each $k \geq 1$.  
This not only resolved Yau's conjecture in the generic sense, but also established the Morse theory for the area functional in $\mc B(M;\mb Z_2)$. 
We refer to \cite{marques2017existence}\cite{song2023infinite}\cite{irie2018density}\cite{marques2019equidistribution} for more on Conjecture \ref{Conj: Yau}. 

\medskip
Motivated by the recent progress on Yau's Conjecture, we would like to consider the following naturally improved version of Conjecture \ref{Conj: Yau}. 
\begin{conjecture}\label{Conj: Yau Conj in homology class}
	There are infinitely many closed embedded minimal hypersurfaces in a given homology class $[\Sigma_0]\in H_n(M^{n+1};\mb Z_2)$. 
\end{conjecture}
By Zhou's multiplicity one theorem \cite{zhou2020multiplicity}*{Theorem A}, the above conjecture is valid in the null homology class under bumpy metrics (and metrics with $\Ric_M>0$ \cite{zhou2020multiplicity}*{Remark 0.1}). 
However, Zhou's resolution relies on the min-max theory for PMC hypersurfaces, which cannot be apply to non-boundary-type hypersurfaces in $[\Sigma_0]\neq 0\in H_n(M^{n+1};\mb Z_2)$. 
A natural idea is to consider the double cover $\tau:\wti M\to M$ so that $[\tau^{-1}(\Sigma_0)]=0\in H_n(\wti M;\mb Z_2)$, and apply the boundary-type min-max theory in $\wti M$ under a $\mb Z_2$-equivariant constraint, where $\mb Z_2$ stands for the deck transformation group of $\tau$. 
Indeed, this method has been proven effective in \cite{wang2024RP2} where Li, Yao, and the author showed a multiple existence result for $1$-sided minimal surfaces with low genus in lens spaces. 
We also refer to \cite{wangzc2023FourSpheres} for the multiplicity one theorem in Simon-Smith min-max theory and its application to another Yau's conjecture on minimal $2$-spheres in $S^3$. 

Further motived by the $\mb Z_2$-equivariant constraint, we focus on closed Riemannian manifold $M$ with a compact Lie group $G$ acting by isometries in this paper. 
Combined the above idea with the $G$-symmetries, we consider the following generalized Yau's conjecture.
\begin{conjecture}\label{Conj: Generalized Yau Conjcture}
	$M$ contains infinitely many closed embedded minimal $G$-hypersurfaces $\Sigma$ in a given $G$-homology class, i.e. $\Sigma=\Sigma_0+\bd\Omega$ as mod $2$ cycles for a fixed $G$-hypersurface $\Sigma_0\subset M$ and some open $G$-set $\Omega\subset M$.  
\end{conjecture}
Here and below, the prefix `$G$-' represents the $G$-invariance for short. 
Note that the above conjecture generalized Conjecture \ref{Conj: Yau} in two aspects: the $G$-homology class constraint and the $G$-equivariant constraint. 
Using the previous idea, the former can be transformed into certain $\mb Z_2$-symmetric constraint. 
Meanwhile, for the latter, the author has established an {\em equivariant Almgren-Pitts min-max theory} in the space of $G$-boundaries $\mc B^G(M;\mb Z_2)$ through a series of works \cite{wang2022min}\cite{wang2023free}\cite{wang2023G-index}, where every $T\in \mc B^G(M;\mb Z_2)$ is induced by the boundary of a $G$-invariant Caccioppoli set. 
In particular, assume (for the regularity theory \cite{schoen1981regularity}) that 
\begin{align}\label{Eq: dimension assuption}
	3 \leq \codim(G\cdot x) \leq 7 \qquad \forall x\in M,
\end{align}
then the equivariant min-max in $\mc B^G(M;\mb Z_2)$ can produce smoothly embedded $G$-invariant minimal hypersurfaces with $G$-index upper bounds \cite{wang2023G-index}.
Additionally, assume \eqref{Eq: dimension assuption} and $\Ric_M>0$, there are infinitely many $G$-invariant minimal hypersurfaces \cite{wang2022min}\cite{wang2023free}, which confirms Conjecture \ref{Conj: Generalized Yau Conjcture} in $\mc B^G(M;\mb Z_2)$ by \cite{wang2024Ricci}*{Theorem 3.8} (see also \cite{zhou2015minmaxRicci}*{Proposition 6.1}).
We also refer to \cite{pitts1987applications}\cite{pitts1988equivariant}\cite{ketover2016equivariant}\cite{liu2021existence}\cite{franz2021equivariant}\cite{wang2025density} for more results on equivariant min-max in various settings.

\subsection{Main results}
In this paper, our first main result is to confirm Conjecture \ref{Conj: Generalized Yau Conjcture} for $G$-bumpy metrics and locally $G$-boundary-type $\Sigma_0$. 
Recall that a $G$-invariant metric is called {\em $G$-bumpy} if every finite cover of a smooth embedded minimal $G$-hypersurface is non-degenerate. 
By \cite{wang2023G-index}*{Theorem 1.3}, the set of $G$-bumpy metrics is generic in Baire sense. 
Also, a $G$-invariant hypersurface $\Sigma_0$ is said to be {\em locally $G$-boundary-type}, if for any $p\in M$, there is a small $G$-invariant neighborhood $U$ of $G\cdot p$ so that $\Sigma_0$ is the boundary of an open $G$-set $\Omega$ restricted in $U$. 
This is a natural generalized assumption since for trivial $G=\{id\}$, every embedded closed hypersurface is locally a boundary (due to the simply connectivity of small geodesic balls).

\begin{theorem}\label{Thm: main 1 generic infinte}
	Let $M$ be a closed manifold, $G$ be a compact Lie group acting by diffeomorphisms on $M$ with \eqref{Eq: dimension assuption} satisfied, and $\Sigma_0\subset M$ be a $G$-invariant hypersurface of locally $G$-boundary-type. 
	\begin{itemize}
		\item[(i)] For each $G$-bumpy metric $g_{_{M}}$, there are infinitely many closed, embedded, $G$-invariant minimal hypersurfaces $\Sigma$ that are $G$-homologous to $\Sigma_0$, i.e. $\Sigma=\Sigma_0+\bd\Omega$ as mod $2$ cycles for some $G$-invariant open set $\Omega\subset M$. 
		\item[(ii)] For a $G$-invariant metric $g_{_{M}}$ with $\Ric_M>0$, there is a sequence of connected, closed, embedded, $G$-invariant minimal hypersurfaces $\{\Sigma_k\}_{k\in\mb Z_+}$ so that 
		\[\mbox{each $\Sigma_k$ is $G$-homologous to $\Sigma_0$} \quad {\rm and}\quad \Area(\Sigma_k) \sim k^{\frac{1}{l+1}} ~ \mbox{as $k\to+\infty$}, \]
		where $l+1 =\min_{x\in M} \codim(G\cdot x)$ is the cohomogeneity of the $G$-actions. 
	\end{itemize}
\end{theorem} 

The proof of Theorem \ref{Thm: main 1 generic infinte} is based on a multiplicity one result in equivariant min-max as an equivariant version of the approach in \cite{zhou2020multiplicity}, but multiple challenges have emerged in order to meet our tighter constraints.
The key novelties lie in the following aspects (we will cover them more in the following subsections):  
\begin{itemize}
	\item an equivariant lifting construction to transform the $G$-homologous constraint into the $G_\pm$-signed symmetric constraint (Section \ref{Subsec: lift to signed symmetric});  
	\item an equivariant min-max theory for the existence and regularity of $G$-invariant hypersurfaces with mean curvature prescribed by generic symmetric functions (Section \ref{Sec: PMC min-max});
	\item a genericity result for good symmetric mean curvature prescribing functions that are applicable in the equivariant min-max theory (Section \ref{Subsec: prescription G-functions});
    \item a compactness theorem, a generic countability result, and $G$-index upper bounds for $G$-invariant PMC hypersurfaces produced by equivariant min-max 
    (Section \ref{Sec: compactness and index}); 
\end{itemize}
With these key ingredients, we are able to establish a multiplicity one theorem and apply it to the generalized equivariant volume spectrum for Theorem \ref{Thm: main 1 generic infinte}. 

More precisely, given a $G$-hypersurface $\Sigma_0\subset M$ of locally $G$-boundary-type, we can take the $G$-homology class $[\Sigma_0]^G$
of $\Sigma_0$ as: 
\[[\Sigma_0]^G := \{{\rm mod~} 2 {\rm ~cycles~}\llbracket \Sigma_0 \rrbracket + \bd\Omega: \Omega {\rm ~is~any~} G\mbox{-invariant Caccioppoli set}\},\]
and show $[\Sigma_0]^G$ is weakly homotopic to $\mb {RP}^\infty$. 
Using the min-max procedure, we can similarly define the equivariant volume spectrum $\{\omega_k([\Sigma_0]^G, g_{_M})\}_{k\in\mb Z_+}$ 
in $[\Sigma_0]^G$ with 
\begin{align}\label{Eq: aymptotic volume spectrum}
    \omega_k([\Sigma_0]^G, g_{_M})\sim k^{\frac{1}{l+1}}\quad {\rm as~}k\to+\infty
\end{align}
by \cite{wang2025density}, where $l+1=\min_{x\in M}\codim(G\cdot x)$. 
As a nonlinear version of spectrum for the area functional, 
each $\omega_k([\Sigma_0]^G, g_{_M})$ can be associated by equivariant min-max to a disjoint collection of embedded closed $G$-invariant minimal hypersurfaces $\{\Sigma^k_i\}_{i=1}^{l_k}$ with multiplicities $\{m_i^k\}$ satisfying 
\[ \omega_k([\Sigma_0]^G, g_{_M}) = \sum_{i=1}^{l_k} m^k_i \Area(\Sigma^k_i) \quad{\rm and}\quad \sum_{i=1}^{l_k}\Index_G(\Sigma^k_i) \leq k . \]
(See Proposition \ref{Lem: generalized volume spectrum}). 
The second main theorem of us is the following multiplicity one result generalizing \cite{zhou2020multiplicity}*{Theorem A}. 

\begin{theorem}\label{Thm: main 3 multiplicity one}
    Suppose \eqref{Eq: dimension assuption} is satisfied, and $g_{_M}$ is a $G$-bumpy metric. 
    Fix a $G$-invariant hypersurface $\Sigma_0\subset M$ of locally $G$-boundary-type. 
	Then for each $k\in\mb N$, there exists a disjoint collection of smooth, $G$-connected, embedded, minimal hypersurfaces $\{\Sigma^k_i: i=1,\dots, l_k\}$ so that  $\cup_{i=1}^{l_k}\Sigma^k_i$ is $G$-homologous to $\Sigma_0$,
	\[ \omega_k( [\Sigma_0]^G, g_{_M})= \sum_{i=1}^{l_k} \Area(\Sigma^k_i) \quad{\rm and}\quad \sum_{i=1}^{l_k}\Index_G(\Sigma^k_i) \leq k . \]
	That is to say, the $G$-invariant minimal hypersurfaces produced by equivariant min-max can be chosen in a given $G$-homology class with multiplicity one for $G$-bumpy metrics. 
\end{theorem}

\begin{remark}\label{Rem: for Ric>0}
    Indeed, for a general $G$-invariant metric, we can show that each $\Sigma_i^k$ either has multiplicity one or is degenerate stable. 
    Additionally, we also show $\{\Sigma_i^k\}$ are suitable to apply our compactness theorem (Theorem \ref{Thm: compactness for min-max Gh-hypersurfaces}, \cite{wang2023G-index}*{Theorem 1.5}). 
    Hence, the conclusions in Theorem \ref{Thm: main 3 multiplicity one} are also valid for a $G$-metric of positive Ricci curvature by a compactness analysis. 
\end{remark}


The conclusions in Theorem \ref{Thm: main 1 generic infinte}(i) follow from the above results and \eqref{Eq: aymptotic volume spectrum}. 
To obtain Theorem \ref{Thm: main 1 generic infinte}(ii), we also need to combine the embedded Frankel Theorem \cite{frankel1966fundamental} and Remark \ref{Rem: for Ric>0}. 

Potential application of our results include the Morse theory in the $G$-invariant hypersurfaces space, which could be the further extended to an equivariant Morse-Bott theory for the area functional. 
Additionally, combined our min-max theory in non-trivial homology class with \cite{wang2024RP2}, we expect further advances for $1$-sided minimal hypersurfaces. 
Moreover, using Song's technique for the resolution of Yau's Conjecture \cite{song2023infinite}, we will discuss in an upcoming paper more about its equivariant generalization Conjecture \ref{Conj: Generalized Yau Conjcture}.

\subsection{Lifting constructions and $G_\pm$-signed symmetries}
In previous (equivariant) min-max, the variational methods are mainly applied in the space of boundaries, i.e. $\mc B(M;\mb Z_2)$ and $\mc B^G(M;\mb Z_2)$. 
This is reasonable and convenient since the boundary-type hypersurfaces $\Sigma=\bd\Omega$ generally exist independent to the topology of $M$, and the functional $\mc A^h(\Omega)=\Area(\bd\Omega) - \int_\Omega h$ can be well-defined to construct PMC hypersurfaces \cite{zhou2019cmc}\cite{zhou2020pmc}. 
Additionally, in equivariant settings, the $G$-boundary-type condition also helps to compare the variations with/without equivariant constraints, e.g. the equivalence between stability and $G$-stability \cite{wang2022min}*{Lemma 7}, and the equivalence between minimizing and equivariant minimizing \cite{wang2023s1-stability}*{Proposition 3.4}. 

Integrated the requirement of $G$-homology constraints and the roles of $G$-boundary-type conditions, we set out this paper to consider the min-max theory in space of locally $G$-boundary-type hypersurfaces. 
Let $\mc {LB}^G(M;\mb Z_2)$ be the space of mod $2$ $G$-invariant $n$-cycles that are locally a boundary of a $G$-invariant Caccioppoli set. 
Then every connected component of $\mc {LB}^G(M;\mb Z_2)$ (in the flat topology) can be seen as a $G$-homology class. 
The first case is to consider the component of $\mc {LB}^G(M;\mb Z_2)$ containing $0$, i.e. $\mc B^G(M;\mb Z_2)$. 
If one assume $M$ is orientable, then $T\in \mc B^G(M;\mb Z_2)$ can be regarded as orientation-preserving under $G$-actions. 

The second case is to consider the component of $\mc {LB}^G(M;\mb Z_2)$ containing some $T\in \mc B(M;\mb Z_2)\setminus \mc B^G(M;\mb Z_2)$, i.e. $T=\bd\Omega$ but $\Omega$ is not $G$-invariant. 
In this case, we can separate $G$ into two components $G_+\sqcup G_-$ so that $G_+$ is a $2$-index subgroup of $G$ (Proposition \ref{Prop: lift to 2-cover}(i)) with 
\begin{align}\label{Eq: signed symmetric}
	G_+\cdot \Omega=\Omega \quad {\rm and} \quad G_-\cdot \Omega=M\setminus\Omega \quad \mbox{as Caccioppoli sets.}
\end{align} 
Roughly speaking, $G$-actions can be divided into the orientation-preserving subgroup $G_+$-actions, and the orientation-reversing coset $G_-$-actions. 
Moreover, the locally $G$-boundary-type condition also contributes the `freedom' between $G_+$ and $G_-$ actions in the following sense:
\begin{align}\label{Eq: free}
	 G\cdot p=G_+\cdot p \sqcup G_-\cdot p \quad{\rm and}\quad \mbox{$G_-$ permutes $\{G_+\cdot p,G_-\cdot p\}$}, \quad \forall p\in M. 
\end{align} 
Benefit from \eqref{Eq: free}, we can consider the variations and min-max under the {\em $G_\pm$-signed symmetric} constraints (Definition \ref{Def: Gpm-signed symmetry}) in this $G$-homology class, i.e. in the space $\mc B^{G_\pm}(M; \mb Z_2)$ of mod $2$ $G$-invariant $n$-cycles induced by the boundary of $\Omega$ satisfying \eqref{Eq: signed symmetric}. 
Indeed, \eqref{Eq: free} plays the same important role as the free assumption in \cite{wang2024RP2} (see Section \ref{Subsec: functional}, \ref{Subsec: prescription G-functions}).

The final case is to consider the component of $\mc {LB}^G(M;\mb Z_2)$ containing some $T\notin \mc B(M;\mb Z_2)$. 
For simplicity, assume $T$ is induced by a $G$-invariant hypersurface $\Sigma_0$. 
Then, we can find a double cover $\tau:\wti M\to M$ and lift $\Sigma_0$ to a null-homologous hypersurface $\wti\Sigma_0\subset\wti M$, i.e. $\wti\Sigma_0=\bd\wti\Omega$. 
Additionally, the $G$-actions on $M$ can also be lifted to a $2$-cover $\wti G$-actions on $\wti M$ with the following commutative diagram
\[\begin{tikzcd}[sep=small]
	{\widetilde{G}\times \widetilde{M}} && {\widetilde{M}} \\
	\\
	G\times M && M
	\arrow["\cdot", from=1-1, to=1-3]
	\arrow["{(p,\tau)}", from=1-1, to=3-1]
	\arrow["\tau", from=1-3, to=3-3]
	\arrow["\cdot", from=3-1, to=3-3]
\end{tikzcd}\]
where $p:\wti G\to G$ is the covering map. 
Indeed, $\wti G\cong G \times \mb Z_2$ is a trivial double cover. 
Moreover, similar to the second case, $\wti G$ can be separated into $\wti G_+ = G\times [0]$ and $\wti G_-=G\times [1]$ satisfying \eqref{Eq: signed symmetric} and \eqref{Eq: free} with `~$\wti{~}$~' added to all the objects (Proposition \ref{Prop: lift to 2-cover}(ii), \ref{Prop: free Gpm}). 
Hence, this case also reduces to the $\wti G_\pm$-signed symmetric constraints as before. 

In summary, we only need to consider the min-max constructions in $\mc B^G(M;\mb Z_2)$ and $\mc B^{G_\pm}(M;\mb Z_2)$. 
Hence, for simplicity, we use $\mc G$ to denote 
\[\mc G=G \qquad{\rm or}\qquad \mc G= G_\pm \mbox{ satisfying \eqref{Eq: free}}. \]

\subsection{Min-max theory for PMC hypersurfaces and related results}
Given $h\in C^\infty(M)$, a hypersurface $\Sigma=\bd\Omega$ has prescribed mean curvature (PMC) $h$ if its mean curvature coincides with $h|_\Sigma$, which is a critical point of the functional 
\[\mc A^h(\Omega):= \Area(\Sigma) - \int_\Omega h.\]
The global existence of PMC hypersurfaces was conjectured by Yau in the 1980s (c.f.\cite{yau1982problem}*{Problem 5.9}). 
In \cite{zhou2019cmc}\cite{zhou2020pmc}, Zhou and Zhu established a min-max theory to construct hypersurfaces that have constant mean curvature (CMC) $c$ for arbitrary $c\in \mb R$, and hypersurfaces with PMC $h$ for $h$ in an open dense set $\mc S\subset C^\infty(M)$, which plays a crucial role in Zhou's resolution of multiplicity one conjecture.  
Note that these PMC hypersurfaces are {\em almost embedded}, which means they can be locally decomposed into embedded sheets that touch but do not cross.
Additionally, by their choices of $h$, the touching set has codimension at least $1$. 

With equivariant constraints, the author and Wu also build a min-max theory for $G$-invariant CMC hypersurfaces in \cite{wang2022Gcmc}. 
In this paper, we further extend the min-max construction for $G$-invariant PMC hypersurfaces with $G$-index upper bounds. 
Note that in our $\mc G$-equivariant constraints, the PMC function $h$ can either be $G$-invariant (denoted by $h\in C^\infty_G(M)$), or be $G_+$-invariant and change signs under $G_-$-actions (denoted by $h\in C^\infty_{G_\pm}(M)$).

\begin{theorem}\label{Thm: main 2 equivariant PMC}
	Let $(M^{n+1},g_{_M})$ be a closed Riemannian manifold with a compact Lie group $G$ acting by isometries satisfying \eqref{Eq: dimension assuption}. 
	There exists an open dense set $\mc S_{\mc G}\subset C^\infty_{\mc G}(M)$ so that for each $h\in \mc S_{\mc G}$, 
	we have a closed, smooth, almost embedded, $G$-invariant hypersurface $\Sigma\subset M$ of prescribed mean curvature $h$ (with multiplicity one) satisfying 
	\begin{itemize}
		\item[(i)] $\Sigma$ has no minimal component;
		\item[(ii)] $\Sigma=\bd\Omega$ for an open $\mc G$-set $\Omega$, i.e. $\Omega$ is either $G$-invariant or $G_\pm$-signed symmetric \eqref{Eq: signed symmetric}. 
	\end{itemize}
\end{theorem}

The prescription functions $h\in\mc S_{\mc G}$ is constructed from equivariant Morse functions generalizing \cite{zhou2020pmc}*{Proposition 0.2}, which provides a size control on the non-embedding part of any $h$-PMC $G$-hypersurfaces and also excludes minimal components in $h$-hypersurfaces (Remark \ref{Rem: small touching and no minimal}). 
These properties for $h$-PMC $G$-hypersurfaces are crucial in the regularity theory of our equivariant min-max extended from \cite{wang2022Gcmc}\cite{zhou2020pmc}.  
Moreover, besides the necessary modifications, we also introduce the notions of {\em good $\cGh$-replacement property} (Definition \ref{Def: good replacement}) and {\em min-max $\cGh$-hypersurfaces} (Definition \ref{Def: min-max Gh-hypersurface}), which not only simplified the proof of the regularity result (Theorem \ref{Thm: main regularity}), but also helps to build a compactness theorem (Theorem \ref{Thm: compactness for min-max Gh-hypersurfaces}). 
Roughly speaking, a min-max $\cGh$-hypersurface is an almost embedded $h$-PMC $G$-hypersurface induced by the boundary of an open $\mc G$-set satisfying the same sufficient condition (i.e. good $\cGh$-replacement property) for regularity as constructed by equivariant min-max. 
One main function of these concepts is to overcome the difficulty that the equivariant index bounds is generally insufficient for compactness theory, which shares the same spirit as \cite{wang2023G-index}*{Section 4.3, Theorem 1.5} for min-max minimal $G$-hypersurfaces (see also \cite{wangzc2023FourSpheres}\cite{liyy2023infinite}). 

In \cite{white1991space}\cite{white2017bumpy}, White showed a {\em bumpy metrics theorem} for the generic non-degeneracy of closed minimal/CMC/PMC hypersurfaces. 
This result has beed extended by the author in \cite{wang2023G-index}*{Theorem 1.3} for embedded $G$-invariant minimal hypersurfaces (see also \cite{white2017bumpy}\cite{franz2021equivariant} for finite $G$), which is also valid for $h$-PMC $\mc G$-boundary-type hypersurfaces (Lemma \ref{Lem: generic good pairs}). 
Together with our compactness theorem, we obtain a {\em generic countability} result for min-max $\cGh$-hypersurfaces (Corollary \ref{Cor: countable min-max hypersurfaces}). 
Although we further expect the generic finiteness similar to the minimal case \cite{wang2023G-index}*{Corollary 1.7}, the generic countability already played an adequate and crucial role to show the index upper bounds in our equivariant min-max. 
Indeed, if $\Sigma$ in Theorem \ref{Thm: main 2 equivariant PMC} is produced by $\mc G$-equivariant min-max over an $k$-dimensional homotopy class, then $\Sigma$ has {\em weak $G$-index} (Definition \ref{Def: weak G-index}) bounded from above by $k$. 
Here, the weak $G$-index is defined similarly to \cite{zhou2020multiplicity}*{Definition 2.1}\cite{marques2016morse}*{Definition 4.1} to deal with hypersurfaces with self-touching.   
In particular, for properly embedded hypersurfaces, it coincides with the classical $G$-equivariant Morse index.

\subsection{Outline}
In Section \ref{Sec: preliminary}, we collect some notations and provide some technical results including the lifting constructions, the comparison between variations with/without equivariant constraints, and genericity of symmetric prescription functions. 
In Section \ref{Sec: PMC min-max}, we build the equivariant min-max theory for PMC hypersurfaces. 
Then we study the compactness theorem and $G$-index upper bounds for min-max $\cGh$-hypersurfaces in Section \ref{Sec: compactness and index}. 
Finally, in Section \ref{Sec: multiplicity one}, we show our main multiplicity one theorem and apply it to equivariant volume spectrum.

\begin{acknowledgements}
	Part of this work was done when T.W. visited Professor Xin Zhou at Cornell University, he would like to thank for their hospitality. 
    T.W. also thanks Xingzhe Li for helpful discussions. 
    T.W. is supported by the Natural Science Foundation of Shanghai 25ZR1402252 and the National Natural Science Foundation of China 12501076. 
\end{acknowledgements}

\section{Preliminary}\label{Sec: preliminary}

In this paper, we always assume that $(M,g_{_M})$ is a closed connected Riemannian $(n+1)$-manifold and $G$ is a compact Lie group acting isometrically on $M$ so that the codimension of every orbit $G\cdot p \subset M$ satisfies $3\leq \codim(G\cdot p)\leq 7$. 
Without loss of generality, we can also assume that $M$ is orientable (by lifting to its orientable double cover, see \cite{wang2023free}*{\S 2}) and the action of $G$ is effective (by using the quotient group $G/\{g\in G: g\cdot x=x~\forall x\in M\}$). 
Let $\mu$ be a bi-invariant Haar
measure on $G$ which has been normalized to $\mu(G) = 1$.

\subsection{Notations for manifolds and group actions}

We now collect some notations about the action of Lie groups and refer \cite{berndt2016submanifolds}\cite{bredon1972introduction}\cite{wall2016differential} for more details. 

By \cite{moore1980equivariant}, $M$ can be $G$-equivariantly isometrically embedded into some $\mb R^L$. 
Hence, we simply denote by $g\cdot p$ the acting of $g\in G$ on $p\in\mb R^L$. 
For a subset (submanifold, hypersurface,...) $A \subset M$ with $G \cdot A = A$, we usually call $A$ a $G$-set ($G$-submanifold, $G$-hypersurface,...) for simplicity. 
Given $p\in M$, we use the following notations. 
\begin{itemize}
    \item $G_p := \{g\in G: g\cdot p = p\}$ is the {\em isotropy group} of $p$ in $G$.
    \item $(G_p)=\{g\cdot G_p\cdot g^{-1}:g\in G\}$ is the conjugate class of $G_p$ in $G$; and $p$ is said to have the {\em orbit type} $(G_p)$. 
    \item $M_{(G_p)}:= \{q\in M: (G_q)=(G_p)\}$ is the $(G_p)$ orbit type stratum, which is a disjoint union of smooth embedded submanifolds of $M$ (see \cite{bredon1972introduction}*{Chapter 6, Corollary 2.5}). 
    \item $M^{prin}:=M_{(P)}$ is the {\em principal orbit type} stratum, where $(P)$ is the (unique) minimal conjugate class of isotropy groups. Note that $M^{prin}$ is an open dense submanifold of $M$. 
    \item ${\rm Cohom}(G)$ is the {\em cohomogeneity} of $G$ defined by the codimension of a principal orbit. 
    \item $\pi: M\to M/G$ is the natural quotient map. 
\end{itemize}

For any $G$-invariant subset $U\subset M$ with connected components $\{U_i\}_{i=1}^I$, we say $U$ is {\em $G$-connected} if for any $i,j\in\{1,\dots,I\}$, there is $g_{i,j}\in G$ with $g_{i,j}\cdot U_j=U_i$. 
We say $U$ is a {\em $G$-component} of $V$ if $U$ is a $G$-connected union of components of $V$. 

Moreover, given $p\in M\subset \mb R^L$ and $r,s,t >0$, we also use the following notations:
\begin{itemize}
    \item $B_r(p), \mb B^L_r(p)$: the geodesic $r$-ball in $M$ and the Euclidean $r$-ball in $\mb R^L$ respectively;
    \item $B_r(G\cdot p)$: the geodesic $r$-tube around $G\cdot p$ in $M$;
    \item $A_{s,t}(G\cdot p)$: the open annulus $B_t(G\cdot p)\setminus B_s(G\cdot p)$;
    \item $\closure(A)$: the closure of a set $A\subset \mb R^L$; 
    \item $T_q(G\cdot p)$: the tangent space of the orbit $G \cdot p$ at $q \in G \cdot p$;
    \item $N_q(G \cdot p)$: the normal vector space of the orbit $G \cdot p$ in $M$ at $q \in G \cdot p$;
    \item $N(G \cdot p)$: the normal vector bundle of the orbit $G \cdot p$ in $M$; 
    \item $\exp_p,\exp^\perp_{G\cdot p}$: the exponential map in $M$ at $p$ and the normal exponential map at $G\cdot p$. 
    \item $\inj(p),\inj(G\cdot p)$: the injectivity radius of $\exp_p$ and $\exp^\perp_{G\cdot p}$ respectively;
    \item $\mk X(M), \mk X(U)$: the space smooth vector fields supported in $M$ or $U$;
    \item $\mk X^G(M):=\{X\in \mk X(M): dg(X)=X\mbox{ for all $g\in G$} \}$, and $\mk X^G(U):=\mk X(U)\cap \mk X^G(M)$. 
\end{itemize}
Sometimes, we would also consider 
\begin{itemize}
    \item $G_+$: an index $2$ Lie subgroup of $G$;
    \item $G_-:=G\setminus G_+$; 
    \item $C^k_G(M):=\{f\in C^k(M): f(g\cdot p)=f(p)\mbox{ for all $g\in G$}\}$;
    \item $C_{G_{ \pm}}^{k}(M):=\left\{f \in C^{k}(M): \begin{array}{ll}f(p)=f(g \cdot p), & \forall p \in M, g \in G_{+} \\ f(p)=-f(g \cdot p), & \forall p \in M, g \in G_{-}\end{array}\right\}$,
\end{itemize}
where $k\in\mb N\cup\{\infty\}$. 
We say $f\in C^k_G(M)$ is $G$-invariant and $f\in C_{G_{ \pm}}^{k}(M)$ is {\em $G_\pm$-signed symmetric}.

\subsection{Notations for geometric measure theory}

We collect some notations in geometric measure theory (see \cite{simon1983lectures} and \cite{pitts2014existence}*{\S 2.1}). 
\begin{itemize}
    \item $\mc H^k$: the $k$-dimensional Hausdorff measure in $\mb R^L$; 
    \item $\mf I_k(M; \mb Z_2)$: the space of $k$-dimensional mod $2$ flat chains in $\mb R^L$ supported in $M$.
    \item $\mc Z_n(M; \mb Z_2)$: the space of $T\in \mf I_n(M;\mb Z_2)$ with $\bd T=0$. 
    \item $\mc C(M), \mc C(U)$: the space of all Caccioppoli sets $\Omega\subset M$ or $\Omega\subset U\subset M$. 
    \item $\llbracket M \rrbracket$: the mod $2$ flat chain associated with $M$. 
    \item $\bd\Omega$: the element in $\mc Z_n(M; \mb Z_2)$ reduced by the boundary of $\Omega\in \mc C(M)$.
    \item $\mc B(M;\mb Z_2):=\{T\in\mc Z_n(M;\mb Z_2): T=\bd \Omega \mbox{ for some $\Omega\in \mc C(M)$}\}$. 
    \item $\mc V_k(M)$: the closure, in the weak topology, of the space of $k$-dimensional rectifiable varifolds in $R^L$ supported in $M$. 
    \item $|T|, \|T\|$: the integer rectifiable varifold and the Radon measure induced by $T\in \mf I_k(M;\mb Z_2)$. 
    \item $\|V\|$: the Radon measure induced by $V\in \mc V_k(M)$. 
    \item $\mf F$: the $\mf F$-metric on $\mc V_k(M)$ that induces the varifolds weak topology on any mass bounded subset of $\mc V_k ( M )$ (\cite{pitts2014existence}*{p.66}).  
    \item $\mc F, \mf F, \mf M$: the flat (semi-)norm, the $\mf F$-metric, and the mass norm on $\mf I_k(M;\mb Z_2)$, where $\mf F(S,T):=\mc F(S-T) + \mf F(|S|,|T|)$ for all $S,T\in \mf I_k(M;\mb Z_2)$. 
\end{itemize}

Since $\mc C(M)$ can be identified with $\mf I_{n+1}(M;\mb Z_2)$ (see for instance \cite{marques2021morse}*{\S 5}), the flat $\mc F$-norm and the mass $\mf M$-norm are naturally defined on $\mc C(M)$. 
In addition, we define 
\[\mf F(\Omega_1,\Omega_2):=\mc F(\Omega_1-\Omega_2) + \mf F(|\bd\Omega_1|,|\bd\Omega_2|) \qquad\forall\Omega_1,\Omega_2\in\mc C(M),\]
and denote $\Bar{B}_r^{\mf F}(\Omega):=\{\Omega'\in \mc C(M): \mf F(\Omega,\Omega')\leq r\}$ for any $\Omega\in\mc C(M)$.

Given $T\in \mf I_k(M;\mb Z_2)$ (resp. $V\in\mc V_k(M)$), we say $T$ (resp. $V$) is {\em $G$-invariant} if $g_\# T=T$ (resp. $g_\#V=V$) for all $g\in G$. 
We can also define the following subspaces of $G$-invariant objects. 
\begin{itemize}
    \item $\mf I_k^G(M;\mb Z_2):=\{T\in \mf I_k(M;\mb Z_2): g_\#T=T,~\forall g\in G\}$;
    \item $\mc Z_n^G(M;\mb Z_2):=\{T\in \mc Z_n(M;\mb Z_2): g_\#T=T,~\forall g\in G\}$;
    \item $\mc C^G(M):=\{\Omega\in\mc C(M): g\cdot \Omega=\Omega, ~\forall g\in G\}$, and $\mc C^G(U):=\mc C^G(M)\cap\mc C(U)$;
    \item $\mc B^G(M;\mb Z_2):=\{T\in \mc B(M;\mb Z_2): T=\bd\Omega, {\rm ~for~some~}\Omega\in \mc C^G(M)\}$;
    \item $\mc {LB}^G(M;\mb Z_2)$ is the space of $T\in \mc Z_n^G(M;\mb Z_2)$ satisfying that for any $p\in M$, there exist $r\in (\inj(G\cdot p)/2,\inj(G\cdot p))$ and $\Omega\in \mc C^G(B_r(G\cdot p))$ so that $T=\bd\Omega$ in $B_r(G\cdot p)$; 
    \item $\mc V^G_n(M):=\{V\in\mc V_n(M): g_\#V=V,~\forall g\in G\}$.
\end{itemize}
The metrics $\mc F,\mf F, \mf M$ can be naturally induced to the above subspaces. 
Additionally, one notices that if $G=\{id\}$, then $B_r(G\cdot p)=B_r(p)$ is simply connected for $0<r<\inj(p)$, and thus $\mc {LB}^{\{id\}}(M;\mb Z_2) = \mc Z_n(M;\mb Z_2)$.

Given a varifold $V\in\mc V_k(M)$, we say $V$ is {\em stationary in an open subset $U\subset M$}, if 
\[0 = \delta V(X):=\left.\frac{d}{d t}\right|_{t=0}\left\|\left(\phi_{t}\right)_{\#} V\right\|(M)=\int \operatorname{div}_{S} X(x) d V(x, S)\quad \mbox{for all $X\in \mk X(U)$,}\]
where $\{\phi_t\}$ are the flow generated by $X$. 
In addition, given $c>0$, a varifold $V\in\mc V_k(M)$ is said to have {\em $c$-bounded first variation in an open subset $U\subset M$}, if 
\[|\delta V(X)|\leq c \int_M |X| d\mu_V \quad \mbox{for all $X\in\mk X(U)$}.\]
Using $G$-invariant objects, we also have the following definitions:
\begin{definition}\label{Def: Gc bounded 1st variation}
    Given an open $G$-set $U\subset M$ and $c>0$, a $G$-varifold $V\in \mc V^G_k(M)$ is said to be {\em $G$-stationary in $U$} if $\delta V(X)=0$ for all $X\in\mk X^G(U)$; and $V$ is said to have {\em $(G,c)$-bounded first variation in $U$} if $|\delta V(X)|\leq c \int_M |X| d\mu_V$ for all $X\in\mk X^G(U)$. 
\end{definition}

The above definitions with $G$-invariant constraints are equivalent to the usual definitions. 
Indeed, for any $X\in \mk X(M)$ supported in an open $G$-set $U\subset M$, the averaged vector field 
\begin{align}\label{Eq: averaged vector field}
   X_G:=\int_G d(g^{-1})(X) d\mu(g) 
\end{align}
is $G$-invariant with compact support in $U$ so that $\delta V(X)= \delta V(X_G)$ for any $V\in \mc V^G_k(M)$ (see \cite{liu2021existence}*{Lemma 3.2}). 
Hence, we have the following lemma (see also \cite{wang2022Gcmc}*{Lemma 2.9})
\begin{lemma}\label{Lem: Gc-bounded 1st variation and c-bounded}
    Given any $G$-varifold $V \in \mc V^G(M)$ and an open $G$-set $U \subset M$, we have that 
    \begin{itemize}
        \item[(i)] $V$ is stationary in $U$ if and only if $V$ is $G$-stationary in $U$; 
        \item[(ii)] $V$ has $c$-bounded first variation in $U$ if and only if $V$ has $(G,c)$-bounded first variation in $U$. 
    \end{itemize}
\end{lemma}

\subsection{Lift to be null-homologous with $G_\pm$-signed symmetries}\label{Subsec: lift to signed symmetric}
Note that $\mc B^G(M;\mb Z_2)\subsetneq \mc Z_n^G(M;\mb Z_2)$ in general, because it is also possible that $T\in \mc Z_n^G(M;\mb Z_2)$ represents a non-trivial $\mb Z_2$-homology class or $T=\bd \Omega$ for some $\Omega\in \mc C(M)\setminus \mc C^G(M)$. 
This leads to some difficulties in the construction of equivariant min-max theory for prescribed mean curvature hypersurfaces, since for $T\in \mc Z_n^G(M;\mb Z_2) \setminus \mc B^G(M;\mb Z_2)$ with a smooth support, the unit normal and the mean curvature function on $\spt(\|T\|)$ is either not well-defined or not $G$-invariant. 

Nevertheless, we can involve the {\em $G_\pm$-signed symmetries} and use the following proposition to make the equivariant min-max applicable for the whole $\mc Z_n^G(M;\mb Z_2)$ space in a variant way. 

\begin{definition}\label{Def: Gpm-signed symmetry}
    Given a connected closed manifold $M$ with a compact Lie group $G$ acting effectively by isometries, let $G_+$ be an $2$-index Lie subgroup of $G$ and $G_-:=G\setminus G_+$. 
    Then we say $\Omega\in\mc C(M)$ is {\em $G_\pm$-signed symmetric} if $G_+\cdot \Omega=\Omega$ and $G_-\cdot \Omega=M\setminus\Omega$ as Caccioppoli sets. 
    Denote by 
    \begin{itemize}
        \item $\mc C^{G_\pm}(M)$: the collection of all $G_\pm$-signed symmetric Caccioppoli sets in $M$,
        \item $\mc B^{G_\pm}(M;\mb Z_2):=\{T\in\mc Z^G_n(M;\mb Z_2): T=\bd\Omega \mbox{ for some }\Omega\in\mc C^{G_\pm}(M)\}$.
    \end{itemize}
\end{definition}

\begin{proposition}\label{Prop: lift to 2-cover}
    Given $T\in\mc Z_n^G(M;\mb Z_2)\setminus \mc B^G(M;\mb Z_2)$, 
    \begin{itemize}
        \item[(1)] if $T\in \mc B(M;\mb Z_2)$, then there exists $\Omega\in\mc C(M)\setminus\mc C^G(M)$ and an $2$-index Lie subgroup $G_+$ of $G$ with $G_-:=G\setminus G_+$ so that $T=\bd\Omega$, $G_\pm$ is the union of some connected components of $G$, and $\Omega\in\mc C^{G_\pm}(M)$;
        \item[(2)] if $T\notin \mc B(M;\mb Z_2)$, then there exists a double cover $\tau:\wti M\to M$, a trivial Lie group $2$-cover $p:\wti G\cong G\times\mb Z_2 \to G$ with $p^{-1}(e)\cong\mb Z_2$, and $\wti\Omega\in\mc C(\wti M)$ so that 
        \begin{itemize}
            \item[(a)] $\wti G$ acts by isometries on $\wti M$ with $\tau\circ \tilde{g} = p(\tilde{g})\circ\tau$ for all $\tilde{g}\in\wti G$,
            \item[(b)] $p^{-1}(e)=\{\tilde{e},\tilde{g}_-^0\}$ is the deck transformation group of $\tau$, where $\tilde{e}\in \wti G$ is the identity, 
            \item[(c)] $\wti\Omega\in \mc C^{\wti G_\pm}(\wti M)$, where $\wti G_+\lhd\wti G$ is a $2$-index Lie subgroup of $\wti G$ and $\wti G_-=\tilde{g}_-^0\cdot \wti G_+$ so that $p|_{\wti G_+}:\wti G_+\to G$ is a Lie group isomorphism. 
            \item[(d)] $\wti T:=\bd\wti \Omega=\bd(\wti M\setminus\wti\Omega)\in\mc B^{\wti G_\pm}(\wti M;\mb Z_2)$ satisfies $\tau_\#(|\wti T|) = 2|T|$.
        \end{itemize}
    \end{itemize}
\end{proposition}
\begin{proof}

    Suppose firstly that $T\in (\mc Z_n^G(M;\mb Z_2)\setminus\mc B^G(M;\mb Z_2)) \cap  \mc B(M;\mb Z_2)$. 
    By the assumption, $T=\bd\Omega$ for some $\Omega\in\mc C(M)\setminus \mc C^G(M)$. 
    For any $g\in G$, since $\bd(g\cdot\Omega) = g_\#T=T=\bd\Omega$, we have $g\cdot \Omega $ coincides with $ \Omega$ or $M\setminus\Omega$ as Caccioppoli sets by the Constancy Theorem \cite{simon1983lectures}*{26.27}. 
    Define 
    \[G_+:=\{g\in G: g\cdot \Omega=\Omega\in\mc C(M)\} \quad{\rm and}\quad  G_-:=\{g\in G: g\cdot \Omega=M\setminus\Omega\in\mc C(M)\}.\]
    One can then easily verify that $G_-\neq \emptyset$ and $G_+$ is a Lie subgroup of $G$. 
    In addition, for any $g_1,g_2\in G_-$, we have $g_1\cdot g_2 \in G_+$, which implies $g_1\cdot G_+=g_2\cdot G_+$ and $[G:G_+]=2$. 
    Using a continuity argument, we also know that every connected component of $G$ must be contained entirely in $G_+$ or $G_-$. 

    Next, we consider the case that $T\in \mc Z_n^G(M;\mb Z_2) \setminus\mc B(M;\mb Z_2)$. 
    For simplicity, we take a $G$-hypersurface $\Sigma\subset M$ so that $T-\llbracket\Sigma\rrbracket = \bd\Omega_0$ for some $\Omega_0\in\mc C^G(M)$, where $\llbracket\Sigma\rrbracket\in \mc Z_n^G(M;\mb Z_2)$ is induced by $\Sigma$. 
    (This is possible by taking the area minimizer in $\{T+\bd\Omega:\Omega\in\mc C^G(M)\}$ with optimal regularity given by the first part of Theorem \ref{Thm: regularity of G,Ah-minimizers} with $h=0$.)
    
    Note that $\llbracket\Sigma\rrbracket\notin \mc B(M;\mb Z_2)$. 
    Hence, $\Sigma$ and $T$ induce the same non-trivial $\mb Z_2$-homology class $0\neq [\Sigma]=[T]\in H_n(M;\mb Z_2)$. 
    Let $[\Sigma]^*\in H^1(M;\mb Z_2)$ be the dual of $[\Sigma]$, which naturally induces a non-trivial homomorphism $[\Sigma]^*: \pi_1(M)\to \mb Z_2$. 
    Consider the double cover $\tau:\wti M\to M$ corresponding to the $2$-index (normal) subgroup $\ker([\Sigma]^*)\lhd \pi_1(M)$. 
    Then we have 
    \begin{itemize}
        \item the deck transformation group of $\tau$ is $\mb Z_2=\pi_1(M)/\ker([\Sigma]^*)$;
        \item $\tau^{-1}(\Sigma)=\bd\wti\Omega _\Sigma$ for some $\wti\Omega _\Sigma\in \mc C^{\mb Z_2^\pm}(\wti M)$, where $\mb Z_2^+:=\{[0]=id\}$ and $\mb Z_2^-:=\{[1]=\mbox{the involution map}\}$. 
    \end{itemize}
    To lift the Lie group $G$ actions into $\wti M$, we firstly notice that if $[\gamma]\in\ker([\Sigma]^*)$, then we may assume the closed curve $\gamma$ meets $\Sigma$ transversally at an even number of points. 
    Thus, for any $g\in G$, both $g\cdot \gamma$ and $g^{-1}\cdot\gamma$ have an even number of intersection points with the $G$-hypersurface $\Sigma$, which implies $g\cdot \ker([\Sigma]^*)=\ker([\Sigma]^*)$. 
    Now, we can apply \cite{bredon1972introduction}*{Chapter 1, Theorem 9.3} to obtain a covering Lie group $p:\wti G\to G$ so that 
    \begin{itemize}
        \item $\wti G$ acts isometrically effectively on $\wti M$ with $\tau\circ \tilde{g} = p(\tilde{g})\circ\tau$ for all $\tilde{g}\in\wti G$;
        \item $\ker(p)=p^{-1}(e)$ is the deck transformation group $\mb Z_2$ of $\tau$. 
    \end{itemize}
    In particular, we know $p:\wti G\to G$ is a double cover. 
    Additionally, by the above constructions, we see $\llbracket\tau^{-1}(\Sigma)\rrbracket\in \mc Z_n^{\wti G}(\wti M;\mb Z_2) \cap ( \mc B(\wti M;\mb Z_2)\setminus\mc B^{\wti G}(\wti M;\mb Z_2))$, and thus the conclusions in (1) indicate that $\wti G= \wti G_+ \sqcup \wti G_-$, where $\wti G_+:=\{\tilde{g}\in\wti G: \tilde{g}\cdot\wti\Omega _\Sigma=\wti\Omega _\Sigma\in\mc C(\wti M)\}$ and $\wti G_-:=\{\tilde{g}\in\wti G: \tilde{g}\cdot\wti\Omega _\Sigma=\wti M\setminus\wti\Omega _\Sigma\in\mc C(\wti M)\}$. 
    Moreover, note that $p^{-1}(e)=\{\tilde e, \tilde g^0_-\}$ with $\tilde e\in \wti G_+$ and $\tilde g^0_-\in\wti G_-$. 
    Hence, $\wti G=\{\tilde{g}_+\cdot \wti{g}^0: \tilde{g}_+\in \wti G_+, \wti{g}^0\in p^{-1}(e)\}$ and $\wti G_+\cap p^{-1}(e) = \{\tilde{e}\}$, which implies $\wti G= \wti G_+ \rtimes \mb Z_2$ is a semidirect product. 

    Next, we claim that $G\cong G_+$ and $\wti G=\wti G_+\times \mb Z_2$ is a direct product. 
    Indeed, given any $g\in G$, let $p^{-1}(g)=\{\tilde{g}_1,\tilde{g}_2\}\subset\wti G$. 
    Suppose $\tilde{g}_1$ and $\tilde{g}_2$ are both in $\wti G_+$ or both in $\wti G_-$. 
    Then $\tilde{g}_1\cdot\tilde{g}_2^{-1}\in\wti G_+$ and $p(\tilde{g}_1\cdot\tilde{g}_2^{-1})=p(\tilde{g}_1)\cdot p(\tilde{g}_2^{-1})=e\in G$. 
    However, since $p^{-1}(e):= \{\tilde{e},\tilde{g}_-^{0}\}$ is the deck transformation group of $\tau$ and $\tilde{g}_-^{0}\in \wti G_-$, we see $\tilde{g}_1\cdot\tilde{g}_2^{-1}$ is the identity $\tilde{e}$ of $\wti G$, which contradicts $\tilde{g}_1\neq\tilde{g}_2$. 
    Therefore, for any $g\in G$, $p^{-1}(g)$ has one element in $\wti G_+$ and the other one in $\wti G_-$. 
    In particular, this implies that $\tilde{g}_-^0 \cdot \tilde{g}_+\cdot (\tilde{g}_-^0)^{-1} = \tilde{g}_+$ for all $\tilde{g}_+\in \wti G_+$, and thus $\wti G=\wti G_+\times \mb Z_2$ is a direct product. 
    Additionally, since $\wti G_+$ and $\wti G_-$ are disjoint unions of components of $\wti G$, we see $p|_{\wti G_+}$ is an isomorphism between $\wti G_+$ and $G$. Hence, $p:\wti G\to G$ is a trivial group $2$-cover. 

    Finally, take $\wti\Omega := \wti\Omega_\Sigma - \tau^{-1}(\Omega_0)\in \mc C(\wti M)$. 
    Since $\tau^{-1}(\Omega_0)$ is $\wti G$-invariant and $\wti\Omega_\Sigma$ is $\wti G_\pm$-signed symmetric, we know $\wti\Omega\in \mc C^{\wti G_\pm}(\wti M)$. 
    Moreover, because $\tau$ is a double cover and a local isometry, we see $\tau(\spt(\bd\wti\Omega))=\spt(T)$ and $\tau_\#(|\bd\wti\Omega|) = \underline{\underline{v}}(\tau(\spt(\bd\wti\Omega)), 2)=2|T|$, where $\underline{\underline{v}}$ is the notation for rectifiable varifolds (see \cite{simon1983lectures}*{Chapter 4}). 
    This proved (3). 
\end{proof}

In Proposition \ref{Prop: lift to 2-cover}, if we further assume $T\in \mc {LB}^G(M;\mb Z_2)$, then the actions of $G_+$ (resp. $\wti G_+$) are {\em free} from $G_-$ (resp. $\wti G_-$) in the following sense:
\begin{proposition}\label{Prop: free Gpm}
    Given $T\in \mc {LB}^G(M;\mb Z_2)\setminus\mc B^G(M;\mb Z_2)$, if $T\in\mc B(M;\mb Z_2)$, then for all $p\in M$,
    \begin{align}\label{Eq: free Gpm}
        G\cdot p=G_+\cdot p \sqcup G_-\cdot p \qquad{\rm and}\qquad \mbox{$G_-$ permutes $\{G_+\cdot p,G_-\cdot p\}$},
    \end{align}
    where $G_\pm$ are given in Proposition \ref{Prop: lift to 2-cover}(1). 
    Similarly, if $T\notin\mc B(M;\mb Z_2)$, then the same result holds in $\wti M$ with $\wti G_\pm$ in place of $G_\pm$, where $\wti M$ and $\wti G_\pm$ are given in Proposition \ref{Prop: lift to 2-cover}(2). 
\end{proposition}
\begin{proof}
    Suppose by contradiction that there exist $p_0\in M$, $g_+\in G_+$, and $g_-\in G_-$ so that $g_+\cdot p_0= g_-\cdot p_0$. 
    Then, $g_+^{-1} g_-\in G_{p_0}\cap G_-$. 
    
    By Proposition \ref{Prop: lift to 2-cover}(1), we also have $\Omega\in\mc C^{G_\pm}(M)$ with $T=\bd\Omega$. 
    Since $T\in \mc {LB}^G(M;\mb Z_2)$, we have an $r\in (\inj(G\cdot p)/2,\inj(G\cdot p))$ associated with $p_0$ so that $T=\bd\Omega_0$ in $B_r(G\cdot p_0)$ for some $\Omega_0\in\mc C^G(B_r(G\cdot p_0))$. 
    Take the connected component $U$ of $B_r(G\cdot p_0)$ containing $p_0$. 
    Using the Constancy Theorem, we can assume without loss of generality that $\Omega\cap U$ coincides with $\Omega_0$. 
    However, since $g_+^{-1} g_-\in G_{p_0}\cap G_-$, we have $(g_+^{-1} g_-)\cdot U=U$ and $(g_+^{-1} g_-)\cdot\Omega=M\setminus\Omega$. Hence, $\Omega_0=(g_+^{-1} g_-)\cdot\Omega_0=(g_+^{-1} g_-)\cdot(\Omega\cap U) = ((g_+^{-1} g_-)\cdot\Omega)\cap U=(M\setminus\Omega)\cap U = U\setminus\Omega_0$, which is a contradiction. 

    If $T\in \mc {LB}^G(M;\mb Z_2)\setminus\mc B(M;\mb Z_2)$, then $\wti T$ in Proposition \ref{Prop: lift to 2-cover}(2) belongs to $\mc {LB}^{\wti G}(\wti M;\mb Z_2)\cap\mc B^{\wti G_\pm}(\wti M;\mb Z_2)$, and thus the above argument would carry over. 
\end{proof}

We mention that \eqref{Eq: free Gpm} may not hold in general. 
For example, let $\tau:\mb S^3\to \mb {RP}^3$ be the standard double cover, and $p:\wti G=\mb Z_2\times\mb Z_2 \to G=\mb Z_2$ is given by $p(([0],[0]))=p(([1],[1]))=[0]$, $p(([0],[1]))=p(([1],[0]))=[1]$. 
Let the $\wti G$-action on $\mb S^3\subset\mb R^4$ satisfies that $([0],[0])=id$, $([1],[1])$ is the antipodal map, $([1],[0])\cdot(x_1,x_2,x_3,x_4)=(-x_1,-x_2,x_3,x_4)$, $([0],[1])\cdot(x_1,x_2,x_3,x_4)=(x_1,x_2,-x_3,-x_4)$, which naturally induces the $G$-action on $\mb {RP}^3$ by the commutativity. 
Then for $\wti G_+:=\{([0],[0]),([1],[0])\}$ and $\wti G_-=\{([0],[1]),([1],[1])\}$, we have $\wti G_+\cdot (\sin\theta,\cos\theta,0,0)=\{(\sin\theta,\cos\theta,0,0), (-\sin\theta,-\cos\theta,0,0)\}=\wti G_-\cdot (\sin\theta,\cos\theta,0,0)$. 

Combining the above results, we make the following remark. 
\begin{remark}\label{Rem: lift to 2-cover}
    In our equivariant min-max constructions, we would consider $\mc F$-continuous maps $\Phi:X\to\mc {LB}^G(M;\mb Z_2)$ for some cubical complex $X$. 
    Note that for any $T\in \mc Z_n^G(M;\mb Z_2)$, the set $T+\mc B^G(M;\mb Z_2):=\{T+\bd\Omega: \Omega\in\mc C^G(M)\}$ is the connected component of $\mc Z_n^G(M;\mb Z_2)$ containing $T$ (by the isoperimetric lemma \cite{wang2023G-index}*{Lemma 2.1}). 
    Additionally, in the above propositions, $G_\pm$, $\tau:\wti M\to M$ and $p:\wti G\to G$ can be chosen uniformly for all $S\in T+\mc B^G(M;\mb Z_2)$. 
    Therefore, if $X$ is connected and $\Phi(x_0)\in\mc {LB}^G(M;\mb Z_2)\setminus\mc B(M;\mb Z_2)$ for some $x_0\in X$, we can take $\tau:\wti M\to M$ and $p:\wti G\to G$ in Proposition \ref{Prop: lift to 2-cover}(2) for $\Phi(x_0)$ so that the corresponding statements also hold for all $\Phi(x)$, $x\in X$, and thus $\Phi:X\to \mc {LB}^G(M;\mb Z_2)\setminus\mc B(M;\mb Z_2)$ can be lifted to 
    \[\mbox{an $\mc F$-continuous map $\wti\Phi: X \to \mc {LB}^{\wti G}(\wti M;\mb Z_2)\cap\mc B^{\wti G_\pm}(\wti M;\mb Z_2)$ with $\tau_\#|\wti\Phi(x)| = 2|\Phi(x)|$.}\]
    In particular, using this lifting construction, it's sufficient to consider $\Phi:X\to \mc {LB}^G(M;\mb Z_2)\cap\mc B(M;\mb Z_2)$ in our equivariant min-max constructions, which has two sub-cases:
    \begin{itemize}
        \item {\bf Case I.} $\Phi: X\to \mc B^G(M;\mb Z_2)\subset \mc {LB}^G(M;\mb Z_2)$;
        \item {\bf Case II.} $\Phi: X\to \mc B^{G_\pm}(M;\mb Z_2)\cap \mc {LB}^G(M;\mb Z_2)$ with $G\cdot p=G_+\cdot p\sqcup G_-\cdot p$ for all $p\in M$. 
    \end{itemize}
    Moreover, in {\bf Case II}, for any $h\in C^\infty_{G_\pm}$ and $\Omega\in\mc C^{G_\pm}(M)$, we have $h\cdot \nu_{\bd\Omega}$ is $G$-invariant, where $\nu_{\bd\Omega}$ is the outward unit normal on $\bd\Omega$. 
\end{remark}

\subsection{Weighted area functional $\mc A^h$}\label{Subsec: functional}
Given a smooth function $h:M\to \mb R$, we consider the following weighted area functional $\mc A^h$ on $\mc C(M)$ defined by 
\begin{align}\label{Eq: Ah functional}
    \mc A^h(\Omega) := \mc H^n(\bd\Omega) - \int_\Omega h d\mc H^{n+1}. 
\end{align}
The {\em first variation formula} of the $\mc A^h$-functional along $X\in \mk X(M)$ is 
\begin{align}\label{Eq: 1st variation of Ah}
    \delta \mathcal{A}^{h}|_{\Omega}(X)=\int_{\partial \Omega} \operatorname{div}_{\partial \Omega} X d \mu_{\partial \Omega}-\int_{\partial \Omega} h\langle X, \nu\rangle d \mu_{\partial \Omega},
\end{align}
where $\nu=\nu_{\bd\Omega}$ is the outward unit normal on $\bd\Omega$. 
The following result is similar to Lemma \ref{Lem: Gc-bounded 1st variation and c-bounded}. 
\begin{lemma}\label{Lem: 1st variation for G-boundary}
    Suppose $h\in C^\infty(M)$ and $\Omega\in \mc C(M)$ are both $G$-invariant or both $G_\pm$-signed symmetric. 
    Then $\delta \mathcal{A}^{h}|_{\Omega}(X)=\delta \mathcal{A}^{h}|_{\Omega}(X_G)$ for any $X\in \mk X(M)$ and $X_G\in \mk X^G(M)$ given by \eqref{Eq: averaged vector field}. 
\end{lemma}
\begin{proof}
    Since $\bd\Omega$ and $h \nu$ are always $G$-invariant, we see $\int_{\partial \Omega} \langle X, h\nu\rangle d \mu_{\partial \Omega}=\int_{\partial \Omega} \langle X_G, h\nu\rangle d \mu_{\partial \Omega}$. 
    Combining with Lemma \ref{Lem: Gc-bounded 1st variation and c-bounded} and \eqref{Eq: 1st variation of Ah}, we get $\delta \mathcal{A}^{h}|_{\Omega}(X)=\delta \mathcal{A}^{h}|_{\Omega}(X_G)$. 
\end{proof}

Let $\phi:\Sigma^n\to M$ be an immersion. We often abuse notation and write $\Sigma\subset M$ by identifying $\Sigma$ with its image $\phi(\Sigma)$ (unless otherwise stated). 
If $\bd\Omega=\Sigma$ is a smooth immersed hypersurface, then the first variation formula can be rewritten by $\delta \mathcal{A}^{h}|_{\Omega}(X)=\int_{\partial \Omega}(H-h)\langle X, \nu\rangle d \mu_{\partial \Omega}$, 
where $H$ is the mean curvature of $\Sigma$ with respect to $\nu$. 
Hence, if $\Omega$ is a critical point of $\mc A^h$, then $\Sigma$ must have mean curvature $H=h|_\Sigma$. 
In this case, we say $\Sigma$ is an {\em $h$-hypersurface}, and can compute the {\em second variation formula for $\mc A^h$} along normal vector field $X=\varphi\nu\in\mk X(M)$:
\begin{align}\label{Eq: 2nd variation of Ah for hypersurface}
    \delta^{2} \mathcal{A}^{h}|_{\Omega}(X, X)  =II^h_{\Sigma}(\varphi, \varphi) =\int_{\Sigma}\left(|\nabla \varphi|^{2}-\left(\operatorname{Ric}_{M}(\nu, \nu)+\left|A_{\Sigma}\right|^{2}+\partial_{\nu} h\right) \varphi^{2}\right) d \mu_{\Sigma},
\end{align}
where $\nabla\varphi$ is the gradient of $\varphi\in C^\infty(\Sigma)$ on $\Sigma$, $\Ric_M$ is the Ricci curvature of $M$, $A_\Sigma$ is the second fundamental form of $\Sigma$. 
For any open set $U\subset M$, we say an $h$-hypersurface $\Sigma=\bd\Omega$ is {\em stable in $U$} if 
\[\delta^{2} \mathcal{A}^{h}|_{\Omega}(X, X) \geq 0\]
for all ambient vector field $X\in\mk X(U)$ with $X=\varphi\nu$ on $\Sigma$. 
Note that $II_\Sigma$ can be intrinsically defined as a quadratic form on the space of $C^\infty$-functions on the preimage $\Sigma$ (not $\phi(\Sigma)$). 
Hence, we define the {\em Jacobi field} of the immersed $h$-hypersurface $\phi:\Sigma\to M$ to be a smooth function $\varphi$ on $\Sigma$ (not $\phi(\Sigma)$) so that
\begin{align}\label{Eq: Jacobi field}
    0=L_\Sigma\varphi := (\Delta + |A_\Sigma|^2 + |\Ric_M(\nu,\nu)| + \bd_\nu h) \varphi \qquad \mbox{on $\Sigma$}. 
\end{align}
The {\em Morse index $\Index(\Sigma)$} of $\Sigma$ is defined as the number of negative eigenvalues (counting multiplicities) of the (intrinsic) quadratic form $II_\Sigma$. 

Moreover, given an open $G$-set $U$, $h\in C^\infty_G(M)$, and $\hat{h}\in C^\infty_{G_\pm}(M)$, we define that
\begin{itemize}
    \item an {\em almost embedded} $G$-hypersurface is an immersed $G$-invariant hypersurfaces that can locally decompose into smooth embedded sheets that touch but do not cross. 
    \item a {\em $(G,h)$-hypersurface} $\Sigma$ is an almost embedded $G$-hypersurface with $H_\Sigma=h|_\Sigma$; 
    \item a {\em $(G,h)$-boundary $\Sigma$ in $U$} is a $(G,h)$-hypersurface $\Sigma\subset U$ such that $\llbracket \Sigma \rrbracket=\partial \Omega\llcorner U$ for some $\Omega\in \mc C^G(M)$ with $\delta\Ah |_\Omega (X) = 0$, $\forall X\in \mk X(U)$; 
    \item a {\em $(G_\pm,\hat{h})$-boundary $\Sigma$ in $U$} is a $(G,\hat{h})$-hypersurface $\Sigma\subset U$ such that $\llbracket \Sigma \rrbracket=\partial \Omega\llcorner U$ for some $\Omega\in \mc C^{G_\pm}(M)$ with $\delta\Ah |_\Omega (X) = 0$, $\forall X\in \mk X(U)$; 
    \item $\mc P^{G,h}(U)$ and $\mc P^{G_\pm,\hat{h}}(U)$ are the sets of $(G,h)$-boundaries and $(G_\pm,\hat{h})$-boundaries in $U$ respectively; for simplicity, we denote by $\mc P^{G,h}:=\mc P^{G,h}(M)$ and $\mc P^{G_\pm,\hat{h}}:=\mc P^{G_\pm,\hat{h}}(M)$.  
\end{itemize}
We can now define the stability and index under certain symmetric constraints. 
\begin{definition}\label{Def: G-stable}
    For any $(h,\bd\Omega)\in (C^\infty_G(M)\times \mc P^{G,h}) \cup (C^\infty_{G_\pm}(M)\times\mc P^{G_\pm,\hat{h}})$ and open $G$-set $U\subset M$, we say $\Sigma=\bd\Omega$ is {\em $G$-stable in $U$} if $\delta^{2} \mathcal{A}^{h}|_{\Omega}(X, X)\geq 0$ for any $X\in \mk X^G(U)$ with $X=\varphi\nu$ on $\Sigma$. 
    Additionally, the {\em equivariant Morse index $\Index_G(\Sigma)$} of $\Sigma$ is defined as the number of negative eigenvalues (counting multiplicities) of the (intrinsic) restricted quadratic form $II_\Sigma \llcorner C^\infty_G(\Sigma)$.  
\end{definition}
\begin{remark}
    If $\Sigma\in \mc P^{G,h}$ is embedded, then it is $G$-stable in $M$ if and only if $II^h_\Sigma\geq 0$ as a quadratic form on $C^\infty_G(\Sigma)$ (i.e. $\Index_G(\Sigma)=0$). 
    But for almost embedded $\Sigma\in \mc P^{G,h}$ with non-empty touching set, (intrinsic) $\Index_G(\Sigma)=0$ is slightly stronger than (extrinsic) $G$-stable. 
    This is also true for $\Hat{\Sigma}\in \mc P^{G_\pm,\hat{h}}$. 
\end{remark}

\begin{lemma}\label{Lem: stable and G-stable}
    For any open $G$-set $U\subset M$, let $h\in C^\infty_G(M)$ and $\Sigma=\bd\Omega\in\mc P^{G,h}$ be embedded. 
    Then $\Sigma$ is stable in $U$ if and only if it is $G$-stable in $U$. 
    In addition, for any $\hat{h}\in C^\infty_{G_\pm}(M)$, $p\in M$, and $0<r<\inj(G\cdot p)$, the same equivalence result also holds for embedded $\Hat\Sigma=\bd\Hat{\Omega}\in\mc P^{ G_\pm,\hat{h}}$ in any open $G$-set $U\subset B_{r}(G\cdot p)$ provided \eqref{Eq: free Gpm}. 
\end{lemma}
\begin{proof}
    For $h\in C^\infty_G(M)$ and $\Sigma=\bd\Omega\in\mc P^{G,h}$, the `only if' part is clear. 
    For the `if' part, the embeddedness of $\Sigma$ implies that we can take the first eigenvector field $X$ and the first eigenfunction $f$ of the Jacobi operator $L_\Sigma^h$ of $\Sigma$ so that $X=f\nu$. 
    Since $\bd_\nu h$ and $\Sigma$ are $G$-invariant, we know $f\circ g$ is also the first eigenfunction of the Jacobi operator $L_\Sigma^h$. 
    Noting that $f$ has a sign and is of multiplicity one, we see $f\circ g=f$ for any isometry $g\in G$, and thus $f\in C^\infty_G(M)$. 
    Using the $G$-invariance of the outward unit normal $\nu$ on $\bd\Omega$, we know $X=f\nu$ is $G$-invariant, which implies the `if' part. 

    Given $p\in M$, $0<r<\inj(G\cdot p)$, and open $G$-set $U\subset B_r(G\cdot p)$, let $U_\pm:=U\cap B_r(G_\pm\cdot p)$ be two disjoint open $G_+$-sets by \eqref{Eq: free Gpm}. 
    Then for $\hat{h}\in C^\infty_{G_\pm}(M)$ and embedded $\Hat\Sigma=\bd\Hat{\Omega}\in\mc P^{G_\pm,\hat{h}}$, we know $\Hat\Sigma$ is stable in $B_r(G\cdot p)$ if and only if it is $G_+$-stable in $U_+$ and $U_-$ respectively. 
    Hence, the desired equivalence follows from the fact that $C^\infty_{G_+}(U_\kappa)=\{f\llcorner U_\kappa: f\in C^\infty_{G_\pm}(U)\}$ for $\kappa\in\{+,-\}$. 
\end{proof}

Given $\Omega\in\mc C(M)$ and an open set $U$, we say $\Omega$ is {\em $\mc A^h$-minimizing in $U$} if for any other $\Lambda\in\mc C(M)$ with $\spt\|\Lambda-\Omega\|\subset U$, we have $\mc A^h(\Lambda)\geq \mc A^h(\Omega)$. 
The regularity of $\mc A^h$-minimizing boundaries can be obtained by Morgan's regularity result \cite{morgan2003regularity} for isoperimetric hypersurfaces (see \cite{zhou2020pmc}*{Theorem 2.2}). 
Using \cite{wang2023s1-stability}*{Proposition 3.4}, we also have the following regularity theorem. 

\begin{theorem}\label{Thm: regularity of G,Ah-minimizers}
    Let $(h,\Omega)\in C^\infty_G(M)\times \mc C^G(M)$, 
    $p\in\spt\|\bd\Omega\|$, and $0<s<\inj(G\cdot p)$. 
    Suppose $\Omega\llcorner B_s(G\cdot p)$ $G$-equivariantly minimizes the $\mc A^h$-functional: that is, $\mc A^h(\Lambda)\geq\mc A^h(\Omega)$ for any $\Lambda\in\mc C^G(M)$ with $\spt\|\Lambda-\Omega\|\subset B_s(G\cdot p)$. 
    Then $\Omega$ is $\mc A^h$-minimizing in $B_s(G\cdot p)$, and $\bd\Omega\llcorner B_s(G\cdot p)$ is a smoothly embedded $G$-hypersurface except for a $G$-set of Hausdorff dimension at most $n-7$. 

    In addition, suppose the $G_+$-action is free from $G_-$ in the sense of \eqref{Eq: free Gpm} and $(h,\Omega)\in C^\infty_{G_\pm}(M)\times \mc C^{G_\pm}(M)$ satisfies that $\mc A^h(\Lambda)\geq\mc A^h(\Omega)$ for any $\Lambda\in\mc C^{G_\pm}(M)$ with $\spt\|\Lambda-\Omega\|\subset B_s(G\cdot p)$. 
    Then the same regularity result also holds for $\bd\Omega\llcorner B_s(G\cdot p)$. 
\end{theorem}
For simplicity, we say such an $\Omega\in\mc C^G(M)$ (resp. $\Omega\in\mc C^{G_\pm}(M)$) in the above theorem is {\em $(G,\mc A^h)$-minimizing in $B_s(G\cdot p)$} (resp. {\em $(G_\pm,\mc A^h)$-minimizing in $B_s(G\cdot p)$)}. 
\begin{proof}
    For $(h,\Omega)\in C^\infty_G(M)\times \mc C^G(M)$, the $\mc A^h$-minimizing property of $\Omega$ is given by \cite{wang2023s1-stability}*{Proposition 4.3, Claim 1}, and the regularity follows immediately from \cite{zhou2020pmc}*{Theorem 2.2} and the $G$-invariance of $\Omega$. 

    If \eqref{Eq: free Gpm} is satisfied and $(h,\Omega)\in C^\infty_{G_\pm}(M)\times \mc C^{G_\pm}(M)$ is $(G_\pm,\mc A^h)$-minimizing in $B_s(G\cdot p)$. 
    Let $U_\pm := B_s(G_\pm\cdot p)$ be two disjoint open $G$-set. 
    Then, for any $\Lambda\in\mc C^{G_\pm}(M)$ with $\spt\|\Lambda-\Omega\|\subset B_s(G\cdot p)$, since $h\in C^\infty_{G_\pm}(M)$ and $G_-\cdot (\Omega\cap U_+) = U_-\setminus \Omega$ as Caccioppoli sets in $U_-$, we have 
    \begin{align}
        \mc H^n(\bd\Lambda\cap U_-) - \int_{\Lambda\cap U_-} h d\mc H^{n+1} = \mc H^n(\bd\Lambda\cap U_+) - \int_{\Lambda\cap U_+} h d\mc H^{n+1} + \int_{U_+} h d\mc H^{n+1},
    \end{align}
    and 
    \[\mc A^h(\Lambda) = 2\left( \mc H^n(\bd\Lambda\cap U_+) - \int_{\Lambda\cap U_+} h d\mc H^{n+1}\right) + C(\Omega,h,U_+),\]
    where $C(\Omega,h,U_+)$ is a constant given by $\int_{U_+} h d\mc H^{n+1} + \mc H^n(\bd\Omega\setminus B_s(G\cdot p)) - \int_{\Omega\setminus B_s(G\cdot p)} h d\mc H^{n+1}$. 
    For any $\Lambda_0'\in\mc C^{G_+}(M)$ with $\spt\|\Lambda_0'-\Omega\|\subset U_+$, set $\Lambda_0:=(\Lambda_0'\setminus U_-) \cup (U_-\setminus G_-\cdot \Lambda_0')\in \mc C^{G_\pm}(M)$ satisfying $\spt\|\Lambda_0-\Omega\|\subset B_s(G\cdot p)$. 
    Then the above formula indicates that $\mc A^h(\Lambda_0')-\mc A^h(\Omega) = \frac{1}{2}(\mc A^h(\Lambda_0)-\mc A^h(\Omega)) \geq 0$.
    Hence, $\Omega\in\mc C^{G_+}(M)$ is $(G_+,\mc A^h)$-minimizing in $U_+=B_s(G_+\cdot p)$, and the desired regularity of $\bd\Omega\llcorner B_s(G\cdot p)$ follows from the first part of the theorem. 
\end{proof}


\subsection{Prescription functions and generic good pairs}\label{Subsec: prescription G-functions}


The prescription functions $h,\hat{h}\in C^\infty(M)$ we considered satisfies the following properties:
\begin{enumerate}
    \item $h\in C^\infty_G(M)$ (resp. $\hat{h}\in C^\infty_{G_\pm}(M)$) so that the critical locus of $h$ (resp. $\hat{h}$) is a union of non-degenerate critical $G$-orbits (in the sense of \cite{wasserman1969equivariant}); 
    \item the zero set $\Sigma_0:=\{h=0\}$ (resp. $\Hat\Sigma_0:=\{\hat h = 0\}$) is a smooth $G$-hypersurface whose mean curvature vanishes at most finite order. 
\end{enumerate}
Let $\mc S_G=\mc S_G(g_{_M})$ and $\mc S_{G_\pm}=\mc S_{G_\pm}(g_{_M})$ be the set of all smooth functions in $C^\infty_G(M)$ and $C^\infty_{G_\pm}(M)$ with the above properties respectively. 
The following result is similar to \cite{zhou2020pmc}*{Proposition 3.7}. 

\begin{proposition}\label{Prop: generic PMC G-functions}
    $\mc S_G$ is open and dense in $C^\infty_G(M)$. 
    In addition, $\mc S_{G_\pm}$ is open and dense in $C^\infty_{G_\pm}(M)$ provided \eqref{Eq: free Gpm}. 
\end{proposition}
\begin{proof}
    Firstly, consider $\mc S_G\subset C^\infty_G(M)$. 
    By \cite{wasserman1969equivariant}, $G$-equivariant Morse functions on $M$ are open and dense in $C^\infty_G(M)$, and a $G$-equivariant Morse function on $M$ has only finitely many critical orbits. 
    Hence, the set 
    \begin{align}\label{Eq: S_G^0}
        \mc S_G^0:=\{\mbox{$G$-equivariant Morse functions $h$ with $\{h=\nabla h =0\}=\emptyset$}\}
    \end{align}
    is also open and dense in $C^\infty_G(M)$. 
    In addition, combining \cite{zhou2020pmc}*{Proposition 3.8} with the uniqueness of the mean curvature flow/level set flow (see \cite{evans1991motion}), we know the level set flow $\{\tilde{h}^t\}_{t\in[0,t_0]}$ starting from $h\in\mc S^0_G$ keeps to be $G$-invariant. 
    With these facts, the arguments in \cite{zhou2020pmc}*{Proposition 3.7} would carry over to show that $\mc S_G$ is open and dense in $C^\infty_G(M)$. 

    Next, consider $\mc S_{G_\pm}\subset C^\infty_{G_\pm}(M)$. 
    For any $G$-subset $A\subset M$, let $\mc M_{G_\pm}(A,M)$ be the set of functions $\hat h\in C^\infty_{G_\pm}(M)$ whose critical locus in $A$ is a union of non-degenerate critical $G$-orbits. 

    \begin{claim}\label{Claim: density of Gpm Morse functions}
        $\mc M_{G_\pm}(M,M)$ is open and dense in $C^\infty_{G_\pm}(M)$ provided \eqref{Eq: free Gpm}. 
    \end{claim}
    \begin{proof}[Proof of Claim \ref{Claim: density of Gpm Morse functions}]
        The openness is clear. 
        As for the density, we need \eqref{Eq: free Gpm} to get $B_{2r}(G\cdot p)=B_{2r}(G_+\cdot p)\sqcup B_{2r}(G_-\cdot p)$ for any $p\in M$ and $0<r<\inj(G\cdot p)/2$. 
        By the density of $G_+$-equivariant Morse functions (\cite{wasserman1969equivariant}*{Lemma 4.8}), every $f\in C^\infty_{G_\pm}(M)\subset C^\infty_{G_+}(M)$ can be arbitrarily approximated by some $G_+$-equivariant Morse function $f'\in C^\infty_{G_+}(M)$. 
        Take a $G$-invariant cut-off function $\varphi\in C^\infty_G(M)$ so that $\varphi\llcorner B_r(G\cdot p)=1$ and $\varphi\llcorner (M\setminus B_{2r}(G\cdot p))=0$. 
        Define 
        \[ F:=\left\{
        \begin{array}{ll}
             f' & \mbox{in $B_{2r}(G_+\cdot p)$}  \\
             -f'\circ g_-  & \mbox{in $B_{2r}(G_-\cdot p)$}
        \end{array}
        \right. ,
        \]
        where $g_-\in G_-$. 
        One easily verifies that $F$ is well-defined and $\varphi\cdot F$ as well as $(1-\varphi)\cdot f$ is $G_\pm$-signed symmetric smooth functions. 
        Thus, $f'':= \varphi\cdot F+(1-\varphi)f \in C^\infty_{G_\pm}(M)$ satisfies
        \[\|f''-f\|_{C^k(M)} = \|\varphi\cdot(F-f)\|_{C^k(M)}=\|\varphi\cdot(F-f)\|_{C^k(B_{2r}(G\cdot p))} = \|\varphi\cdot (f'-f)\|_{C^k(B_{2r}(G_+\cdot p))},\]
        for all $k\in\mb N$. 
        Therefore, $f\in C^\infty_{G_\pm}(M)$ can be arbitrarily approximated by some $f''\in C^\infty_{G_\pm}(M)$ whose critical locus in $B_{r}(G\cdot p)$ is a union of non-degenerate critical $G$-orbits. 
        In particular, $\mc M_{G_\pm}(\closure(B_{r/2}(G\cdot p)),M)$ is open and dense in $C^\infty_{G_\pm}(M)$. 
        Together with the compactness of $M$ and Baire’s theorem, we know $\mc M_{G_\pm}(M,M)$ is dense in $C^\infty_{G_\pm}(M)$. 
    \end{proof}
    Consider the set $\mc S^0_{G_\pm}$ 
    given by 
    \begin{align}\label{Eq: S_Gpm^0}
        \mc S^0_{G_\pm}:=\{\mbox{$\hat h\in\mc M_{G_\pm}(M,M)$: $\{\hat h=\nabla \hat h=0\}=\emptyset$}\}.
    \end{align}
    Similar to the proof of Claim \ref{Claim: density of Gpm Morse functions}, we know $\mc S^0_{G_\pm}$ is open and dense in $C^\infty_{G_\pm}(M)$ provided \eqref{Eq: free Gpm}. 
    Then, the proof of \cite{zhou2020pmc}*{Proposition 3.7} would carry over as the previous case. 
\end{proof}

\begin{remark}\label{Rem: small touching and no minimal}
    Notice that every function $h$ in $\mc S_G$ or $\mc S_{G_\pm}$ also satisfies \cite{zhou2020pmc}*{\S 3.2 ($\dagger$)}. 
    Hence, it follows directly from \cite{zhou2020pmc}*{Proposition 3.16} that any almost embedded $G$-hypersurface $\Sigma$ of prescribed mean curvature $h\in \mc S_G$ or $\mc S_{G_\pm}$ has its touching set (i.e. the non-embedding part of $\Sigma$) contained in a countable union of connected, embedded $(n-1)$-dimensional submanifolds. 
    Moreover, the definitions of $\mc S_G$ and $\mc S_{G_\pm}$ also imply that such a $\Sigma$ has no minimal component. 
\end{remark}


Next, we denote by 
\begin{align}\label{Eq: G-metrics space}
    \Gamma_G:=\{\mbox{$G$-invariant Riemannian metrics on $M$}\},
\end{align}
and make the following useful definitions.

\begin{definition}\label{Def: good pairs}
    Let $M$ be a closed connected $(n+1)$-manifold and $G$ be a compact Lie group acting effectively by diffeomorphisms on $M$ so that $3\leq \codim(G\cdot p)\leq 7$ for all $p\in M$.
    The pair $(g_{_M}, h)\in\Gamma_G\times C^\infty_G(M)$ 
    is called a {\em $G$-equivariant good pair} (or a {\em good $G$-pair} for short) if 
    \begin{itemize}
        \item[(1)] $h\in\mc S_G(g_{_M})$; and 
        \item[(2)] $g_{_M}$ is {\em $G$-bumpy} for $\mc P^{G,h}$, i.e., every $\Sigma\in\mc P^{G,h}$ under the metric $g_{_M}$ is non-degenerate (nullity equal to zero). 
    \end{itemize}
    Additionally, suppose $G$ has a $2$-index Lie subgroup $G_+$ with co-set $G_-=G\setminus G_+$ satisfying \eqref{Eq: free Gpm}. 
    Then the pair $(g_{_M}, \hat h)\in \Gamma_G\times C^\infty_{G_\pm}(M)$ 
    is called a {\em good $G_\pm$-pair} if the above (1)(2) hold with $\mc S_{G_\pm}(g_{_M})$ and $ \mc P^{G_\pm,\hat h}$ in place of $\mc S_{G}(g_{_M})$ and $ \mc P^{G,h}$ respectively. 
\end{definition}

Recall that the set $\mc S_G^0$ defined in \eqref{Eq: S_G^0} (resp. $\mc S_{G_\pm}^0$ in \eqref{Eq: S_Gpm^0}) is open and dense in $C^\infty_G(M)$ (resp. in $C^\infty_{G_\pm}(M)$ provided \eqref{Eq: free Gpm}), which is independent to the choice of the $G$-metric. 
Then we have the following lemma generalized from \cite{zhou2020multiplicity}*{Lemma 3.5}

\begin{lemma}\label{Lem: generic good pairs}
    Given $h\in\mc S^0_G$, the set $\{g_{_M}\in\Gamma_G: (g_{_M}, h)\mbox{ is a good $G$-pair}\}$ is generic in the Baire sense. 
    Additionally, for $\hat h\in\mc S^0_{G_\pm}$, 
    the set $\{g_{_M}\in\Gamma_G: (g_{_M}, h)\mbox{ is a good $G_\pm$-pair}\}$ is generic in the Baire sense provided \eqref{Eq: free Gpm}. 
\end{lemma}

\begin{proof}
    Firstly, we show $\Gamma_G^1(h):=\{g_{_M}\in\Gamma_G: h\in \mc S_G(g_{_M})\}$ is open and dense in $\Gamma_G$. 
    Since the vanishing order of $H_{\{h=0\}}$ is bounded under a small perturbation of $g_{_M}$, we see $\Gamma_G^1(h)$ is open. 
    Next, fix any $g_{_M}\in \Gamma_G$ and denote by $H_0$ the mean curvature of $\{h=0\}$ under $g_{_M}$ with respect to a unit normal $\nu_{\{h=0\}}$. 
    Note $H_0$ and $\nu_{\{h=0\}}$ are $G$-invariant. 
    By \cite{wasserman1969equivariant}, we can take an arbitrarily small $\varphi\in C^\infty_{G}(\{h=0\})$ so that the critical locus of $H_0-n\varphi$ on $\{h=0\}$ is a union of non-degenerate critical orbits. 
    Let $\phi(x)=d(x)\eta(x)\varphi(\bar{n}(x)) \in C^\infty_G(M)$, where $d$ is the signed $g_{_M}$-distance function to $\{h=0\}$ so that $(\nabla d) \llcorner\{h=0\} = \nu_{\{h=0\}}$, $\eta(x)$ is a $G$-invariant cut-off function on $M$ so that $\eta\llcorner B_r(\{h=0\}) = 1$ and $\eta\llcorner M\setminus B_{2r}(\{h=0\}) = 0$ for some small $r>0$, and $\bar{n}:B_{2r}(\{h=0\}) \to \{h=0\}$ is the $g_{_M}$-geodesic nearest projection. 
    Then the metric $g_{_M}':=e^{2\phi}g_{_M}$ is close to $g_{_M}$, and the mean curvature of $\{h=0\}$ under $g_{_M}'$ with respect to $\nu_{\{h=0\}}$ is given by $H_0-n\varphi$ (see the proof of \cite{irie2018density}*{Proposition 2.3}), which vanishes to at most finite order. 
    This proved the density of $\Gamma_G^1(h)$. 

    If $\hat{h}\in \mc S_{G_\pm}^0$, then $H_0$, $\nu_{h=0}$, and $d$ are all $G_\pm$-signed symmetric. 
    Using Claim \ref{Claim: density of Gpm Morse functions} in place of \cite{wasserman1969equivariant}, we also make $\varphi\in C^\infty_{G_\pm}(\{h=0\})$, which implies $\phi$ is still $G$-invariant. 
    Hence, the above arguments also show that $\Gamma_G^1(\hat{h}):=\{g\in\Gamma_G: \hat{h}\in\mc S_{G_\pm}(g_{_M})\}$ is open and dense. 

    Next, we show condition (2) in Definition \ref{Def: good pairs} is generic. 
    In \cite{white1991space}*{\S 7}, White has shown that the bumpy metric theorem \cite{white1991space}*{Theorem 2.2} also holds for simple immersed CMC hypersurfaces. 
    This proof can extends to the case of simple immersed (see \cite{white1991space}*{P. 197}) PMC hypersurfaces. 
    Indeed, using the notations in \cite{white1991space}*{\S 7}, one only needs to modify $\omega_\gamma(x,y)$ on P.197 so that $d\omega_\gamma = E^*(h\cdot \Omega_\gamma)$ for $h\in C^\infty(M)$, and remove `$h$' in the first two formulas on P.198, where $\Omega_\gamma$ is the $\gamma$-volume form of the ambient manifold. 
    In addition, suppose $h,\gamma$ and the immersion $u$ are $G$-invariant so that $E\circ u(\Sigma)=\bd\Omega$ for some $\Omega\in \mc C^G(M)$. 
    Then, on \cite{white1991space}*{P. 198}, $H(\gamma,u)$ and $H_0(\gamma,u)$ in the 2nd formula are $G$-invariant. 
    Hence, by Lemma \ref{Lem: 1st variation for G-boundary}, $[E\circ u]$ has mean curvature $h$ if and only if $u$ is {\em $G$-stationary}  for the $A^*$-functional on P.198 (e.g. take $G$-invariant $v=H(\gamma,u)-H_0(\gamma,u)$). 
    Combining with the equivariant generalization in \cite{wang2023G-index}*{\S 3}, we know the set of $G$-metrics admitting no degenerate simple immersed elements in $\mc P^{G,h}$ is generic in the Baire sense for a fixed $h\in \mc S^0_G$. 
    Moreover, one notices that $\Gamma_G^1(h)$ is open and dense, and every $\Sigma\in \mc P^{G,h}$ under the metric $g_{_M}\in \Gamma_G^1(h)$ is simple immersed (by Remark \ref{Rem: small touching and no minimal}). 
    Hence, we get the desired result for $h\in \mc S_G^0$.

    If $\hat{h}\in \mc S_{G_\pm}^0$, $u$ is $G$-equivariant, and $E\circ u(\Sigma)=\bd\Omega$ for some $\Omega\in\mc C^{G_\pm}(M)$, then using the notations in \cite{white1991space}*{\S 7}, $g\in G_\pm$ acts on the normal bundle $V\cong \Sigma\times \mb R$ by $g\cdot (x,y)=(g\cdot x, \pm y)$. 
    Hence, by a similar modification using $\hat{h}$, $H(\gamma,u)$ and $H_0(\gamma,u)$ in \cite{white1991space}*{P.198} are still $G$-invariant since $\hat{h}$ and the outer unit normal of $\Omega$ will change/fix sign simultaneously under $G$-actions. 
    The proof can be finished similarly to the previous case. 
\end{proof}

\subsection{Compactness of stable $(G,h)$-hypersurfaces}
At the end of this section, we introduce the following compactness and regularity theorem for stable (not $G$-stable) $(G,h)$-hypersurfaces. 

\begin{theorem}\label{Thm: compactness for stable (G,h)-hypersurfaces}
	Let $3\leq \codim(G\cdot p)\leq 7$ for all $p\in M$, $U\subset M$ be an open $G$-subset, and $\{h_k\}_{k\in\mb N}\subset \mc S_G$ with $\lim_{k\to\infty} h_k = h_\infty$ in the $C^\infty$-topology, where $h_\infty\in\mc S_G$ or $h_\infty = 0$. 
	Suppose $\Sigma_k\subset U$ is a sequence of smooth, almost embedded, stable $(G,h_k)$-hypersurfaces in $U$ satisfying $\sup_k \Area(\Sigma_k)<+\infty$ and $\Sigma_k=\bd\Omega_k$ in $U$ for some $\Omega_k\in\mc C^G(U)$. 
	Then up to a subsequence, $\Sigma_k$ converges (possibly with multiplicity) in the varifold sense 
	to a smooth, stable, almost embedded $(G,h_\infty)$-hypersurface $\Sigma_\infty\subset U$ so that 
	\begin{itemize}
		\item[(i)] the convergence $\Sigma_k\to\Sigma_\infty$ is also locally smooth and graphical in $U$;
		\item[(ii)] if $h_\infty\in \mc S_G$, then $\Sigma_\infty$ has multiplicity $1$, and $\Sigma_\infty = \bd\Omega_\infty$ in $U$ for some $\Omega_\infty\in \mc C^G(U)$.
	\end{itemize}
	The same result also holds with $\mc S_{G_\pm}$ and $\mc C^{G_\pm}(U)$ in place of $\mc S_{G}$ and $\mc C^{G}(U)$ respectively.
\end{theorem}
\begin{proof}
	Without group actions, the regularity of $\Sigma_\infty$ and the locally smooth graphical convergence have been proved in \cite{zhou2020pmc}*{Theorem 3.19} for $3\leq n+1\leq 7$ and also in \cite{bellettini2020stablePMC}*{Corollary 1.3, Remark 7.1} for $n+1\geq 3$, where a singular set is allowed in $\Sigma_\infty$ with Hausdorff dimension at most $n-7$. 
	Since each $\Sigma_k$ as well as $\Omega_k$ is $G$-invariant, we know $\Sigma_\infty$ and $\Omega_\infty=\lim_{k\to\infty}\Omega_k$ (as Caccioppoli sets) are $G$-invariant. 
	In particular, the singular set of $\Sigma_\infty$ is $G$-invariant, which is either empty or has dimension at least $n+1-\max\{\codim(G\cdot p): p\in M\}\geq n-6$. 
	Hence, the regularity result in \cite{bellettini2020stablePMC}*{Corollary 1.3, Theorem 1.5(1)} implies $\Sigma_\infty$ is smooth and almost embedded everywhere. 
	Using Remark \ref{Rem: small touching and no minimal} and the arguments in \cite{zhou2020pmc}*{Theorem 3.19} (see also \cite{zhou2020multiplicity}*{Theorem 2.6 (iii)}), we know $\Sigma_\infty=\bd\Omega_\infty$ in $U$ provided $h_\infty\in\mc S_G$. 
	The proof is the same in the $G_\pm$-setting. 
\end{proof}

\section{Relative equivariant min-max theory for prescribing mean curvature hypersurfaces}\label{Sec: PMC min-max}

We now show the multi-parameter min-max theory for PMC hypersurface to certain symmetric settings. 
Specifically, let 
\begin{align}\label{Eq: mc G}
    \mc G=G \qquad{\rm or}\qquad \mc G= G_\pm \mbox{ satisfying \eqref{Eq: free Gpm}}. 
\end{align}
Then in this section we consider the subspace $\cCcG(M)$ of $\mc G$-elements in $\mc C(M)$, and the set $\cScG=\cScG(g_{_M})\subset C^\infty_{\mc G}(M)$ (see Section \ref{Subsec: prescription G-functions}) of the smooth prescription $\mc G$-functions $h$. 

\subsection{Equivariant min-max constructions for $(X,Z)$-homotopy class}\label{Subsec: min-max setup}

For $m,j\in\mb Z_+$, define $I(1,j)$ as the cubical complex on $[0,1]$ with $0$-cells $\{k\cdot 3^{-j}\}$ and $1$-cells $\{[k\cdot 3^{-1}, (k+1)\cdot 3^{-1}]\}$, $k=0,1,\dots,3^j$. Define $I(m,j):= I(1,j)^{\otimes m}$ as the $m$-dimensional cubical complex. 

Let $X^k$ be a $k$-dimensional cubical complex in some $I(m,j)$ and $Z\subset X$ be a cubical subcomplex. 
For $q\leq l\leq m \in \mb N$, denote by $X(l)$ the union of cells of $I(m,j+l)$ with support contained in some cell of $X$, and by $X(l)_q$ the set of all $q$-cells in $X(l)$. 
Similarly, for an $p$-cell $\alpha\in X(l)_p$, denote by $\alpha_q$ the set of $q$-cells supported in $\alpha$. 
Then for any $i,j\in\mb N$, define $\mf n(i,j):X(i)_0\to X(j)_0$ as the nearest projection. 

\begin{definition}\label{Def: relative homotopy sequence of mappings into C}
    Let $(X,Z)$ be a pair of cubical complex as above, and $\Phi_0:X\to (\cCcG(M), \mf F)$ be a continuous map (under the $\mf F$-topology on $\cCcG(M)$). 
    Then a sequence of continuous maps $\{\Phi_i:X\to (\cCcG(M),\mf F)\}_{i\in\mb N}$ is said to be a {\em $(X,Z)$-homotopy sequence of mappings into $\cCcG(M)$} if for each $i\in\mb N$, there exists a flat $\mc F$-continuous homotopy map $\Psi_i:[0,1]\times X\to \cCcG(M)$ so that $\Psi_i(0,\cdot)=\Phi_i,\Psi(1,\cdot)=\Phi_0$, and 
    \begin{align}
        \limsup_{i\to\infty}\sup\{\mf F(\Psi_i(t,x),\Phi_0(x)) : t\in [0,1], x\in Z\} = 0.
    \end{align}
    The set of all such sequences is called the {\em $(X,Z)$-homotopy $\mc G$-class} of $\Phi_0$ and denoted by $\Pi$. 
\end{definition}

\begin{definition}\label{Def: h-width and minimizing sequence}
    The {\em $\cGh$-width} of $\Pi$ is defined by 
    \[\mf L^h=\mf L^h(\Pi):=\inf_{\{\Phi_i\}\in\Pi}\limsup_{i\to\infty}\sup_{x\in X}\{\mc A^h(\Phi_i(x))\}.\]
    A sequence $\{\Phi_i\}_{i\in\mb N}\in\Pi$ is called a {\em minimizing sequence} if 
    \[\mf L^h(\Pi) = \mf L^h(\{\Phi_i\}_{i\in\mb N}):=\limsup_{i\to\infty}\sup_{x\in X}\mc A^h(\Phi_i(x)).\]
\end{definition}

By a diagonalization process similar to \cite{zhou2020multiplicity}*{Lemma 1.5}, the min-max sequence exists for any $\Phi_0$ and $\Pi$. 
If $\{\Phi_i\}_{i\in\mb N}$ is a minimizing sequence of $\Pi$, then we define the {\em critical set} of $\{\Phi_i\}$ by 
\[\mf C(\{\Phi_i \}_{i\in\mb N}) := \{V=\lim_{j\to\infty}|\bd\Phi_{i_j}(x_j)| \in \mc V^G(M) : {\rm with~} \mc A^h(\Phi_{i_j}(x_j)) = \mf L(\Pi) \}  .\]

We can now state the main result of this section, that is a relative equivariant min-max theorem for symmetric PMC hypersurfaces. 
We will have several subsections for preparation, and the proof is left in Section \ref{Subsec: proof of PMC Equivariant min-max}.

\begin{theorem}[Equivariant min-max theorem]\label{Thm: Equivariant min-max for PMC}
    Let $(M^{n+1},g_{_M})$ be a closed Riemannian manifold with a compact Lie group $G$ acting by isometries so that $3\leq \codim(G\cdot p)\leq 7$ for all $p\in M$, and let $h\in\cScG$ with $\int_M h\geq 0$. 
    Given cubical complexes $Z\subset X^m$ and an $\mf F$-continuous map $\Phi_0: X\to (\cCcG(M),\mf F)$, suppose the associated $(X,Z)$-homotopy $\mc G$-class $\Pi$ satisfies that 
    \begin{align}\label{Eq: width > relative}
        \mf L^h(\Pi) > \max\left\{ \max_{x\in Z}\mc A^h(\Phi_0(x), 0) \right\},
    \end{align}
    and $\{\Phi_i\}_{i\in\mb N}\in \Pi$ is a minimizing sequence. 
    Then there exists $V=\lim_{j\to\infty}|\bd\Phi_{i_j}(x_j)|\in \mf C(\{\Phi_i\})$ for some $\{i_j\}\subset\{i\}$ and $\{x_j\}\subset X\setminus Z$ with $\lim_{j\to\infty}\mc A^h(\Phi_{i_j}(x_i))=\mf L^h(\Pi)$ so that 
    \begin{itemize}
        \item[(i)] $\Phi_{i_j}(x_j)$ converges in the $\mf F$-topology to some $\Omega\in \cCcG(M)$ with $\mc A^h(\Omega)=\mf L^h(\Pi)$;
        \item[(ii)] $V=|\bd\Omega|$, and $\Sigma=\bd\Omega$ is a nontrivial smooth, closed, almost embedded $G$-hypersurface with prescribed mean curvature $h$ (with respect to the outer unit normal $\nu_{\bd\Omega}$). 
    \end{itemize}
\end{theorem}

\subsection{Discretization and interpolation results}\label{Subsec: discret/interpolate}

We record several discretization and interpolation results developed by Marques-Neves \cite{marques2014willmore}\cite{marques2017existence} (for closed manifolds), and extended by the author \cite{wang2022min}\cite{wang2023free} (in equivariant settings). 
We here only point out some modifications to adapt for sweepouts in $\cCcG(M)$. 

Firstly, we refer to \cite{zhou2020pmc}*{\S 4} for the notations of discrete sweepouts, where the domain $[0,1]$ and the target space $\mc C(M)$ are replaced by $X$ and $\cCcG$ respectively. 
Next, for any discrete map $\phi: X(k)_0\to\cCcG$, the {\em fineness} of $\phi$ is defined as 
\[\mf f(\phi) :=\{\mc F(\phi(x)-\phi(y)) + \mf M(\bd\phi(x)-\bd\phi(y)) : \{x,y\}=\alpha_0,\alpha\in X(k)_1\}.\]
In addition, given a flat $\mc F$-continuous map $\Phi: X\to \cCcG(M)$, we say $\Phi$ has {\em no concentration of mass on orbits} if 
\[\lim_{r\to 0}\sup\{\|\bd\Phi(x)\|(B_r(G\cdot p)) : p\in M, x\in X\} = 0.\]
This is a mild technical condition, which is satisfied if $\bd\circ\Phi$ is $\mf F$-continuous (\cite{wang2022min}*{Lemma 8}). 

Using the following discretization theorem, we can generate
a homotopy sequence of discrete maps into $\cCcG(M)$ from certain continuous maps $\Phi$. 

\begin{theorem}[Discretization, \cite{zhou2020multiplicity}*{Theorem 1.11},\cite{wang2022min}*{Theorem 2}]\label{Thm: discretization}
Let $\Phi: X\to \cCcG(M)$ be a flat $\mc F$-continuous map with no concentration of mass on orbits and $\sup_{x\in X}\mf M(\bd\Phi(x))<+\infty$. 
Suppose $\Phi\llcorner Z$ is continuous in the $\mf F$-topology. 
Then there exist sequences of maps 
\[\phi_i: X(k_i)_0\to\cCcG(M),\qquad \psi_i: I(k_i)_0\times X(k_i)_0\to\cCcG(M),\]
with $k_i<k_{i+1}$, $\psi_i(0,\cdot)=\phi_{i-1}\circ\mf n(k_i,k_{i-1})$, $\psi_i(1,\cdot)=\phi_i$, and a sequence of $\delta_i\to 0$ so that 
\begin{enumerate}[label=(\roman*)]
    \item 
    $\mf f(\psi_i)<\delta_i$;
    \item $\sup\{\mc F(\psi_i(t,x)-\Phi(x)) : t\in I(k_i)_0, x\in X(k_i)_0\}\leq\delta_i$;
    \item for some sequence $l_i\to\infty$ with $l_i<k_i$, 
    \[
    \mf M(\bd \psi_i(t,x)) \leq \sup\{\mf M(\bd\Phi(y)): x,y\in\alpha,\alpha\in X(l_i)\} + \delta_i,
    \]
    which directly implies 
    \[\sup\{\mf M(\bd\phi_i(x)): x\in X(k_i)_0\}\leq \sup\{\mf M(\bd\Phi(x)): x\in X\} + \delta_i.\]
\end{enumerate}
In addition, it follows from (iii) and the $\mf F$-continuity of $\Phi\llcorner Z$ that 
\[ \mf M(\bd\psi_i(t,x)) \leq \mf M(\bd\Phi(x)) + \eta_i, \qquad\forall t\in I(k_i)_0, x\in Z(k_i)_0, \]
for some $\eta_i\to 0$ as $i\to\infty$. 
Combining with \cite{marques2014willmore}*{Lemma 4.1} and (ii), we get 
\begin{enumerate}
    \item[(iv)] $\sup\{\mf F(\psi_i(t,x), \Phi(x)): t\in I(k_i)_0, x\in Z(k_i)_0\} \to 0$ as $i\to\infty$;
    \item[(v)] for any $h\in C^\infty_{\mc G}(M)$ with $c=\sup_M |h|$, we have from (ii)(iii) that 
    \[\Ah(\phi_i(x))\leq \sup\{\Ah(\Phi(y)): x,y\in\alpha, \alpha\in X(l_i)\} + (1+c)\delta_i,\]
    and thus, $\sup\{\Ah(\phi_i(x)) : x\in X(k_i)_0\}\leq \sup\{\Ah(\Phi(x)): x\in X\} + (1+c)\delta_i$. 
\end{enumerate}
\end{theorem}

\begin{proof}
    This theorem was proved in \cite{marques2014willmore}*{Theorem 13.1}\cite{marques2017existence}*{Theorem 3.9} and \cite{wang2022min}*{Theorem 2} for mappings into $\mc Z_n(M;\mb Z_2)$ and $\mc B^G(M;\mb Z_2)$ respectively. 
    Using the isoperimetric lemma \cite{wang2024Ricci}*{Lemma 2.2}, the arguments would also carry over in our setting. 
\end{proof}

We can also get an $\mf F$-continuous map into $\cCcG(M)$ from a discrete map with small fineness. 
\begin{theorem}[Interpolation, \cite{zhou2020multiplicity}*{Theorem 1.12},\cite{wang2022min}*{Theorem 3}]\label{Thm: interpolation}
    There are positive constants $C_0=C_0(M,G,m)$ and $\delta_0=\delta_0(M)$ so that if $Y$ is a cubical subcomplex of $I(m,k)$ and 
    \[\phi: Y_0\to \cCcG(M)\]
    has $\mf f(\phi)<\delta_0$, then there exists a map (called the {\em Almgren $\mc G$-extension} of $\phi$)
    \[\Phi: Y\to\cCcG(M)\]
    continuous in the $\mf F$-topology and satisfying
    \begin{enumerate}[label=(\roman*)]
        \item $\Phi(x)=\phi(x)$ for all $x\in Y_0$;
        \item for any $j$-cell $\alpha\in Y_j$, $\Phi\llcorner\alpha$ depends only on $\phi\llcorner\alpha_0$;
        \item $\sup\{\mf F(\Phi(x),\Phi(y)) : x,y \mbox{ lie in a common cell of } Y\}\leq C_0\mf f(\phi)$.
    \end{enumerate}
\end{theorem}

\begin{proof}
    This theorem was proved in \cite{marques2014willmore}*{Theorem 14.1} and \cite{wang2022min}*{Theorem 3} for mappings into $\mc Z_n(M;\mb Z_2)$ and $\mc B^G(M;\mb Z_2)$ respectively. 
    The arguments would also carry over for maps into $\mc B^{\mc G}(M;\mb Z_2)$. 
    Hence, the map $\bd\circ\phi : Y_0\to \cBcG(M;\mb Z_2)$ can be extended to a mass $\mf M$-continuous map $\Phi^\bd: Y\to \cBcG(M;\mb Z_2)$ satisfying (i)-(iii) with $\Phi^\bd,\bd \circ\phi,\mf M$ in place of $\Phi,\phi,\mf F$ respectively. 
    Next, we notice that $\bd: \cCcG(M)\to\cBcG(M;\mb Z_2)$ is a double cover (see \cite{marques2021morse}*{\S 5}, \cite{wang2023G-index}*{\S 2.1}), and using the isoperimetric lemma \cite{wang2023G-index}*{Lemma 2.1}, the procedure in \cite{marques2021morse}*{Claim 5.2} would lift $\Phi^\bd$ uniquely to a continuous map $\Phi: Y\to \cCcG(M)$ so that $\bd\circ\Phi=\Phi^\bd$ and $\Phi\llcorner Y_0=\phi$. 
    Finally, since $|\bd\Phi(x)|=|\Phi^\bd(x)|$, we know $\Phi$ is $\mf F$-continuous, and thus (i) and (ii) hold. 
    For (iii), we can use $\bd\circ\Phi=\Phi^\bd$ to see $\mf F(\Phi(x),\Phi(y)) = \mc F(\Phi^\bd(x)-\Phi^\bd(y))+\mf F(|\Phi^\bd(x)|,|\Phi^\bd(y)|)\leq 2C_0\mf f(\phi)$. 
\end{proof}

Combining the above interpolation theorem with the proof of \cite{zhou2020multiplicity}*{Proposition 1.14}, we immediately obtain the following homotopy equivalent property of Almgren $\mc G$-extensions. 
\begin{proposition}\label{Prop: homotopy equivalent Almgren extension}
    Let $Y$ be a cubical complex of $I(m,k)$, $\eta>0$, and $\phi_i: Y(l_i)_0\to\cCcG(M)$ ($i=1,2$) be two discrete maps. 
    Suppose $\phi_1$ and $\phi_2$ are {\em homotopic in $\cCcG(M)$ with fineness less than $\eta$}, that is, there exists $l>l_1,l_2$, and a map 
    \begin{align}\label{Eq: def discrete homotopy}
        \psi: I(1,k+l)_0\times Y(l)_0 \to \cCcG(M)
    \end{align}
    with $\mf f(\psi)<\eta$ and $\psi([i-1],)=\phi_i\circ \mf n(k+l,k+l_i)$ for $i=1,2$. 
    If $\eta<\delta_0$ in Theorem \ref{Thm: interpolation}, then the Almgren $\mc G$-extensions 
    \[\Phi_1,\Phi_2 : Y\to \cCcG(M)\]
    of $\phi_1,\phi_2$, respectively, are homotopic to each other in $(\cCcG(M),\mf F)$. 
\end{proposition}

Now, the proof of \cite{zhou2020multiplicity}*{Proposition 1.15} can be taken almost verbatim to our settings, and give the following result analogue to \cite{marques2017existence}*{Corollary 3.12}. 
\begin{proposition}\label{Prop: dis/inter homotopy}
    Given $\Phi: X\to (\cCcG(M),\mf F)$, let $\{\phi_i\}_{i\in\mb N}$ and $\{\psi_i\}_{i\in\mb N}$ be given by Theorem \ref{Thm: discretization} applied to $\Phi$. 
    Then the Almgren $\mc G$-extension $\Phi_i$ of $\phi_i$ is homotopic to $\Phi$ in $(\cCcG(M),\mf F)$ for $i$ large enough. 
    In particular, for $i$ large enough, there exist homotopy maps $\Psi_i: [0,1]\times X\to (\cCcG(M), \mf F)$ with $\Psi_i(0,\cdot)=\Phi_i$, $\Psi_i(1,\cdot)=\Phi$, and 
    \[\limsup_{i\to\infty}\sup_{t\in[0,1],x\in X} \mf F(\Psi_i(t,x),\Phi(x)) = 0.\]
    Hence, given $h\in C^\infty_{\mc G}(M)$, we have 
    $\overline{\lim}_{i\to\infty}\sup_{x\in X} \Ah(\Phi_i(x))\leq \sup_{x\in X}\Ah(\Phi(x))$.
\end{proposition}

\subsection{Pull-tight}
Given $\Phi_0,\Pi,$ and $h$ as in Theorem \ref{Thm: Equivariant min-max for PMC}, define $L^c:=2\mf L^h(\Pi) + c\Vol(M)$ and 
\[A^c_\infty = \{V\in\mc V^G_n(M): \|V\|(M)\leq L^c, V \mbox{ has $c$-bounded first variation or }V\in|\bd\Phi_0|(Z)\}.\]
By the constructions in \cite{wang2022Gcmc}*{\S 4} (see also \cite{zhou2020pmc}*{\S 5}), we have a continuous map 
\[H: [0,1]\times (\cCcG(M),\mf F)\cap \{\mf M(\bd\Omega)\leq L^c\} \to (\cCcG(M),\mf F)\cap \{\mf M(\bd\Omega)\leq L^c\}\]
so that $H(0,\cdot)=\Id$; $H(t,\Omega)=\Omega$ if $|\bd\Omega|\in A^c_\infty$; and $\Ah(H(1,\Omega)) < \Ah(\Omega)$ if $|\bd\Omega|\notin A^c_\infty$.
Therefore, the proof of \cite{zhou2020multiplicity}*{Lemma 1.8} would carry over to get the following lemma.
\begin{lemma}[Tightening]\label{Lem: tightening}
    For any minimizing sequence $\{\Phi_i^*\}_{i\in\mb N}\in \Pi$, the sequence $\{\Phi_i:=H(1,\Phi_i^*(\cdot))\}_{i\in\mb N}$ is also a minimizing sequence in $\Pi$ so that $\mf C(\{\Phi_i\})\subset \mf C(\{\Phi_i^*\})$ and every element of $\mf C(\{\Phi_i\})$ either has $c$-bounded first variation, or belongs to $|\bd\Phi_0|(Z)$. 
\end{lemma}

\subsection{$\cGh$-almost minimizing}
In this subsection, we introduce the concept of $(\mc G,h)$-almost minimizing varifolds, and show the existence of such varifolds in equivariant min-max constructions.

\begin{definition}\label{Def: almost minimizing varifolds}
    Let $\mf \nu$ be the $\mc F$-norm, or the $\mf M$-norm, or the $\mf F$-metric. 
    For any given $\epsilon,\delta>0$ and an open $G$-set $U \subset M$, we define $\mk a^{\mc G,h}(U;\epsilon,\delta;\nu)$ to be the set of all $\Omega\in \cCcG(M)$ such that if $\Omega= \Omega_0,\Omega_1,\Omega_2,\dots,\Omega_m \in\cCcG(M)$ is a sequence with:
    \begin{itemize}
        \item[(i)] $\spt(\Omega_i - \Omega)\subset U$, 
        \item[(ii)] $\nu(\bd\Omega_{i+1}-\bd\Omega_i)\leq \delta$, 
        \item[(iii)] $\Ah(\Omega_i)\leq\Ah(\Omega)+\delta$, for $i=1,\dots,m$, 
    \end{itemize}
    then $\Ah(\Omega_m)\geq \Ah(\Omega) - \epsilon$.
    In addition, 
    a $G$-varifold $V\in\mc V^G_n(M)$ is said to be {\em $(\mc G,h)$-almost minimizing in $U$} if for any $\epsilon>0$, there exist $\delta>0$ and $\Omega\in\mk a^{\mc G,h}(U;\epsilon,\delta;\mc F)$ with $\mf F(|\bd\Omega|, V)\leq \epsilon$.
\end{definition}

\begin{definition}\label{Def: almost minimizing in annuli}
    We say a $G$-varifold $V \in \mc V_n^G(M)$ is {\em $(\mc G, h)$-almost minimizing in annuli} if for every $p \in M$, there exists $r_{am}(G \cdot p) > 0$ such that $V$ is $(\mc G, h)$-almost minimizing in $A_{s,t}(G\cdot p)$ for any $0 < s < t \leq r_{am}(G \cdot p)$. 
\end{definition}

\begin{remark}\label{Rem: almost minimizing in annuli}
    In Definition \ref{Def: almost minimizing varifolds} and thus \ref{Def: almost minimizing in annuli}, we used $\nu = \mc F$-norm to define $\cGh$-almost minimizing. 
    Nevertheless, it is equivalent to use $\nu=\mf F$ or $\mf M$ in Definition \ref{Def: almost minimizing in annuli} by shrinking $r_{am}$ if necessary. 
    Indeed, for $\mc G=G$ and $h=0$, the desired equivalence follows from \cite{wang2023free}*{Theorem 3.14} since $B_r(G\cdot p)$ has no {\em isolated orbit} (in the sense of \cite{wang2023free}*{Definition 2.1}) provided $r<\inj(G\cdot p)$. 
    Moreover, the proof of \cite{wang2023free}*{Theorem 3.14} is mainly based on a combinatorial argument using the isoperimetric lemma and the slicing trick, which also works in our $(\mc G,h)$-setting. 
\end{remark}

We also need the following helpful notations. 
\begin{definition}\label{Def: admissible family of annuli}
    For any $p\in M$ and $\mk c\geq 1$, let $\mk A:=\{A_{s_i,t_i} (G\cdot p)\}_{i=1}^{\mk c}$ be a set of $G$-annuli with $0< s_i < t_i < \inj(G\cdot p)$ for all $i=1,\dots,\mk c$. 
	We say $\mk A$ is a {\em $\mk c$-admissible} family of $G$-annuli, if $0<s_i<t_i<\frac{1}{2}s_{i+1} < \frac{1}{2}t_{i+1}$, for all $i=1,\dots,\mk c-1$. 
\end{definition}
By a contradiction argument, one easily get the following lemma. 
\begin{lemma}\label{Lem: admissible to all annuli amv}
	Let $V\in\mc V^G_n(M)$ be a $G$-varifold, $\mk c\geq 1$ be an integer. 
	If for any $\mk c$-admissible family of $G$-annuli $\mk A$, $V$ is $(\mc G,h)$-almost minimizing in some $A\in\mk A$. 
	Then $V$ is $(\mc G,h)$-almost minimizing in annuli.
\end{lemma}

Next, we show the existence of $\cGh$-almost minimizing varifolds. 
\begin{theorem}\label{Thm: existence of amv}
    Let $\Phi_0$, $\Pi$, and $h$ be given as in Theorem \ref{Thm: Equivariant min-max for PMC}. 
    Then for any minimizing sequence $\{\Phi_i\}_{i\in\mb N}\in\Pi$, there is a non-trivial $V\in\mf C(\{\Phi_i\}) \subset \mc V^G_n(M)$ so that 
    \begin{itemize}
        \item[(i)] $V$ has $c$-bounded first variation; 
        \item[(ii)] for any $\mk c$-admissible family of $G$-annuli $\mk A$, $V$ is $\cGh$-almost minimizing in some $A\in\mk A$, and thus $V$ is $\cGh$-almost minimizing in annuli by Lemma \ref{Lem: admissible to all annuli amv}, 
    \end{itemize}
    where $c=\sup_M|h|$ and $\mk c=(3^m)^{3^m}$. 
\end{theorem}
\begin{proof}
    Firstly, (i) follows from the tightening lemma (Lemma \ref{Lem: tightening}) and \eqref{Eq: width > relative}. 
    Next, we notice that \cite{zhou2020multiplicity}*{Theorem 1.16} can be adapted into our $\cGh$-setting using the equivariant modifications in \cite{wang2023G-index}*{Theorem 4.12}. 
    Finally, with Theorem \ref{Thm: discretization}, \ref{Thm: interpolation} and Proposition \ref{Prop: homotopy equivalent Almgren extension}, \ref{Prop: dis/inter homotopy} in place of \cite{zhou2020multiplicity}*{Theorem 1.11, 1.12, Proposition 1.14, 1.15} respectively, the proof in \cite{zhou2020multiplicity}*{P. 783-784} would also carry over to give (ii). 
    The main idea is that if no $V\in\mf C(\{\Phi_i\})$ satisfies (ii), then we can discretize each $\Phi_i$ to $\{\phi_i^j\}_{j\in\mb N}$ by Theorem \ref{Thm: discretization} and obtain a diagonal subsequence $\varphi_i=\phi_i^{j(i)}$ so that $\mf L^h(\Pi)= \mf L^h(\{\varphi_i\}):=\overline{\lim}_{i\to\infty}\sup_{{\rm dmn}(\varphi_i)}\Ah(\varphi_i(x))$. 
    Then by the $\mc G$-generalized \cite{zhou2020multiplicity}*{Theorem 1.16}, $\{\varphi_i\}_{i\in\mb N}$ can be deformed homotopically in $\cCcG(M)$ to a new sequence $\{\tilde{\varphi}_i\}_{i\in\mb N}$ with $\mf L^h(\{\tilde{\varphi}_i\})<\mf L^h(\{\varphi_i\})$, whose Almgren extensions $\{\tilde{\Phi}_i\}$ (Theorem \ref{Thm: interpolation}) belongs to $\Pi$ by Proposition \ref{Prop: homotopy equivalent Almgren extension}, and $\mf L(\{\tilde{\Phi}_i\})<\mf L^h(\Pi)$ as a contradiction. 
    See also \cite{sun2024multiplicity}*{Theorem 3.11, Step A, B} for a free boundary version. 
\end{proof}

\subsection{Regularity results}

To show the regularity of the min-max $G$-varifolds, we introduce the following definition of $\cGh$-replacements. 

\begin{definition}\label{Def: replacement}
	Let $U\subset M$ be an open $G$-subset, $h\in \cScG$ or $h\equiv 0$ with $c:=\sup_M|h|$, $\Omega\in\cCcG(M)$, and $V\in\mc V^G_n(M)$ be a $G$-varifold with $c$-bounded first variation in $U$ so that $\|\bd\Omega\|\leq \|V\|$. 
	For a compact $G$-set $K\subset U$, a pair $(V^*,\Omega^*)\in \mc V^G_n(M)\times \cCcG(M)$ with $\|\bd\Omega^*\|\leq \|V^*\|$ is said to be a {\em $\cGh$-replacement of $(V,\Omega)$ in $K$} if 
	\begin{enumerate}[label=(\roman*)]
		\item $V\llcorner(M\setminus K ) = V^* \llcorner (M\setminus K )$ and $\Omega \llcorner(M\setminus K ) = \Omega^*\llcorner(M\setminus K )$;
		\item $-c\Vol(K)\leq \|V\|(M) - \|V^*\|(M)\leq c\Vol(K)$;
		\item $V^*$ has $c$-bounded first variations in $U$;
		\item $V^*\llcorner\interior(K) = \sum_{i=1}^l m_i|\Sigma_i|$, where $\{m_i\}_{i=1}^l\subset\mb Z_+$ and $\{\Sigma_i\}_{i=1}^l$ are disjoint, smooth, almost embedded, stable $(G,h)$-hypersurfaces; 
		\item if $h\in\cScG$, then each $\Sigma_i$ is non-minimal with $m_i=1$, and 
			\begin{itemize}
				\item[(v.1)] $V^*=|\bd\Omega^* |$ in $\interior(K)$, and $\sum_{i=1}^l \llbracket \Sigma_i \rrbracket =\bd\Omega^*\llcorner \interior(K)$ is a $\cGh$-boundary in $\interior(K)$;
				\item[(v.2)] $V^{*}=|\bd\Omega^*|$ in $ K$ provided $\|V^*\|(\bd K)= 0$. 
			\end{itemize}
		\item[(vi)] if $h\equiv0$, then each $\Sigma_i$ is smoothly embedded, and (v.1), (v.2) also hold in any small open $G$-sets of $\interior(K),K$ respectively where $V^*$ has multiplicity one.  
	\end{enumerate}
	For simplicity, we also say $V^*$ is a $\cGh$-replacement of $V$ in $K$ (with respect to $\Omega$) and $\Omega^*$ is the associated region. 
\end{definition}

Note that $\Omega,\Omega^*$ represent the `orientation' of $V, V^*$ respectively, and Definition \ref{Def: replacement}(i)(v) indicate the matching of the orientations. 
If $h=0$, then we do not need the associated $\Omega,\Omega^*$. 

One also notices that if $V$ has $c$-bounded first variation in $M$ (not only in $U$), then the $G$-varifold $V^*$ in the above definition also has $c$-bounded first variation in $M$ by Definition \ref{Def: replacement}(i)(iii). 

\begin{definition}\label{Def: good replacement}
	We say $V\in\mc V^G_n(M)$ has {\em good $\cGh$-replacement property in $U$ w.r.t. $\Omega\in \cCcG(M)$}, if for any finite sequence of compact $G$-sets $\{K_i\subset U\}_{i=1}^l$, there exist $\{(V^{(i)}, \Omega^{(i)})\}_{i=0}^l\subset \mc V^G_n(M)\times\cCcG(M)$ so that $(V^{(0)},\Omega^{(0)})=(V,\Omega)$ and 
	\begin{itemize}
		\item $\|\bd\Omega^{(i)}\|\leq \|V^{(i)}\|$ for every $i=0,1,\dots,l$;
		\item $(V^{(i)},\Omega^{(i)})$ is a $\cGh$-replacement of $(V^{(i-1)},\Omega^{(i-1)})$ in $K_{i}$ for every $i=1,\dots, q$.
	\end{itemize}
	Moreover, $V$ is said to have {\em good $\cGh$-replacement property in annuli (w.r.t. $\Omega\in \cCcG(M))$}, if for any $p\in M$, there exists $r_{g.r.}(G\cdot p)>0$ so that $V$ has good $\cGh$-replacement property in every $A_{s,t}(G\cdot p)$, $0<s<t<r_{g.r.}(G\cdot p)$, (w.r.t. $\Omega$).
\end{definition}

\begin{remark}\label{Rem: replacement with bounded 1st variation}
	If $V$ has $c$-bounded first variation in $M$ and has good $\cGh$-replacement property in annuli (w.r.t. $\Omega$). Then for any $\cGh$-replacement $V^*$ of $V$ in $\closure(A_{s,t}(G\cdot p))$, $0<s<t<r_{g.r.}(G\cdot p)$, $V^*$ also has good $\cGh$-replacement property in annuli (w.r.t. the associated $\Omega^*$). 
\end{remark}

In the following two results, we see that the $G$-varifold $V$ in Theorem \ref{Thm: existence of amv} has good $\cGh$-replacement property in annuli if it has $c$-bounded first variation. 
\begin{proposition}\label{Prop: amv implies good replacement}
	Let $h\in \cScG\cup\{0\}$  with $c=\sup_M|h|$, and $3\leq \codim(G\cdot p)\leq 7$ for all $p\in M$. 
	If $V\in\mc V^G_n(M)$ is $\cGh$-almost minimizing in an open $G$-set $U\subset M$, then $V$ has good $\cGh$-replacement property in $U$ with respect to some $\Omega\in\cCcG(M)$. 
\end{proposition}
\begin{proof}
	This result is a generalization of \cite{zhou2020pmc}*{Proposition 6.8, 6.9} and \cite{wang2022Gcmc}*{Proposition 6.3, 6.4}, so we only points out the modifications. Let $K\subset U$ be a compact $G$-set. 
	
	Firstly, the constrained minimizing result \cite{wang2022Gcmc}*{Lemma 6.1} is also valid with $\Ah$ and $\cCcG(M)$ in place of $\mc A^c$ and $\mc C^G(M)$ respectively, which generates a constrained $\Ah$-minimizer $\Omega^*\in \cCcG(M)$ from a given $\Omega\in \mk a^{\mc G,h}(U;\epsilon,\delta;\mc F)$ so that 
	\begin{itemize}
		\item $\spt(\Omega-\Omega^*)\subset K$;
		\item $\Omega^*$ is locally $(\mc G, \Ah)$-minimizing in $\interior(K)$. 
	\end{itemize}
	Next, by Theorem \ref{Thm: regularity of G,Ah-minimizers}, we know $\Omega^*$ is indeed locally $\Ah$-minimizing in $\interior(K)$, and $\bd\Omega^*\llcorner \interior(K)$ is a smoothly {\em embedded} $\cGh$-boundary except for a $G$-set of Hausdorff dimension at most $n-7$. 
	Note that the singular set is $G$-invariant, and any $G$-set in $M$ has dimension at least $n+1-7=n-6$. 
	Hence, $\bd\Omega^*\llcorner \interior(K)$ has no singularity. 
	
	Next, let $V=\lim|\bd\Omega_i|$, where $\Omega_{i}\in \mk a^{\mc G,h}(U;\epsilon_i,\delta_i;\mc F)$ is the approximating sequence with $\epsilon_i,\delta_i\to 0$ as in Definition \ref{Def: almost minimizing varifolds}. 
	We apply the generalized \cite{wang2022Gcmc}*{Lemma 6.1} (or \cite{zhou2020pmc}*{Lemma 6.7}) in $K$ to $\Omega_{i}\in \mk a^{\mc G,h}(U;\epsilon_i,\delta_i;\mc F)$, and obtain $\Omega_i^*\in \mk a^{\mc G,h}(U;\epsilon_i,\delta_i;\mc F)$ and $V^*=\lim |\bd\Omega_i^*|$. 
	By the proof of \cite{wang2022Gcmc}*{Proposition 6.3} (or \cite{zhou2019cmc}*{Proposition 5.8}), we know $V^*$ satisfies (i)(ii) in Definition \ref{Def: replacement} and is also $\cGh$-almost minimizing in $U$, which implies Definition \ref{Def: replacement}(iii) by Lemma \ref{Lem: Gc-bounded 1st variation and c-bounded} and the arguments in \cite{zhou2019cmc}*{Lemma 5.2}. 
	
	Additionally, $\bd\Omega_i^*$ is $G$-stable in $\interior(K)$ in the sense of Definition \ref{Def: G-stable}, because otherwise, one can deform $\Omega_i^*$ by ambient $G$-equivariant isotopies in $\interior(K)$ to strictly decrease the $\Ah$ values, which contradicts the constrained $\Ah$-minimization of $\Omega_i^*$ \cite{wang2022Gcmc}*{Lemma 6.1(i)}. 
	Now, by Lemma \ref{Lem: stable and G-stable}, we know $\bd\Omega_i^*\llcorner \interior(K)$ is a smoothly embedded {\em stable} $\cGh$-boundary. 
	Thus, the regularity of $V^*=\lim |\bd\Omega_i^*|$ in $\interior(K)$ follows from the compactness theorem (Theorem \ref{Thm: compactness for stable (G,h)-hypersurfaces}). 
	
	Furthermore, up to a subsequence, $\Omega_i, \Omega^*_i$ weakly converge to some $\Omega,\Omega^*\in \cCcG(M)$ respectively, indicating that $\|\bd\Omega\|\leq \|V\|$, $\|\bd\Omega^*\|\leq \|V^*\|$, and $\spt(\Omega-\Omega^*)\subset K$. 
	Then by the compactness theorem \ref{Thm: compactness for stable (G,h)-hypersurfaces}, we know $\Omega^*$ satisfies Definition \ref{Def: replacement}(v.1) and (vi). 
	Additionally, if $h\in\cScG$ and $\|V^*\|(\bd K)=0$, then one can use the local smooth convergence $|\bd\Omega_i^*|\to V^*$ in $\interior(K)$ to see that $\|V^*\|(K)=\|\bd\Omega^* \|(K)$, which implies $V^* = \bd\Omega^*$ in $K$ by \cite{pitts2014existence}*{2.1(18)(f)}. 
	
	Hence, we see that $(V^*,\Omega^*)$ is a $\cGh$-replacement of $(V,\Omega)$. 
	Finally, in another compact $G$-set $K_2\subset U$, we can apply the generalized \cite{wang2022Gcmc}*{Lemma 6.1} to $\Omega_i^{(1)}:=\Omega_i^*$, and get $\Omega_i^{(2)}\in \mk a^{\mc G,h}(U;\epsilon_i,\delta_i;\mc F)$ with $\spt(\Omega_i^{(2)}-\Omega_i^{(1)} )\subset K_2$. 
	The above arguments then give the $\cGh$-replacement $(V^{(2)}, \Omega^{(2)}):=(\lim|\bd\Omega_i^{(2)}|, \lim\Omega_i^{(2)})$ of $(V^{(1)},\Omega^{(1)}):=(V^*,\Omega^*)$ in $K_2$. 
	The proof is finished since this construction is repeatable. 
\end{proof}

 Similar to Lemma \ref{Lem: admissible to all annuli amv}, the result below follows from a simple contradiction argument. 
\begin{lemma}\label{Lem: admissible to all annuli good replace}
	Let $\Omega\in\cCcG(M)$, $V\in\mc V^G_n(M)$ be a $G$-varifold, and $\mk c\geq 1$ be an integer. 
	If for any $\mk c$-admissible family of $G$-annuli $\mk A$, $V$ has good $\cGh$-replacement property in some $A\in\mk A$ w.r.t. $\Omega$. 
	Then $V$ has good $\cGh$-replacement property in annuli w.r.t. $\Omega$. 
\end{lemma}

Therefore, it is sufficient to show the regularity of $V\in\mc V^G_n(M)$ with $c$-bounded first variation and good $\cGh$-replacement property in annuli. 
We first classify their tangent cones. 
\begin{lemma}\label{Lem: tangent cone}
	Let $h\in \cScG\cup\{0\}$ with $c=\sup_M|h|$, and $3\leq \codim(G\cdot p)\leq 7$ for all $p\in M$. 
	Suppose $V \in \mc V_n^G(M)$ has $c$-bounded first variation in $M$ and good $\cGh$-replacement property in annuli w.r.t. some $\Omega\in\cCcG(M)$. 
	Then $V$ is integer rectifiable. Moreover, for any $C \in  \VarTan(V,p)$ with $p \in \spt\|V\|$, 
	\begin{align}\label{E:tangent cones are planes}
		C= \Theta^n (\|V\|, p) |S| \mbox{ for some $n$-plane $S \subset T_p M$, where $\Theta^n(\|V\|, p)\in\mb N$}.
	\end{align}
\end{lemma}
\begin{proof}
	Firstly, by the assumption on $V$, we can get a uniform volume ratio bound as \cite{wang2023free}*{Lemma 5.6}. 
	Indeed, the main difference is that Definition \ref{Def: replacement}(ii) does not imply $\|V\|(M) = \|V^*\|(M)$ for the $\cGh$-replacement $V^*$ of $V$ in $\closure(A_{r,2r}(G\cdot p))$. 
	Nevertheless, given $B_{4r}(G\cdot p)\subset\cup_{i=1}^N B_{20r}(p_i)$ 
	for some $\{p_i\}_{i=1}^N\subset G\cdot p$, we have $N\geq \Vol(B_{4r}(G\cdot p))/\Vol(B_{20r}(p))$, and 
	\[-\frac{c\Vol(A_{r,2r}(G\cdot p))}{N r^n}\geq -\frac{c\Vol(A_{r,2r}(G\cdot p))}{\Vol(B_{4r}(G\cdot p))} \cdot \frac{B_{20r}(p)}{r^n} \to 0 \qquad {\rm as ~}r\to 0. \]
	Thus, using $\|V\|(M)\geq \|V^*\|(M) - c\Vol(A_{r,2r}(G\cdot p))$, the proof in \cite{wang2023free}*{Lemma 5.6} still works. 
	
	This together with the proof of \cite{wang2023free}*{Lemma 5.7} imply that V as well as $C$ is rectifiable, and  $C$ is a cone of the form $T_p(G\cdot p)\times W$ for some $W\in \mc V^{G_p}_{n-\dim(G\cdot p)}( N_p(G\cdot p))$. 
	Finally, note $3\leq \dim(N_p(G\cdot p))\leq 7$, and $\lim_{r\to 0}\Vol(A_{sr, tr}(G\cdot p))/r^n = 0$ for any fixed $0<s<t$. 
	Hence, the desired result follows from the arguments in \cite{wang2023free}*{Proposition 5.10} with Definition \ref{Def: replacement}(ii) in place of the condition $\|V\|(M) = \|V^*\|(M)$. 
\end{proof}

Now, we are ready to show the following main regularity theorem. 
\begin{theorem}[Main regularity]\label{Thm: main regularity}
	Let $h\in \cScG\cup\{0\}$ with $c:=\sup_M|h|$, and $3\leq \codim(G\cdot p)\leq 7$ for all $p\in M$. 
	Suppose $V\in \mc V^G_n(M)$ is a $G$-varifold that has $c$-bounded first variation in $M$, and has good $\cGh$-replacement property in annuli (w.r.t. some $\Omega\in\cCcG(M)$). 
	Then $V$ is induced by a closed, almost embedded, $(G,h)$-hypersurface $\Sigma$ (possibly disconnected) so that 
	\begin{itemize}
		\item if $h\in\cScG$, then each component of $\Sigma$ has multiplicity one, i.e. $V=|\Sigma|$;
		\item if $h=0$, then $\Sigma$ is embedded, and $V$ has an integer density on each component of $\Sigma$. 
	\end{itemize}
\end{theorem}
\begin{proof}
	For $h\equiv 0$, the above result has been shown in \cite{wang2023G-index}*{Theorem 4.18} with $\mc G=G$, which can be easily adapted to the case of $\mc G=G_\pm$ with \eqref{Eq: free Gpm} using the above ingredients. 
	For $h\in\cScG$, the proof is similar to \cite{zhou2020pmc}*{Theorem 7.1}, so we only point out some modifications.  
	Firstly, given $p\in \spt\|V\|$, take $0<r_0<r_{g.r.}(G\cdot p)$ so that the mean curvature of $\bd B_r(G\cdot p)$ is greater than $c$ for $r\in (0,r_0)$, where $r_{g.r.}$ is given in Definition \ref{Def: good replacement}. 
	 In particular, we have 
	 \begin{align}\label{Eq: regularity thm: maximum principle}
	 	\emptyset \neq \spt\|W\| \cap \partial B_{r}(G\cdot p)=\closure \left(\spt\|W\| \backslash \closure \left(B_{r}(G\cdot p)\right)\right) \cap \partial B_{r}(G\cdot p)
	 \end{align}
	 for any $W\in\mc V_n(M)$ that has $c$-bounded first variation in $B_r(G\cdot p)$ and $W\llcorner B_r(G\cdot p)\neq 0$. 
	 
	 {\bf Step 1:} {\it Construct successive $\cGh$-replacements $V^*$ and $V^{**}$ on two overlapping concentric $G$-annuli.} 
	 Firstly, take $0<s<t<r$. By the assumption, we have a replacement $V^*$ of $V$ in $\closure(A_{s,t}(G\cdot p))$ (w.r.t. $\Omega$) so that $\Sigma_1:=\spt\|V^*\|\cap A_{s,t}(G\cdot p)$ is an almost embedded stable $(G,h)$-hypersurface. 
	 In addition, since $h\in\cScG$, we know $V^*\llcorner A_{s,t}(G\cdot p) = |\Sigma_1|$, and $\Sigma_1=\bd\Omega^*$ in $A_{s,t}(G\cdot p)$ has a unit (outer) normal $\nu_1$ for some $\Omega^*\in \cCcG(M)$ associated with $V^*$. 
	 
	 Note the touching set of $\Sigma_1$ is contained in a countable union of $(n-1)$-dimensional connected submanifolds $\cup S_1^{(k)}$ (Remark \ref{Rem: small touching and no minimal}). 
	 Hence, by Sard's theorem, we can choose $s_2\in (s,t)$ so that $\bd B_{s_2}(G\cdot p)$ is transversal to $\Sigma_1$ and each $S_1^{(k)}$. 
	 Given any $s_1\in (0,s)$, we also have a $\cGh$-replacement $V^{**}$ of $V^*$ in $\closure(A_{s_1,s_2}(G\cdot p))$ (w.r.t. $\Omega^*$) and an associated $\Omega^{**}\in \cCcG(M)$ so that $V^{**}\llcorner A_{s_1,s_2}(G\cdot p)$ is an almost embedded $(G,h)$-hypersurface $\Sigma_2:=\spt\|V^{**}\|\cap A_{s_1,s_2}(G\cdot p)$ with multiplicity $1$, and $\Sigma_2=\bd\Omega^{**}$ in $A_{s_1,s_2}(G\cdot p)$ has a unit outer normal $\nu_2$. 
	 We mention that $V^*$ and $V^{**}$ both have $c$-bounded first variation in $M$, and $\spt(\Omega^*-\Omega^{**})\subset\closure(A_{s_1,s_2}(G\cdot p))$. 
	 
	 By Definition \ref{Def: replacement}(i)(v), one can get a generalization of \cite{zhou2020pmc}*{P.346, Claim 1}: 
	 \begin{itemize}
	 	\item[(a)]  $\Sigma_1=\bd\Omega^{**}$ in $A_{s_2,t}(G\cdot p)$, and $\Sigma_2=\bd\Omega^{**}$ in $A_{s_1,s_2}(G\cdot p)$;
	 	\item[(b)] $\nu_1=\nu_{\bd\Omega^{**}}$ in $A_{s_2,t}(G\cdot p)$, and $\nu_2=\nu_{\bd\Omega^{**}}$ in $A_{s_1,s_2}(G\cdot p)$;
	 	\item[(c)] if $\|V^{**}\|(\bd B_{s_2}(G\cdot p)) = 0$, then $V^{**}=|\bd\Omega^{**}|$ in $A_{s_1,t}(G\cdot p)$, 
	 \end{itemize}
	 where $\nu_{\bd\Omega^{**}}$ is the outer unit normal of $\Omega^{**}$. 
	 
	 {\bf Step 2:} {\it Gluing the $\cGh$-replacements $V^*,V^{**}$ smoothly across $\bd B_{s_2}(G\cdot p)$.}
	 Let 
	 \[\Gamma:= \Sigma_1\cap \bd B_{s_2}(G\cdot p) \qquad {\rm and }\qquad \mc S(\Gamma):=\{q\in\Gamma: \mbox{$q$ is a touching point of $\Sigma_1$} \}. \]
	 Note that by transversality, $\Gamma$ is an almost embedded $G$-hypersurface in $\bd B_{s_2}(G\cdot p)$ with a closed touching set $\mc S(\Gamma)$, and $\mc R(\Gamma)=\Gamma\setminus\mc S(\Gamma)$ is open dense in $\Gamma$. 
	 Using \eqref{Eq: regularity thm: maximum principle}, Definition \ref{Def: replacement}, and Lemma \ref{Lem: tangent cone}, the proof of \cite{zhou2020pmc}*{P. 347 Claim2 and (6.6)} would carry over to show that $\Sigma_1$ glues continuously with $\Sigma_2$ along $\Gamma$, and 
	 \begin{align}\label{Eq: regularity thm: no mass on boundary}
	 	\|V^{**}\|(\bd B_{s_2}(G\cdot p))=\|V^{*}\|(\bd B_{s_2}(G\cdot p))=0, \quad {\rm and}\quad V^{**}=|\bd\Omega^{**}| \mbox{ in } A_{s_1,t}(G\cdot p). 
	 \end{align} 
	 In addition, given any $x\in \mc R(\Gamma)\neq \emptyset$, one notices that $\Gamma$ and $\Sigma_1$ are embedded near $G\cdot x$. 
	 Hence, one can use Lemma \ref{Lem: tangent cone} to follow the proof in \cite{schoen1981regularity}*{(7.26)-(7.32)} (or \cite{li2021min}*{Step 2, Claim 1(A)}) and show that for a sequence $x_i\to x$ with $x_i\in \mc R(\Gamma)$ and $r_i\to 0$, the associated various blow up of $V^{**}$ at $x$ is $\Theta^n(\|V^*\|, x)|T_x\Sigma|$. 
	 Then the proof of \cite{zhou2020pmc}*{P. 348-349} would carry over, and thus $\Sigma_2$ glues smoothly with $\Sigma_1\cap A_{s_2,t}(G\cdot p)$ near $G\cdot x$ along $\mc R(\Gamma)$ with matching unit normals (by (a)-(c) in {\bf Step 1}). 
	 Using the unique continuation of $h$-hypersurfaces \cite{zhou2020pmc}*{Corollary 3.17}, we conclude that $\Sigma_1=\Sigma_2$ in $A_{s,s_2}(G\cdot p)$.  
	 This finished Step 2.
	 
	 {\bf Step 3:} {\it Extend the $\cGh$-replacements to $G\cdot p$ and get $\wti V$ on the punctured tube.}
	 Rewrite $V^{**}$ and $\Sigma_2$ by $V_{s_1}^{**}$ and $\Sigma_{s_1}$ respectively to indicate the dependence on $s_1$. 
	 Then $\Sigma := \cup_{0<s_1<s}\Sigma_{s_1}$ is a smooth, almost embedded, stable $(G,h)$-hypersurface in $B_{s_2}(G\cdot p)\setminus \{G\cdot p\}$, and $\wti V=\lim_{s_1\to 0} V^{**}_{s_1}$ has $c$-bounded first variation so that $\wti V:= |\Sigma|$ in $B_{s_2}(G\cdot p)\setminus \{G\cdot p\}$, $\wti V= V^*$ in $M\setminus B_{s_2}(G\cdot p) $, $\|\wti V\|(G\cdot p) = 0$, and $G\cdot p \subset \spt\|V\|$. 
	 We refer to 	\cite{zhou2020pmc}*{P.350, Step 3} for the details. 
	 
	 {\bf Step 4:} {\it Remove the singularity of $\wti V$ at $G\cdot p$.} 
	 The removing singularity result has been proved in \cite{wang2023free}*{P. 43-44} and \cite{zhou2020pmc}*{P. 351, Step 4} for the equivariant minimal setting and the PMC setting respectively. 
	 Combining the proof given there, $\wti V$ and $\Sigma$ also extend smoothly across $G\cdot p$ as an almost embedded hypersurface in $B_{s}(G\cdot p)$. 
	 Indeed, by Remark \ref{Rem: replacement with bounded 1st variation}, we know $\wti V$ also has $c$-bounded first variations in $M$ and good $\cGh$-replacement property in annuli, which implies Lemma \ref{Lem: tangent cone} holds for $\wti V$. 
	 Additionally, we know $n-6\leq \dim(G\cdot p)\leq n-2$, and $\wti V$ is a smooth, almost embedded, stable $\cGh$-boundary $\Sigma$ in annulus $A_{a,b}(G\cdot p)$ for all $0<a<b<s$. 
	This verified the key ingredients to mimic \cite{wang2023free}*{P. 43-44} and \cite{zhou2020pmc}*{P. 351, Step 4}. 
	One can also directly apply \cite{bellettini2020stablePMC}*{Corollary 1.1} to the $\cGh$-boundary $\Sigma$ with the above ingredients. 
	
	{\bf Step 5:} {\it Show $V$ coincides with $\wti V$ near $G\cdot p$.} 
	After replacing \cite{zhou2020pmc}*{(7.1) and Proposition 6.1} by \eqref{Eq: regularity thm: maximum principle} and Lemma \ref{Lem: tangent cone} respectively, the proof in \cite{zhou2020pmc}*{P. 352 Step 5} can be taken almost verbatim with symmetric objects to show $V=\wti V$ near $G\cdot p$. 
\end{proof}

Using the above proof, we also have the following result, which indicates the $\cGh$-replacement coincides with the original varifold $V$ if $V$ is a smooth almost embedded $(G,h)$-hypersurface. 

\begin{proposition}\label{Prop: replacement V* and V}
	Let $3\leq \codim(G\cdot p)\leq 7$ for all $p\in M$, and $h\in\cScG$ or $h\equiv 0$. 
	Given $p\in M$, let $0<r<\inj(G\cdot p)/2$ so that $B_t(G\cdot p)$ has mean convex boundary for all $t\in (0, 2r)$. 
	Suppose $V\in\mc V_n^G(M)$ has good $\cGh$-replacement property in $B_{2r}(G\cdot p)$ (w.r.t. $\Omega\in\cCcG(M)$), and 
	\begin{itemize}
		\item if $h\in\cScG$, then $V\llcorner B_{2r}(G\cdot p)$ is a smooth, almost embedded $\cGh$-boundary $\Sigma=\bd\Omega\llcorner B_{2r}(G\cdot p)$ with multiplicity $1$ and $\bd\Sigma\cap B_{2r}(G\cdot p)=\emptyset$; 
		\item if $h\equiv 0$, then $V\llcorner B_{2r}(G\cdot p)$ is a smooth embedded minimal hypersurface $\Sigma$ with integer multiplicity $m$ and $\bd\Sigma\cap B_{2r}(G\cdot p)=\emptyset$. 
	\end{itemize}
	Suppose also that $s\in (r/2,r)$ so that $\spt\|V\|$ is transversal to $\bd B_s(G\cdot p)$ (even at the touching set), and $V^*$ is a $\cGh$-replacement of $V$ in $\closure(B_s(G\cdot p))$ (w.r.t. $\Omega$). 
	Then $V=V^*$, and $V$ is stable in $B_s(G\cdot p)$. 
\end{proposition}
\begin{proof}
	The case $h\equiv 0$ has been shown in \cite{wang2023G-index}*{Proposition 4.19}. 
	For $h\in\cScG$, by maximum principle, we may assume without loss of generality that $V\llcorner B_{2r}(G\cdot p)=\sum_{k=1}^K|\Sigma_k|$ for some smooth, almost embedded $(G,h)$-hypersurfaces $\Sigma_k$ so that $\bd\Sigma_k\cap B_{2r}(G\cdot p)=\emptyset$, $\Sigma_k\cap B_{s}(G\cdot p)\neq \emptyset$, and $\Sigma_k$ transversal to $\bd B_s(G\cdot p)$. 
	Then after applying the smooth gluing result in {\bf Step 2} of Theorem \ref{Thm: main regularity} to $(V,\Omega)$ and $(V^*,\Omega^*)$, we know $V\llcorner A_{\frac{r}{2},s}(G\cdot p) = V^*\llcorner A_{\frac{r}{2},s}(G\cdot p)$ and $\|V\|(\bd B_{s(G\cdot p)}) = \|V^*\|(\bd B_{s(G\cdot p)}) = 0$ (by \eqref{Eq: regularity thm: no mass on boundary}). 
	Combining with Definition \ref{Def: replacement}(i), we see $V=V^*$ in $M\setminus B_{\frac{r}{2}}(G\cdot p)$. 
	Additionally, note every component of $\spt\|V\|\cap B_s(G\cdot p)$ and $\spt\|V^*\|\cap B_s(G\cdot p)$ must intersect with $A_{\frac{r}{2},s}(G\cdot p)$ by the maximum principle. 
	It then follows from the unique continuation result (see \cite{zhou2020pmc}*{Corollary 3.17}) and the $1$-multiplicity of $V,V^*$, that $V=V^*$ in $B_s(G\cdot p)$. 
	Together, we know $V=V^*$, and $V$ is stable in $B_s(G\cdot p)$. 
\end{proof}

Combined Proposition \ref{Prop: replacement V* and V} with the Constancy Theorem, one can follow the proof of \cite{zhou2020pmc}*{Proposition 7.3} to have the following corollary. 
\begin{corollary}\label{Cor: min-max boundary}
	In Theorem \ref{Thm: main regularity}, if $h\in\cScG$, then $\Sigma=\bd\Omega$ is a $\cGh$-boundary. 
\end{corollary}

\subsection{Proof of Theorem \ref{Thm: Equivariant min-max for PMC}}\label{Subsec: proof of PMC Equivariant min-max}

\begin{proof}
	By Theorem \ref{Thm: existence of amv}, we have a non-trivial $V\in\mf C(\{\Phi_i\})\subset\mc V^G_n(M)$ that is $\cGh$-almost minimizing in annuli and either has $c$-bounded first variation or belongs to $|\bd\Phi_0|(Z)$. 
	Using Proposition \ref{Prop: amv implies good replacement}, Theorem \ref{Thm: main regularity}, and Corollary \ref{Cor: min-max boundary}, it is sufficient to show $V$ has $c$-bounded first variation. 
	Indeed, since $V$ is $\cGh$-almost minimizing in annuli, we know $V$ has $c$-bounded first variation away from a finite many orbits $\{G\cdot p_i\}_{i=1}^l$ by Lemma \ref{Lem: Gc-bounded 1st variation and c-bounded} and the proof of \cite{wang2022Gcmc}*{Lemma 5.2}. 
	Then by a cut-off trick in \cite{harvey1975extending}*{Theorem 4.1} using $\dist(G\cdot p_i, \cdot)$, we can show that $\|V\|(B_\rho(G\cdot p) ) \leq C \rho^{n - \dim(G\cdot p) - 1/2}$ for $\rho>0$ small enough. 
	Hence, the first variation extends across each $G\cdot p_i$, and $V$ has $c$-bounded first variation in $M$. 
\end{proof}

\section{Compactness of min-max $(\mc G,h)$-hypersurfaces and weak $\mc G$-index upper bounds}\label{Sec: compactness and index}

In this section, we first show a compactness theorem for the $(\mc G,h)$-hypersurfaces constructed by our equivariant min-max theory. 
More precisely, we consider the following hypersurfaces. 
\begin{definition}\label{Def: min-max Gh-hypersurface}
	Given $\mk c\in\mb Z_+$ and $h\in\cScG\cup\{0\}$, a $G$-varifold $\Sigma\in\mc V^G_n(M)$ is said to be a {\em min-max $(\mk c,\mc G, h)$-hypersurface (with multiplicity)} if 
	\begin{itemize}
		\item $\Sigma$ has $c$-bounded first variation in $M$, where $c=\sup_{M}|h|$;
		\item there exists $\Omega\in\cCcG(M)$ so that for any $\mk c$-admissible family of $G$-annuli $\mk A$, $\Sigma$ has good $\cGh$-replacement property in some $A\in\mk A$ w.r.t. $\Omega$.
	\end{itemize}
	In general, we say $\Sigma$ is a min-max $\cGh$-hypersurface (with multiplicity) if it is a min-max $(\mk c,\mc G, h)$-hypersurface for some $\mk c\in\mb Z_+$. 
\end{definition}

By the main regularity theorem \ref{Thm: main regularity} and Lemma \ref{Lem: admissible to all annuli good replace}, it is reasonable to call the $G$-varifold $\Sigma$ in the above definition a $\cGh$-hypersurface. 
In particular, if $h\in\cScG$, then $\Sigma$ is a $\cGh$-boundary with multiplicity $1$ (Corollary \ref{Cor: min-max boundary}), so we often abuse the notations and identify the varifold $\Sigma$ with its support. 
Also, by Section \ref{Subsec: proof of PMC Equivariant min-max}, we know the $G$-varifold $V$ in Theorem \ref{Thm: Equivariant min-max for PMC}(ii) is a min-max $((3^m)^{3^m}, \mc G, h)$-hypersurface. 
However, since $\Sigma$ may touch itself, it is useful to introduce the weak $G$-index generalized from \cite{marques2016morse}*{Definition 4.1}\cite{zhou2020multiplicity}*{Definition 2.1}.

\begin{definition}(Weak $G$-index)\label{Def: weak G-index}
	Given $\Sigma\in\mc P^{\mc G,h}$ with $\Sigma=\bd\Omega$, $k\in\mb N$ and $\epsilon \geq 0$, we say that $\Sigma$ is {\em $(G,k)$-unstable in an $\epsilon$-neighborhood} if there exists $c_0\in (0,1)$ and a smooth family $\{F_v\}_{v\in\overline{\mb{B}}^k_1}\subset \Diff_G(M)$ of $G$-equivariant diffeomorphisms with $F_0=id$, $F_{-v}=F_v^{-1}$, for all $v\in\overline{\mb B}^k$ such that for any $\Omega'\in \overline{\mf B}_{2\epsilon}^{\mf F}(\Omega)$, the smooth function 
	\[\mc A^h_{\Omega'}: \overline{\mb B}^k_1 \to [0,\infty),\quad\mc A^h_{\Omega'}(v):= \mc A^h(F_v(\Omega')) \]
	has a unique maximum at $m(\Omega')\in \mb B^k_{c_0/\sqrt{10}}(0)$ and $\frac{1}{c_0}id\leq D^2 \mc A^h_{\Omega'}(u)\leq -c_0 id$ for all $u\in\overline{\mb B}^k_1$. 
	In addition, we say {\em the weak Morse $G$-index of $\Sigma\in \mc P^{\mc G,h}$ is bounded (from above) by $k$}, denoted by 
	\[\Index_{G,w}(\Sigma)\leq k,\]
	if $\Sigma$ is not $(G,j)$-unstable in the $0$-neighborhood for any $j\geq k+1$. $\Sigma$ is said to be {\em weakly stable} if $\Index_{G,w}=0$.  
\end{definition}

\begin{remark}
	We have the following remarks on the above definitions. 
	\begin{itemize}
		\item For $h=0$ and stationary $\Sigma\in\mc V^G_n(M)$, we have a similar definition in \cite{wang2023G-index}*{Definition 6.1}. Also, the above concepts can be localized to an open $G$-set $U\subset M$ using $\Diff_G(U)$. 
		\item Using $\Diff(M)$ in place of $\Diff_G(M)$, the above concepts then reduce to the weak index $\Index_w$ defined in \cite{zhou2020multiplicity}*{Definition 2.1}.
		\item For smoothly embedded $\Sigma$ (without self-touching), $\Sigma$ is $(G,k)$-unstable if and only if $\Index_{G}(\Sigma)\geq k$ (see \cite{marques2016morse} or \cite{wang2023G-index}). 
		\item If $\Sigma$ is $(G,k)$-unstable in the $0$-neighborhood, then it is $(G,k)$-unstable in an $\epsilon$-neighborhood for some $\epsilon>0$. Thus, for simplicity, we also say $\Sigma$ is $(G,k)$-unstable if it is $(G,k)$-unstable in the $0$-neighborhood
	\end{itemize}
\end{remark}


\subsection{Compactness for min-max $\cGh$-hypersurfaces}

\begin{theorem}\label{Thm: compactness for min-max Gh-hypersurfaces}
	Let $(M^{n+1},g_{_M})$ be a closed Riemannian manifold with a compact Lie group $G$ acting by isometries so that $3\leq \codim(G\cdot p)\leq 7$ for all $p\in M$, and let $\mc G$ be given as in \eqref{Eq: mc G}. 
	Suppose $\{h_i\}_{i\in\mb N}\subset\cScG$ with $\lim_{i\to\infty}h_i=h_\infty\in\cScG\cup\{0\}$ in the smooth topology. 
	Given $\mk c\in\mb Z_+$ and $C>0$, let $\Sigma_i$ ($i\in\mb N$) be a min-max $(\mk c,\mc G, h_i)$-hypersurface with $\|\Sigma_i\|(M)\leq C$. 
	Then 
	\begin{itemize}
		\item[(i)] $\Sigma_i$ converges (up to a subsequence) to a min-max $(\mk c,\mc G, h_\infty)$-hypersurface $\Sigma_\infty\in\mc V^G_n(M)$ in the sense of varifold, and also in the Hausdorff distance by the monotonicity formula;
		\item[(ii)] there exists a finite union of orbits $\mc Y=\cup_{k=1}^KG\cdot p_k$ so that the convergence $\Sigma_i\to\Sigma_\infty$ is is locally smooth and graphical on $\spt(\|\Sigma_\infty\|)\setminus\mc Y$ with a constant integer multiplicity on each $G$-connected component of $\spt(\|\Sigma_\infty\|)$;
		\item[(iii)] if $h_\infty\in\cScG$, then the multiplicity of $\Sigma_\infty$ is $1$ and $\Sigma_\infty$ is a $\cGh$-boundary; moreover, $\Index_{G,w}(\Sigma_\infty)\leq I$ provided $\Index_{G,w}(\Sigma_i) \leq I$ for all $i\in\mb N$ large enough; 
		\item[(iv)] assume $\Sigma_i\neq\Sigma_\infty$ eventually and $h_i=h_\infty=h\in\cScG$, then 
		\begin{itemize}
			\item[$\bullet$] if 
			$\mc Y=\emptyset$, then $\Sigma_\infty$ has a non-trivial Jacobi $G$-field; 
			\item[$\bullet$] if 
			$\mc Y\neq\emptyset$, then $\Index_{w}(\Sigma_\infty)< \Index_{w}(\Sigma_i)$ for $i$ large enough;
		\end{itemize}
		\item[(v)] if $h=0$, then $\Index_G(\Sigma_\infty)\leq\limsup_{i\to\infty}\Index_{G,w}(\Sigma_i)$. 
	\end{itemize}
\end{theorem}

\begin{proof}
	Using the compactness result (Theorem \ref{Thm: compactness for stable (G,h)-hypersurfaces}) and the regularity results (Theorem \ref{Thm: main regularity}, Corollary \ref{Cor: min-max boundary}), the proof of \cite{wang2023G-index}*{Theorem 5.3} would carry over to show (i)(ii)(iii). 
	By Remark \ref{Rem: small touching and no minimal}, the statement in (iv) follows from \cite{sun2024multiplicity}*{Theorem 2.9, Part 2,3}. 
	In particular, since $\Sigma_i$ and $\Sigma_\infty$ are $\cGh$-boundaries, we know the function $\varphi$ in \cite{sun2024multiplicity}*{Theorem 2.9, Part 2} and the unit normal of $\Sigma_\infty$ are $\mc G$-symmetric, which implies the $G$-invariance of the Jacobi field. 
	Finally, the arguments in \cite{zhou2020multiplicity}*{Theorem 2.6(v)} and \cite{wang2023G-index}*{Proposition 6.4} also works here for (v). 
\end{proof}

\begin{remark}
	If $h_i\in\cScG(g_{_M}^i)$ and $\Sigma_i$ is a min-max $(\mk c,\mc G, h)$-hypersurface under the $G$-invariant metric $g_{_M}^i$ so that $\lim g_{_M}^i = g_{_M}$ in the smooth topology. 
	Then Theorem \ref{Thm: compactness for min-max Gh-hypersurfaces}(i)(ii)(iii) are also valid for min-max $(\mk c,\mc G, h)$-hypersurface $\Sigma_\infty$ under $g_{_M}$. 
\end{remark}

Note that in Theorem \ref{Thm: compactness for min-max Gh-hypersurfaces}(iv), we used $\Index_w$ instead of $\Index_{G,w}$, since the stability and $G$-stability are only known to be equivalent for {\em embedded} $\cGh$-hypersurfaces.
Nevertheless, combined Theorem \ref{Thm: compactness for min-max Gh-hypersurfaces}(iv) with Lemma \ref{Lem: generic good pairs}, we still have the following corollary by the proof of \cite{sun2024multiplicity}*{Lemma 2.12}. 

\begin{corollary}\label{Cor: countable min-max hypersurfaces}
	Let $(g_{_M}, h)$ be a good $\mc G$-pair. Then for $\Lambda>0$ and $I,\mk c\in\mb N$, the set 
	\begin{align}
		\mc P^{\mc G,h}(\mk c, \Lambda, I):=\{ \mbox{min-max $(\mk c,\mc G, h)$-hypersurfaces $\Sigma$ with }\Index_{w}(\Sigma)\leq I, \|\Sigma\|(M)\leq \Lambda \}
	\end{align}
	is a countable set, and thus $\cup_{I\geq 0} \mc P^{\mc G,h}(\mk c, \Lambda, I)$ is countable. 
\end{corollary}

In light of \cite{zhou2020multiplicity}*{Lemma 3.5, Theorem 2.9} and \cite{wang2023G-index}*{Corollary 1.6}, one may expect  the $C^\infty_G$-generic finiteness for min-max $(\mk c,\mc G, h)$-hypersurfaces with bounded area. 
However, for our purpose, it is sufficient to have the above countablility result.

\subsection{Weak $G$-index upper bounds for min-max $\cGh$-hypersurfaces}

By the above compactness and finiteness results, we have the weak index upper bound as \cite{marques2016morse}*{Theorem 6.1}. 

\begin{theorem}\label{Thm: index upper bound - good pairs}
	Let $(M^{n+1},g_{_M})$ be a closed Riemannian manifold with a compact Lie group $G$ acting by isometries so that $3\leq \codim(G\cdot p)\leq 7$ for all $p\in M$, and let $\mc G$ be given as in \eqref{Eq: mc G}. 
	Suppose $(h,g_{_M})$ is a good $\mc G$-pair (Definition \ref{Def: good pairs}), and $\{\Phi_i:X^m\to \cCcG(M) \}_{i\in\mb N}\in\Pi$ is a minimizing sequence with $h$-width $\mf L^h=\mf L^h(\Pi)$ satisfying \eqref{Eq: width > relative}. 
	Then there exists a min-max $(\mk c,\mc G, h)$-hypersurface $\Sigma\in \mf C(\{\Phi_i\}_{i\in\mb N}) $ with $\mk c=(3^m)^{3^m}$ so that $\Sigma=\bd\Omega\in \mc P^{\mc G,h}$, 
	\[\mf L^h(\Pi) = \mc A^h(\Omega)\qquad {\rm and}\qquad \Index_{G,w}(\Sigma)\leq m. \]
\end{theorem}

\begin{proof}
	The proof is parallel to \cite{marques2016morse}*{Theorem 6.1} and \cite{zhou2020multiplicity}*{Theorem 3.6}, so we only point out some modifications. 
	Firstly, 
	\begin{itemize}
		\item $\mc W_G:=\{\Omega\in\cCcG(M): \mc A^h(\Omega)=\mf L^h {\rm ~and~} |\bd\Omega| \mbox{ is a min-max $(\mk c,\mc G, h)$-hypersurface} \}$,
		\item $\mc W_G^{m+1}:=\{ \Omega\in\mc W_G: \Sigma=\bd\Omega\in\mc P^{\mc G,h} \mbox{ is $(G,m+1)$-unstable} \} $,
		\item $\mc W_G(r) :=\{ \Omega\in\mc W_G: \mf F(|\bd\Omega|, \mf C(\{\Phi_i\}_{i\in\mb N})) \geq r\} $. 
	\end{itemize}
	Using the proof in \cite{zhou2020multiplicity}*{Lemma 3.7}, we have $i_0\in\mb N$ and $\epsilon_0>0$ so that $\mf F(\Phi_i(X), \mc W_G(r)) > \epsilon_0$ for all $i\geq i_0$. 
	Since $(h,g_{_M})$ is a good $\mc G$-pair, it then follows from Corollary \ref{Cor: countable min-max hypersurfaces} that $\mc W_G$ and
	\[ \mc W_G^{k+1}\setminus \overline{B}^{\mf F}_{\epsilon_0}(\mc W_G(r)) = \{\Omega_1,\Omega_2,\dots\} \]
	are countable sets. Denote by $\Sigma_j = \bd\Omega_j$ a min-max $(\mk c,\mc G, h)$-hypersurface, $j\geq 1$. 
	Then we take $\epsilon_1>0$ so that $\overline{B}^{\mf F}_{\epsilon_1}(\Omega_1)\cap \overline{B}^{\mf F}_{\epsilon_0}(\mc W_G(r))=\emptyset$. 
	
	Using Definition \ref{Def: weak G-index} in place of \cite{zhou2020multiplicity}*{Definition 2.1}, one can generalize the deformation result \cite{zhou2020multiplicity}*{Theorem 3.4} to an equivariant version (see also \cite{wang2023G-index}*{Theorem 6.7} for $h=0$, and \cite{marques2016morse}*{Theorem 5.1} for $h=0,G=\{id\}$). 
	Therefore, after shrinking $\epsilon_1>0$ and applying this equivariant deformation result to $\mc K=\overline{B}^{\mf F}_{\epsilon_0}(\mc W_G(r))$ and $\Omega=\Omega_1$, we obtain $i_1\in\mb N$ and $\{\Phi_i^1\}_{i\in\mb N}$ so that
	\begin{itemize}
		\item $\Phi_i^1$ is homotopic to $\Phi_i$ in $(\cCcG(M), \mf F)$ for all $i\in\mb N$, and $\Phi_i^1|_Z=\Phi_i|_Z$ for $i\geq i_1$;
		\item $\mf L^h(\{\Phi_i^1\}_{i\in\mb N}) \leq \mf L^h(\Pi)$;
		\item $\mf F(\Phi_i^1(X),  \overline{B}^{\mf F}_{\epsilon_1}(\Omega_1)\cup \overline{B}^{\mf F}_{\epsilon_0}(\mc W_G(r))) > 0$ for all $i\geq i_1$;
		\item no $\Omega_j$ belongs to $\bd \overline{B}^{\mf F}_{\epsilon_1}(\Omega_1)$.
	\end{itemize}
	Next, if $\mf F(\Omega_2,\Omega_1) < \epsilon_1$, we may skip to the procedure for $\Omega_3$. 
	Otherwise, we can apply the equivariant deformation result again to $\mc K= \overline{B}^{\mf F}_{\epsilon_1}(\Omega_1)\cup \overline{B}^{\mf F}_{\epsilon_0}(\mc W_G(r)) $ and $\Omega=\Omega_2$, which will give us the corresponding $\epsilon_2>0$, $i_2\in\mb N$, and $\{\Phi_i^2\}_{i\in \mb N}$. 
	
	By repeating this procedure, either we stop in finitely many steps at $\{\Phi_i^k\}_{i\in\mb N}$ with $\{\epsilon_l>0\}_{l=1}^k$ and $i_k\in\mb N$, or we inductively get a sequence of $\{\Phi_i^l\}_{i\in\mb N}$ with $\epsilon_l>0$ and $i_l\in\mb N$ for $l=1,2,\dots$. 
	Then after passing to a diagonal sequence on $l$, we can obtain a sequence $\{\Psi_l\}_{l\in\mb N}$ in both cases so that $\Psi_l^1$ is homotopic to $\Phi_l$ in $(\cCcG(M), \mf F)$ with $\Phi_l^1|_Z=\Phi_l|_Z$ for all $l$, $\mf L^h(\{\Psi_l\})=\mf L^h(\Pi)$, and $\mf C(\{\Psi_l\}) \cap ( \mc W_G^{m+1}\cup\mc W_G(r) ) = \emptyset$. 
	Using the equivariant min-max theorem \ref{Thm: Equivariant min-max for PMC}, there exists $\Omega_r \in \mf C(\{\Psi_l\}) \cap ( \mc W_G \setminus ( \mc W_G^{m+1}\cup\mc W_G(r) ) ) $. 
	Finally, the desired min-max $(\mk c, \mc G, h)$-hypersurface $|\bd\Omega|=\lim_{r\to 0}|\bd\Omega_r|$ can be obtained by the compactness theorem \ref{Thm: compactness for min-max Gh-hypersurfaces}
\end{proof}

\begin{corollary}\label{Cor: index upper bound}
	In Theorem \ref{Thm: Equivariant min-max for PMC}, $\Sigma=\bd\Omega\in\mc P^{\mc G,h}$ can be further chosen as a min-max $(\mk c,\mc G, h)$-hypersurface satisfying $\Index_{G,w}(\Sigma)\leq m$, where $m=\dim(X)$. 
\end{corollary}

\begin{proof}
	Given $(g_{_M},h)$ as in Theorem \ref{Thm: Equivariant min-max for PMC}, we know $h\in \cScG^0$ (see \eqref{Eq: S_G^0}\eqref{Eq: S_Gpm^0}) and there are $G$-invariant metrics $\{g^j_{_M}\}_{j\in\mb N}$ so that $(g^j_{_M},h)$, $j=1,2,\dots$, are good $\mc G$-pairs and $g^j_{_M}\to g_{_M}$ in the smooth topology. 
	Denote by $L^h_j$ the $h$-width of $\Pi$ under $g^j_{_M}$, which satisfies \eqref{Eq: width > relative} with $g^j_{_M}$ in place of $g_{_M}$ since $\lim_{j\to\infty} L^h_j = \mf L^h(\Pi)$. 
	Finally, consider the min-max $(\mk c, \mc G, h)$-hypersurface $\Sigma_j=\bd\Omega_j$ associated with $(g^j_{_M},h)$ and $\Pi$ by Theorem \ref{Thm: index upper bound - good pairs}. 
	The compactness theorem \ref{Thm: compactness for min-max Gh-hypersurfaces} then gives the desired min-max $(\mk c, \mc G, h)$-hypersurface $|\bd\Omega_\infty|= \lim_{j\to\infty}|\bd\Omega_j|$ (up to a subsequence).  
\end{proof}

\section{Multiplicity one min-max minimal $\mc G$-hypersurfaces}\label{Sec: multiplicity one}

\subsection{Multiplicity one for equivariant min-max}

Recall that a $G$-invariant Riemannian metric $g_{_M}$ is said to be {\em $G$-bumpy} if every finite cover of a smooth embedded minimal $G$-hypersurface is non-degenerate. 
By \cite{wang2023G-index}*{Theorem 1.3}, the set of $G$-bumpy metrics is generic. 

\begin{theorem}\label{Thm: main multiplicity one}
	Let $(M^{n+1},g_{_M})$ be a closed Riemannian manifold with a compact Lie group $G$ acting by isometries so that $3\leq \codim(G\cdot p)\leq 7$ for all $p\in M$. 
    Given a $k$-dimensional cubical complex $X$ of $I(m,j)$ with a subcomplex $Z\subset X$, let $\Phi_0: X\to (\cCcG(M),\mf F)$ be an $\mf F$-continuous map so that the associated $(X,Z)$-homotopy class $\Pi$ satisfies that 
    \begin{align}\label{Eq: width > relative}
        \mf L(\Pi) > \max\left\{ \max_{x\in Z}\mf M(\bd\Phi_0(x), 0) \right\},
    \end{align}
    where we let $h\equiv 0 $ in Section \ref{Subsec: min-max setup}. 
    
    If $g_{_M}$ is a $G$-bumpy metric, then there exists a disjoint collection of smooth, $G$-connected, closed, embedded, minimal $G$-hypersurfaces $\Sigma = \cup_{i=1}^N \Sigma_i$ with a unit normal $\mc G$-vector field $\nu$ so that 
    \[ \mf L(\Pi) = \sum_{i=1}^N \Area(\Sigma_i ) \quad {\rm and}\quad \Index_G(\Sigma )= \sum_{i=1}^N\Index_G(\Sigma_i ) \leq k. \]
    In particular, each component of $\Sigma$ is $2$-sided with multiplicity one. 
\end{theorem}
\begin{proof}
	The proof is essentially the same as \cite{zhou2020multiplicity}*{Theorem 4.1} using $\mc G$-symmetric objects, so we only point out some modifications. 
	
	Firstly, take $h\in\cScG(g_{_M})$ with $\int_M h\geq 0$. 
	Then for $\epsilon>0$ small enough, we have $\epsilon h \in \cScG(g_{_M})$ and $\mf L^{\epsilon h}(\Pi) > \max \{ \max_{x\in Z}\mc A^{\epsilon h}(\Phi_0(x)), 0\}$. 
	Using Theorem \ref{Thm: Equivariant min-max for PMC} and Corollary \ref{Cor: min-max boundary}, \ref{Cor: index upper bound}, we can take a sequence $\epsilon_j\to 0$ and construct smooth, almost embedded, $(\mc G,\epsilon_jh)$-boundaries $\Sigma_j=\bd\Omega_j$ ($j\in\mb N$) so that each $\Sigma_j$ is a min-max $(\mk c,\mc G, \epsilon_jh)$-hypersurface satisfying 
	\[\mc A^{\epsilon_j h}(\Omega_{j}) = \mf L^{\epsilon_jh } := \mf L^{\epsilon_jh}(\Pi) \quad {\rm and}\quad \Index_{G,w}(\Sigma)\leq k,\]
	where $k=\dim(X)$, $\mk c=(3^k)^{3^k}$. 
	Since $\mf L^{\epsilon_j h}\to \mf L(\Pi)$ as $j\to\infty$, we can apply Theorem \ref{Thm: compactness for min-max Gh-hypersurfaces} to have a subsequence $\Sigma_j$ converging to a smooth embedded minimal $G$-hypersurface $\Sigma_\infty$ (with integer multiplicity) so that $\Sigma_\infty$ is a min-max $(\mk c,G)$-hypersurface ($h=0$ in Definition \ref{Def: min-max Gh-hypersurface}, see also \cite{wang2023G-index}*{Definition 5.1}) with 
	\[ \mf M(\Sigma_\infty) = \mf L(\Pi) \quad{\rm and}\quad \Index_G(\Sigma)\leq k. \] 
	Let $\mc Y$ be the finite union of orbits where the convergence $\Sigma_j\to\Sigma_\infty$ fails to be smooth. 
	Without loss of generality, we assume $\Sigma_\infty$ has only one $G$-connected component with multiplicity $m\in\mb Z_+$. 
	
	Next, we consider the case that $\Sigma_\infty$ has a $\mc G$-invariant unit normal $\nu$. 
	Take any exhaustion by compact $G$-set $\{ U_j \subset \Sigma_\infty\setminus\mc Y\}$ and $\delta>0$. 
	For $j$ large enough, $\Sigma_j$ has a decomposition in the thickened $G$-neighborhood $U_j\times(-\delta,\delta)$ by $m$-normal graphs $\{u^1_j,\dots,u^m_j: u^i_j\in C^\infty(U_j)\}$ over $U_j$ so that $u^1_j\leq u^2_j\leq\dots \leq u^m_j $, and $u^i_j\to 0$ in the smooth topology as $j\to\infty$
	
	If $m\geq 3$ is odd. 
	Then the construction in \cite{zhou2020multiplicity}*{Theorem 4.1, Part 4} gives a positive function $\varphi$ on $\Sigma_\infty\setminus\mc Y$ so that $L_{\Sigma_\infty}\varphi = 0$ outside $\mc Y$. 
	In addition, although a single graph of $u^i_j$ may not be $G$-invariant, we still have the $G$-invariance of $\mf h_j:=u^m_j - u^1_j$.
	Hence, $\varphi$ is $G$-invariant by \cite{zhou2020multiplicity}*{Theorem 4.1, Part 4}. 
	Using the local $G$-invariant PMC foliations given by Proposition \ref{Prop: local PMC G-foliation} and Remark \ref{Rem: local PMC signed-symmetric foliation}, the proof in \cite{zhou2020multiplicity}*{Theorem 4.1, Part 5} would carry over to show $\varphi$ is uniformly bounded and hence extends smoothly across any $G\cdot p\subset\mc Y$. 
	Therefore, $\Sigma_\infty$ is degenerated stable, which contradicts the $G$-bumpiness of $g_{_M}$.

	If $m\geq 2$ is even. 
	Then the construction in \cite{zhou2020multiplicity}*{Theorem 4.1, Part 6,7} also gives a $G$-invariant smooth function $\varphi$ on $\Sigma_\infty$ with a sign so that either $L_{\Sigma_\infty}\varphi = 0$ or $L_{\Sigma_\infty}\varphi = 2h|_{\Sigma_\infty}$. 
	The $G$-bumpiness of $g_{_M}$ indicates that $L_{\Sigma_\infty}\varphi = 2h|_{\Sigma_\infty}$. 
	However, since there are only finitely many min-max minimal $(\mk c, G)$-hypersurface $\Sigma$ with $\Area(\Sigma)\leq \mf L(\Pi)$ by \cite{wang2023G-index}*{Corollary 1.6}, the proof of \cite{zhou2020multiplicity}*{Lemma 4.2} would carry over in the equivariant setting. 
	Hence, there exists $h\in\cScG(g_{_M})$ so that the solution of $L_{\Sigma_\infty}\psi = 2h|_{\Sigma_\infty}$ on $\Sigma_\infty$ must change sign, which is a contradiction. 
	
	Now, consider the case that $\Sigma_\infty$ doesn't admit a $\mc G$-invariant unit normal. 
	In this case, the convergence must have multiplicity at least $2$; otherwise, the Allard regularity implies the convergence is smooth, and thus $\Sigma_j$ does not admit $\mc G$-invariant unit normal for $j$ large enough, which is a contradiction. 
	Then define $E: S\Sigma_\infty\times (-r,r)\to B_r(\Sigma_\infty)$ by $E(v,t):=\exp_{\Sigma_\infty}^\perp(tv)$, where $S\Sigma_\infty$ is the unit normal bundle of $\Sigma_\infty$, and $\exp_{\Sigma_\infty}^\perp$ is the normal exponential map (see \cite{wang2023G-index}*{(5.1)}). 
	Note that $G$ naturally acts on $S\Sigma_\infty\times (-r,r)$ by $g\cdot (v,t)=(dg(v),t)$, and $E^{-1}(\Sigma_\infty)= S\Sigma_\infty\times \{0\}$ is a double cover of $\Sigma_\infty$ with a $G$-invariant unit normal (pointing outside of $S\Sigma_\infty\times [0,r)$). 
	Hence, after lifting to the double covers by $E^{-1}$, the above arguments can be taken almost verbatim to show a contradiction as $m\geq 2$. 
	
	Finally, since $\Sigma_\infty$ has multiplicity $1$, the Allard regularity theorem \cite{allard1972first} implies the convergence $\Sigma_j\to\Sigma_\infty$ is smooth. 
	Hence, $\Sigma_\infty$ also has a unit normal $\nu$ that is $\mc G$-symmetric as $\nu_{\bd\Omega_j}$. 
\end{proof}

\begin{remark}
	We mention that the minimal hypersurface $\Sigma$ in Theorem \ref{Thm: main multiplicity one} is a min-max $((3^k)^{3^k}, G)$-hypersurface. 
	Moreover, without the $G$-bumpiness assumption on $g_{_M}$, the above proof implies that each $G$-component $\Sigma_i$ of $\Sigma$ is either multiplicity one or degenerate stable. 
\end{remark}

\subsection{Application to equivariant volume spectrum}

In this subsection, we apply the above multiplicity one result to study a generalized version of volume spectrum, which was first introduced by Gromov \cite{gromov1988width}\cite{gromov2003waists}, Guth \cite{guth2009minimax}, and Marques-Neves \cite{marques2017existence}. 

In the following, fix an embedded $G$-hypersurface $\Sigma_0\subset M$ with $\llbracket \Sigma_0\rrbracket\in \mc{LB}^G(M;\mb Z_2)$. 
Define
\begin{align}
	[\Sigma_0]^G := \llbracket \Sigma_0\rrbracket + \mc B^G(M;\mb Z_2) = \{\llbracket \Sigma_0\rrbracket + \bd\Omega :  \Omega\in \mc C^G(M) \}
\end{align}
as the space of mod 2 $n$-cycles that are different from $\llbracket \Sigma_0\rrbracket$ by a $G$-boundary. 
Note $ [\Sigma_0]^G \subset \mc{LB}^G(M;\mb Z_2)\subset \mc Z^G_n(M;\mb Z_2)$. 
In particular, if $G=\{id\}$, then $[\Sigma_0]^G=[\Sigma_0]$ represents the $\mb Z_2$-homology class of $\Sigma_0$. 

Given $\Phi: X\to [\Sigma_0]^G$, let $\Phi': X\to \mc B^G(M;\mb Z_2)$ be given by $\Phi'(x)=\Phi(x) - \llbracket\Sigma_0\rrbracket$. 
Then $\Phi$ is continuous in the $\mc F$ (or $\mf M$) topology if and only if $\Phi'$ is continuous in the $\mc F$ (or $\mf M$) topology. 
Hence, by Almgren's isomorphism \cite{almgren1962homotopy} and its equivariant version \cite{wang2023G-index}*{Theorem 3.1}, 
\[  \pi_1([\Sigma_0]^G; \llbracket\Sigma_0\rrbracket )  \cong H_{n+1}(M^{n+1};\mb Z_2)\cong \mb Z_2, \]
and $ \pi_m([\Sigma_0]^G; \llbracket\Sigma_0\rrbracket ) =0$ for $m\geq 2$. 
In particular, we have 
\[ H^1([\Sigma_0]^G;\mb Z_2) = \mb Z_2 = \{0,\bar\lambda\}.\]
Additionally, since $[\Sigma_0]^G$ is weakly homotopically equivalent to $\mb{RP}^\infty$, we can generalize the volume spectrum of Marques-Neves \cite{marques2017existence}*{Definition 4.3} to the following equivariant version. 

\begin{definition}\label{Def: G-width}
	Given $k\in\mb N$, a $\mc F$-continuous map $\Phi: X\to [\Sigma_0]^G$ is called a {\em $(G,k)$-sweepout (in $[\Sigma_0]^G$)} if 
	\[\Phi^*(\bar\lambda^k)\neq 0 \in H^k(X;\mb Z_2), \]
	where $\bar\lambda^k$ is the cup product of $\bar\lambda$ with itself for $k$-times. 
	If $\Phi$ has no concentration of mass on orbits, then we say $\Phi$ is {\em admissible}. 
	Denote by ${\mc P}^{[\Sigma_0]^G}_k$ the set of all admissible $(G,k)$-sweepouts in $[\Sigma_0]^G$. 
	Then the {\em $(G,k)$-width} of $(M,g_{_M})$ in the class $[\Sigma_0]^G$ is 
	\[\omega_k([\Sigma_0]^G, g_{_M}) := \inf_{\Phi\in {\mc P}^{[\Sigma_0]^G}_k} \sup \{ \mf M(\Phi(x)) : x\in{\rm dmn}(\Phi) \} ,\]
	where ${\rm dmn}(\Phi)$ is the domain of $\Phi$. 
\end{definition}

\begin{remark}
	If we take $\Sigma_0=\bd\Omega$ for some $\Omega\in \mc C^G(M)$, then $\omega_k([\Sigma_0]^G,g_{_M})=\omega^G_k(M, g_{_M})$ is the $k$-th $G$-equivariant volume spectrum of $(M,g_{_M})$ defined in \cite{wang2025density}. 
	Assume further that $G=\{id\}$, then $\omega_k([\Sigma_0]^G,g_{_M})=\omega_k(M, g_{_M})$ is the $k$-th volume spectrum defined in \cite{marques2017existence}\cite{zhou2020multiplicity}. 
\end{remark}

Note the discretization/interpolation theorems in Section \ref{Subsec: discret/interpolate} has been proved for maps into $\mc Z^G_n(M;\mb Z_2)$ in \cite{wang2022min}, and thus are valid for $[\Sigma_0]^G$. 
Hence, we have the following results. 
\begin{theorem}\label{Lem: generalized volume spectrum}
	Let ${\mc P}^{[\Sigma_0]^G,\mf F}_k$ be the set of $\Phi\in {\mc P}^{[\Sigma_0]^G}_k$ that are continuous in the $\mf F$-topology with domain $X={\rm dmn}(\Phi)$ of dimension $k$ (and $X$ is identical to its $k$-skeleton). Then
	\[ \omega_k([\Sigma_0]^G, g_{_M}) = \inf_{\Phi\in {\mc P}^{[\Sigma_0]^G,\mf F}_k} \sup\{\mf M(\Phi(x)) : x\in{\rm dmn}(\Phi)\}.\]
	Additionally, there exists a disjoint collection of smooth, $G$-connected, closed, embedded minimal $G$-hypersurfaces $\{\Sigma^k_i: i=1,\dots, l_k\}$ with integer multiplicities $\{m^k_i: i=1\dots,l_k\}$ so that $\Sigma^k=\sum_{i=1}^{l_k}m^k_i|\Sigma^k_i|$ is a min-max $((3^k)^{3^k}, G)$-hypersurface ($h=0$ in Definition \ref{Def: min-max Gh-hypersurface}) with
	\[ \omega_k([\Sigma_0]^G, g_{_M})=  \sum_{i=1}^{l_k} m^k_i \Area(\Sigma^k_i) \quad{\rm and}\quad \Index_G(\spt(\Sigma^k))=\sum_{i=1}^{l_k}\Index_G(\Sigma^k_i) \leq k . \]
	Moreover, $\omega_k([\Sigma_0]^G, g_{_M}) \sim k^{\frac{1}{l+1}}$ as $k\to\infty$, where $l+1 = \min_{p\in M}\codim(G\cdot p)$. 
\end{theorem}
\begin{proof}
	Since $[\Sigma_0]^G\subset \mc Z^G_n(M;\mb Z_2)$, the first part of the lemma follows from \cite{wang2022min}*{Corollary 1} and the proof of \cite{wang2023G-index}*{Corollary 1.7}. 
	Next, note that the {\em boundary type} assumption (i.e. $\Phi(x)\in \mc B^G(M;\mb Z_2)$) in \cite{wang2022min}\cite{wang2023free}\cite{wang2023G-index} was only used in the regularity theory of equivariant min-max constructions (see for instance \cite{wang2022min}*{\S 6}), which can be easily adapted to our case of $\Phi(x)\in \mc {LB}^G(M;\mb Z_2)$. 
	Hence, the results in \cite{wang2023G-index}, in particular \cite{wang2023G-index}*{Corollary 1.7}, are also valid for $\omega_k([\Sigma_0]^G, g_{_M})$, which gives the desired min-max $((3^k)^{3^k}, G)$-hypersurface $\Sigma^k$. 
	Finally, one notices that $\Phi\in {\mc P}^{[\Sigma_0]^G}_k$ if and only if $\Phi' := \Phi - \llbracket\Sigma_0\rrbracket $ is an admissible $(G,k)$-sweepout in $\mc B^G(M;\mb Z_2)$. 
	In addition, since $|\mf M(\llbracket\Sigma_0\rrbracket + \bd\Omega)  -  \mf M(\bd\Omega) | \leq \mf M(\llbracket\Sigma_0\rrbracket)$ for all $\Omega\in\mc C^G(M)$, we have $|\omega_k([\Sigma_0]^G, g_{_M}) - \omega^G_k(M, g_{_M})|\leq \mf M(\llbracket\Sigma_0\rrbracket)$, where $\omega^G_k(M, g_{_M})$ is the $(G,k)$-width in $\mc B^G(M;\mb Z_2)$. 
	Hence, $\omega_k([\Sigma_0]^G,g_{_M}) \sim k^{\frac{1}{l+1}}$ as $k\to\infty$ by \cite{wang2025density}. 
\end{proof}

Now, we can show the following multiplicity one result generalizing \cite{zhou2020multiplicity}*{Theorem A}.

\begin{theorem}\label{Thm: main multiplicity one for spectrum}
	If $g_{_M}$ is a $G$-bumpy metric and $3\leq \codim(G\cdot p)\leq 7$ for all $p\in M$. 
	Then for each $k\in\mb N$, there exists a disjoint collection of smooth, $G$-connected, embedded, minimal hypersurfaces  $\{\Sigma^k_i: i=1,\dots, l_k\}$ so that $\Sigma^k=\cup_{i=1}^{l_k}\Sigma^k_i$ is a min-max $((3^k)^{3^k},G)$-hypersurface (in the sense of Definition \ref{Def: min-max Gh-hypersurface} with $h=0$), $\llbracket\Sigma^k\rrbracket = \llbracket \Sigma_0 \rrbracket + \bd \Omega_k $ for some $\Omega_k\in\mc C^G(M)$, 
	\[ \omega_k([\Sigma_0]^G, g_{_M}) = \sum_{i=1}^{l_k} \Area(\Sigma^k_i) \quad{\rm and}\quad \Index_G(\Sigma^k)=\sum_{i=1}^{l_k}\Index_G(\Sigma^k_i) \leq k .  \]
\end{theorem}

\begin{proof}
	If $\llbracket \Sigma_0 \rrbracket\notin\mc  B(M;\mb Z_2)$, then we can use Proposition \ref{Prop: lift to 2-cover}(ii) to lift the constructions into the $2$-cover of $M$ with the $2$-cover of $G$ acting by isometries. 
	Hence, we only need to consider the following two cases:
	\begin{itemize}
		\item[(1)] $\llbracket \Sigma_0 \rrbracket\in\mc B^G(M;\mb Z_2)$, and $\mc G=G$;
		\item[(2)] $\llbracket \Sigma_0 \rrbracket\in\mc B^{G_\pm}(M;\mb Z_2)$, and $\mc G=G_\pm$ satisfies \eqref{Eq: free Gpm}. 
	\end{itemize} 
	As we mentioned in the proof of Lemma \ref{Lem: generalized volume spectrum}, all the results in \cite{wang2022min}\cite{wang2023free}\cite{wang2023G-index} for Case (1) can also be adapted for Case (2). 
	Therefore, the proof of \cite{zhou2020multiplicity}*{Theorem 5.2} can be taken almost verbatim to show the above theorem, so we only point out the necessary modifications. 
	
	Firstly, we replace `embedded minimal hypersurfaces' in the proof of \cite{zhou2020multiplicity}*{Theorem 5.2} by `min-max $((3^k)^{3^k}, G)$-hypersurfaces' (in the sense of Definition \ref{Def: min-max Gh-hypersurface} with $h=0$, see also \cite{wang2023G-index}*{Definition 1.4}). 
	Then the $G$-bumpiness of $g_{_M}$ implies that there are only finitely many min-max $((3^k)^{3^k}, G)$-hypersurfaces with $\Area \leq \Lambda$ for any given $\Lambda >0$ (\cite{wang2023G-index}*{Corollary 1.6}). 
	Combined with Lemma \ref{Lem: generalized volume spectrum} and \cite{wang2023G-index}*{Theorem 6.8}, one immediately obtains the equivariant version of \cite{zhou2020multiplicity}*{Lemma 5.3, 5.4}. 
	Next, noting \cite{wang2025density}*{Proposition 3.3} is also valid in $[\Sigma_0]^G$, we directly have an equivariant version of \cite{zhou2020multiplicity}*{Lemma 5.6} as a corollary. 
	Additionally, we replace `$\mb Z_2$-almost minimizing in small annuli' by `good $\mc G$-replacement property in some $A\in\mk A$ for any $\mk c$-admissible family of $G$-annuli $\mk A$' (in the sense of Definition \ref{Def: good replacement} with $h=0$, see also \cite{wang2023G-index}*{Definition 1.4}). 
	Then \cite{zhou2020multiplicity}*{Lemma 5.7} can be generalized to our settings. 
	Moreover, we can use the equivariant discretization/interpolation theorems in \cite{wang2022min}*{\S 4} in place of those in \cite{marques2017existence}*{\S 3}, and use Theorem \ref{Thm: main multiplicity one} in place of \cite{zhou2020multiplicity}*{Theorem 4.1}. 
	We also use $\omega_k([\Sigma_0]^G, g_{_M}), {\mc P}^{[\Sigma_0]^G,\mf F}_k, \cCcG(M)$ and $[\Sigma_0]^G$ in place of $\omega_k(M, g_{_M}), \mc P_k, \mc C(M)$ and $\mc Z_n(M;\mb Z_2)$ respectively. 
	Then, the rest part of the proof of \cite{zhou2020multiplicity}*{Theorem 5.2} would carry over using our equivariant objects. 
\end{proof}


\appendix
	\addcontentsline{toc}{section}{Appendices}
	\renewcommand{\thesection}{\Alph{section}}

\section{Local $G$-invariant PMC foliations}\label{Sec: foliations}

In this appendix, we show a generalization of White's local foliation result \cite{white1987curvature}*{Appendix}, where we add a subscript $G$, $G_p$ to $C^{2,\alpha}$, $C^{2,\alpha}_0$ indicating that the functions are $G$ or $G_p$ invariant. 

\begin{proposition}\label{Prop: local PMC G-foliation}
	Let $(M,g_{_M})$ be a closed Riemannian manifold with a compact Lie group $G$ acting by isometrics so that $3\leq \codim(G\cdot x)\leq 7$ for all $x\in M$. 
	Given $p\in M$ and $R>0$, suppose $\Sigma\subset B_{R}(G\cdot p)$ is an embedded $G$-invariant minimal hypersurface with a $G$-invariant unit normal $\nu$ satisfying $G\cdot p\subset \Sigma$ and $\bd\Sigma\cap B_{R}(G\cdot p)=\emptyset$. 
	Then for any $h\in C^\infty_G(B_{R}(G\cdot p))$, there exist $\epsilon>0$ and a $G$-neighborhood $U\subset M$ of $G\cdot p$ so that if 
	\begin{align}\label{Eq: appendix assumption}
		w\in C_G^{2,\alpha}(\Sigma\cap U) \mbox{ satisfies } \|w\|_{C^{2,\alpha}} < \epsilon,
	\end{align}
	then for any $t\in (-\epsilon,\epsilon)$, there exists a $C^{2,\alpha}_G$-function $v_t: U\cap\Sigma\to\mb R$ satisfying
	\begin{itemize}
		\item[(i)] $\Graph(v_t):=\exp^\perp_\Sigma(v_t\nu)$ is a $G$-invariant $C^{2,\alpha}$ hypersurface with mean curvature $h|_{\Graph(v_t)}$, where the mean curvature is evaluated w.r.t. the upward pointing unit normal of $\Graph(v_t)$;
		\item[(ii)] $v_t(x)=w(x) + t$ for $x\in \bd (U\cap\Sigma)$;
		\item[(iii)] $v_t$ depends on $t,h,w$ in a $C^1$ way, and $\{\Graph(v_t)\}_{t\in [-\epsilon,\epsilon]}$ forms a foliation.
	\end{itemize}
\end{proposition}

\begin{proof}
	Let $N^\Sigma(G\cdot p)$ be the normal bundle of $G\cdot p$ in $\Sigma$, $B^{N^\Sigma(G\cdot p)}_{s}(0):=\{v\in N^\Sigma(G\cdot p): |v|< r\}$ be the $r$-neighborhood of $0$ in the normal bundle ($r>0$), and $\mb B_s=B^{N_p^\Sigma(G\cdot p)}_{s}(0):=B^{N^\Sigma(G\cdot p)}_{s}(0) \cap N_p^\Sigma(G\cdot p)$. 
	Denote by $C_r(G\cdot p):= B^{N^\Sigma(G\cdot p)}_{r}(0) \times (-r,r)$. 
	Then for $r_0\in (0,R)$ small enough, we have a $G$-equivariant diffeomorphism in a neighborhood of $G\cdot p$: 
	\[ E: C_1(G\cdot p) \to M, \quad E(v, t):=\exp_q(r_0t\cdot \nu(q)) \mbox{ with }q=\exp^{\Sigma,\perp}_{G\cdot p}(r_0v), \]
	where 
	$\exp^{\Sigma,\perp}_{G\cdot p}$ is the normal exponential map at $G\cdot p$ in $\Sigma$. 
	Then we use $E$ to pull back the constructions to $C_1(G\cdot p)$. 
	Define $\mu_r: C_1(G\cdot p)\to C_r(G\cdot p)$, $\mu_r(v,t):=(rv, rt)$ as the $r$-dilation in $N^\Sigma(G\cdot p)$, and 
	\[ g^r:=\frac{1}{r^2}\mu_r^* E^* g_{_M}.  \]
	Note as $r\to 0$, $g^r$ tends to the locally trivial metric $g^0$ on the bundle $N^\Sigma(G\cdot p)\times \mb R$, i.e. $(N^\Sigma(G\cdot p)\times \mb R, g^0)$ is locally isometric to the product of $G\cdot p$ with the Euclidean fibers. 
	
	Denote by $S_{p,r}:=\{ (v, t) \in C_r(G\cdot p) : v\in N_p^\Sigma(G\cdot p)  \}\cong N_p(G\cdot p)$ the slice in $C_r(G\cdot p)$ at $p$, which is $G_p$-invariant.  
	For $r_0$ even smaller, it follows from \cite{wang2023G-index}*{Lemma A.2}, that there is a smooth $G_p$-invariant even elliptic integrand $\varphi_r: TS_{p,1}\setminus (S_{p,1}\times \{0\}) \to \mb R$ for all $r\in [0,1]$ so that 
	\[ \mc H^n_r(\Gamma) = \mc H_r^{\dim(G\cdot p)}(G\cdot p) \cdot  \int_{\Gamma\cap S_{p,1}} \varphi_r(q,v(q)) d\mc H_r^{n-\dim(G\cdot p)}(q) \]
	for any $G$-invariant hypersurface $\Gamma\subset C_1(G\cdot p)$, where $\mc H^k_r$ is the $k$-Hausdorff measure under the metric $g^r$, and $v(q)$ is a unit normal of $\Gamma\cap S_{p,1}$ at $q$. 
	In particular, if $\Gamma=\Graph(\tilde{f})$ for some $G$-function $\tilde{f}$ on $B^{N^\Sigma(G\cdot p)}_{1}(0)$, then by the area formula, $\varphi_r$ induces a $G_p$-invariant functional $A_r$ so that 
	\[\int_{\Gamma\cap S_{p,1}} \varphi_r(q,v(q)) d\mc H_r^{n-\dim(G\cdot p)}(q)  = \int_{\mb B_1} A_r(x,f(x),\nabla^rf(x) ) d\mc H_0^{n-\dim(G\cdot p)}(x), \]
	where $f=\tilde{f}\llcorner B^{N_p^\Sigma(G\cdot p)}_{1}(0)$. 
	Moreover, for any $\Omega\in \mc C^G(C_1(G\cdot p))$ and $G$-invariant smooth function $\tilde{h}$ in $C_1(G\cdot p)$, the co-area formula implies 
	\[ \int_\Omega \tilde{h} d\mc H^{n+1}_r = \mc H_r^{\dim(G\cdot p)}(G\cdot p) \cdot \int_{\Omega\cap S_{p,1}} \frac{\tilde{h}}{J^*_P} d\mc H^{n+1-\dim(G\cdot p)}_r, \]
	where $P:C_1(G\cdot p)\to G\cdot p$ is the bundle projection and $J^*_P$ is its Jacobian. 
	
	Suppose now that $\tilde f$ and $\tilde \eta$ are $C_G^{2,\alpha}$-functions on $B^{N^\Sigma(G\cdot p)}_{1}(0)$ so that $\tilde\eta\llcorner(\bd B^{N^\Sigma(G\cdot p)}_{1}(0))=0$. 
	Denote by $f=\tilde{f}\llcorner B^{N_p^\Sigma(G\cdot p)}_{1}(0)$ and $\eta=\tilde{\eta}\llcorner B^{N_p^\Sigma(G\cdot p)}_{1}(0)$ as $C_{G_p}^{2,\alpha}$-functions on $\mb B_1\cong B^{N_p^\Sigma(G\cdot p)}_{1}(0)$, and by $F_f:\mb B_1\to \Graph(f)$ the graph map $F_f(v):=(v,f(v))$. 
	For $s$ near $0$, define $\Omega_s=\{(v, t)\in C_1(G\cdot p): -1<t< \tilde{f}(v)+s\tilde{\eta}(t)\}$ with $\bd\Omega_s = \Graph(\tilde{f}+s\tilde{\eta})$ in $C_1(G\cdot p)$. 
	Then we see
	\begin{align*}
		\left. \frac{d}{ds}\right|_{s=0} \frac{\mc A^{\tilde{h}}_{r}(\Omega_s)}{\mc H_r^{\dim(G\cdot p)}(G\cdot p)} = \int_{\mb B_1} \left(H_r(f)(x) - \bar{h}_r(f)(x)\right) \cdot \eta(x) d\mc H_0^{n-\dim(G\cdot p)}(x), 
	\end{align*}
	where $\mc A^{\tilde{h}}_r$ is the $\mc A^{\tilde{h}}$-functional under the metric $g^r$,  
	$H_r(f)\in C_{G_p}^{0,\alpha}(\mb B_1)$ is given by \cite{white1991space}*{Theorem 1.1} associated to $A_r$, 
	and $\bar{h}_r(f) := J_{F_f} \cdot  \left(\frac{\tilde{h}}{J^*_P} \cdot \langle \frac{\bd}{\bd t},\nu_{\Graph(f)} \rangle_{g^r} \right)\circ F_f\in C_{G_p}^{2,\alpha}(\mb B_1)$. Here, $J_{F_f}$ is the Jacobian of $F_f$ in the area formula, and $\nu_{\Graph(f)}$ is the upper unit normal of $\Graph(f)$. 
	Therefore, $\Graph(\tilde{f})$ has mean curvature $\tilde{h}$ under the metric $g^r$ if and only if $H_r(f)=\bar{h}_r(f)$. 
	
	The rest of the proof is very similar to \cite{white1987curvature}*{Appendix} (see also \cite{wang2023G-index}*{Proposition A.3}). 
	Indeed, we can define 
	\[ \Psi: [0,1]\times \mb R\times C_{G_p}^{2,\alpha}(\mb B_1)\times C_{G_p}^{2,\alpha}(\mb B_1)\times C_{0,G_p}^{2,\alpha}(\mb B_1) \to C_{G_p}^{0,\alpha}(\mb B_1)\]
	by $\Psi(r,t,h,w,u) = H_r(t+w+u) - \bar{h}_r(t+w+u)$. 
	Then $\Psi$ is $C^1$ and $D_5\Psi(0,t,0,0,0)$ is an isomorphism between $C_{0,G_p}^{2,\alpha}(\mb B_1)$ and $ C_{G_p}^{0,\alpha}(\mb B_1)$. 
	Hence, given $h\in C^\infty_G(B_R(G\cdot p))$, the implicit function theorem gives an $\epsilon>0$ and a $G$-neighborhood $U=E(C_r(G\cdot p))$ of $G\cdot p$ with $r<\epsilon$ so that if $w$ is given as in \eqref{Eq: appendix assumption}, then there exists $\{v_t\}_{t\in (-\epsilon,\epsilon)}\subset C^{2,\alpha}_{G_p}(\Sigma\cap U\cap E(S_{p,1}))$ with $v_t=t+w$ on $\bd(\Sigma\cap U \cap E(S_{p,1}))$ satisfying that $\{\Graph(v_t)\}_{t\in(-\epsilon,\epsilon)}$ forms a $G_p$-invariant local foliation in the slice $E(S_{p,1})$, and $H_r(v_t\circ E\circ\mu_r)=\bar{h}_r(v_t\circ E\circ\mu_r)  $, where $\bar{h}_r(\cdot)$ is induced by $\tilde{h}:= r\cdot (h\circ E\circ \mu_r)\in C^\infty_{G}(C_1(G\cdot p))$. 
	Finally, define the $G$-equivariant extension of $v_t$ in $\Sigma\cap U$ by $v_t(g\cdot q):=v_t(q)$ for all $g\in G$ and $ q\in \Sigma\cap U\cap E(S_{p,1})$. 
	One easily checks that the extension is well defined and satisfies (i)(ii)(iii). 
\end{proof}

\begin{remark}\label{Rem: local PMC signed-symmetric foliation}
	In the above proposition, if $G=G_+\sqcup G_-$ satisfies \eqref{Eq: free Gpm}, $R<\inj(G\cdot p)$, and $\Sigma$ has a unit normal $G_\pm$-vector field. 
	Then the similar result are also valid for $G_\pm$-signed symmetric $h,w$ and $v_t$ by applying Proposition \ref{Prop: local PMC G-foliation} in $B_R(G_+\cdot p)$ and $B_R(G_-\cdot p)$ respectively. 
\end{remark}

\bibliographystyle{abbrv}

\bibliography{reference.bib}   
\end{document}